\documentclass[english]{article}
\usepackage[T1]{fontenc}
\usepackage[numbers]{natbib}
\usepackage{filecontents}
\usepackage[latin9]{inputenc}
\usepackage[letterpaper]{geometry}
\usepackage{pifont}
\usepackage{float}
\usepackage{amsmath}
\usepackage{graphicx}
\usepackage{color}
\usepackage{epstopdf}
\usepackage{adjustbox} 
\usepackage{diagbox}
\usepackage{multirow}
\usepackage{array}
\usepackage{mathrsfs}
\usepackage{amsmath, amssymb, amsthm}
\usepackage{subfigure}
\usepackage{tikz}
\usepackage{pgfplots}
\usepackage[hidelinks]{hyperref}
\usepackage{comment}

\graphicspath{{./Figures/}}

\newcommand{\ms}{\text{ms}}

\newtheorem{theorem}{Theorem}[section]
\newtheorem{example}{Example}[section]

\newtheorem{lemma}{Lemma}[section]
\newtheorem{remark}{Remark}

\usepackage{amsfonts} 
\newcommand\norm[1]{\Vert#1\Vert}

\newcommand\abs[1]{\lvert#1\rvert}

\newcommand{\tand}{\quad\text{and}\quad}

\newcommand\dt{\,\text{d}t}
\newcommand\dx{\,\text{d}x}

\newcommand{\Nv}{N_\mathrm{v}}
\newcommand{\Ne}{N_\mathrm{e}}

\usepackage{color}
\usepackage{adjustbox}
\usepackage{diagbox}
\definecolor{black}{rgb}{0,0,0}

\definecolor{red}{rgb}{1,0,0}

\definecolor{blue}{rgb}{0,0,1}

\usepackage{multirow}

\makeatother

\usepackage{babel}
\usepackage{authblk}

\usepackage{xcolor}

\usepackage{blindtext} 
\usepackage{fancyhdr}
\usepackage{algorithm}
\usepackage{algpseudocode}
\usepackage{soul}
\usepackage{titling}

\usepackage{footnote}
\begin{document}
\title{\large{\bfseries{\scshape{Partially Explicit Generalized Multiscale Method for Poroelasticity Problem}}}}

\author{Xin Su \thanks{Department of Mathematics, University of California, Santa Barbara, CA 93106, United States} , Wing Tat Leung \thanks{Department of Mathematics, City University of Hong Kong, Hong Kong, China.} , Wenyuan Li \thanks{Department of Mathematics, Texas A\&M University, College Station, TX 77843, United States}  ~ and ~ Sai-Mang Pun\thanks{Qube Research \& Technologies, Hong Kong SAR}}

\maketitle

\begin{abstract}
We develop a partially explicit time discretization based on the framework of constraint energy minimizing generalized multiscale finite element method (CEM-GMsFEM) for the problem of linear poroelasticity with high contrast. Firstly, dominant basis
functions generated by the CEM-GMsFEM approach are used to capture important degrees of
freedom and it is known to give contrast-independent convergence that scales with the mesh size.
In typical situation, one has very few degrees of freedom in dominant basis functions. This part is
treated implicitly. Secondly, we design and introduce an additional space in the complement space
and these degrees are treated explicitly. We also investigate the CFL-type stability restriction for
this problem, and the restriction for the time step is contrast independent.
\end{abstract}
\section{Introduction}
 Poroelasticity is a field of study that deals with the behavior of fluids and solids that interact in a porous medium. It is a multidisciplinary field that combines principles of fluid mechanics, solid mechanics, and mathematics to understand the mechanics of fluid-saturated porous materials such as soil, rock, and biological tissues. Poroelasticity is an important area of research in reservoir geomechanics \cite{zoback2010reservoir, geomechanics}, hydrogeology \cite{wang2000theory,rice1976some,roeloffs1996poroelastic}, biomedical engineering \cite{tully2011cerebral,guo2020multiple,wirth2006axisymmetric}, and many other fields. Biot developed the theory of the elasticity and viscoelasticity of fluid-saturated porous solids \cite{biot1941general, biot1956theory}. The study of poroelasticity involves understanding how the movement and flow of fluids affect the mechanical behavior of the surrounding solid material, and how the deformation of the solid material can in turn affect the movement and flow of fluids.  

There are different numerical treatments on the poroealsticity problem. Ern and Muenier 
applied the implicit Euler method to the poroelasticity model \cite{ern2009posteriori}, which results in an
unconditionally stable first-order method. A weighted $\theta-$scheme was considered in \cite{kolesov2014splitting} and a stability result was obtained for the weight $\theta\geq 0.5$. 
However, both the purely implicit scheme and the $\theta$-scheme require to solve a large coupled linear system. Due to the considerations of computational costs, different approaches were proposed to solve poroelasticity problem efficiently. In \cite{kolesov2014splitting}, an additive splitting scheme was investigated, which employs an additive representation of the primary operator. The scheme is designed to include an operator that is easily invertible, which can help increase the computational efficiency. Altmann and Maier 
\cite{altmann2022decoupling} designed and analyzed a semi-explicit time discretization scheme of first order for poroelasticity with nonlinear permeability provided that the coupling between the elasticity model and the flow equation is relatively weak. They showed that the semi-explicit scheme significantly outperforms the implicit ones, in general. There are some iterative coupling techniques employed in coupling flow and mechanics in porous media, for example, undrained split, the fixed stress split, the drained split, and the fixed strain split iterative methods \cite{mikelic2013convergence, kim2010sequential,altmann2023novel, dana2018convergence}. Iterative coupling refers to a sequential method for solving coupling problems in which either the flow or mechanics is solved first and then the other problem is solved using the latest information from the previous solution. This procedure is repeated at each time step until a solution is obtained that meets an acceptable tolerance level. However, these methods are limited by their assumptions about the model. For instance, the undrained split method assumes constant fluid mass during the structure deformation. 

For poroelastic problem, if the medium is homogeneous, typical finite element techniques can be used to simulate the poroelastic behavior, see for instance \cite{chaabane2018splitting, ern2009posteriori,lewis1998finite,murad1994stability}. 
However, if the material is strongly heterogeneous, displacement and pressure might oscillate on a fine scale. It is prohibitively expensive to use fine scale method due to the level of details incorporated in fine-scale models.  Coarse models are often used to get around the expensive computational cost under such circumstances. These methods include the generalized multiscale finite element method (GMsFEM) \cite{gmsfem_poro1,gmsfem_poro2, efendiev2013generalized}, the variational multiscale methods \cite{hughes1998variational}, the localized orthogonal decomposition technique \cite{lod_poro,maalqvist2017generalized,maalqvist2014localization}, and the constraint energy minimizing GMsFEM (CEM-GMsFEM) \cite{chung2017constraint, fu2019computational}. All these techniques attempt to construct multiscale spaces to obtain accurate coarse-scale solutions, which could reflect spatial fine-scale features.

Constraint energy minimizing GMsFEM (CEM-GMsFEM) is applied to solve the poroelasticity problem and an online adaptive enrichment method \cite{su2022fast} based on CEM-GMsFEM is proposed to improve the solution. However, as more and more basis functions are added to correct the solution through the online adaptive enrichment, the simulation becomes slower since the inverse of the stiffness matrix takes excessive time. In this paper, we adapt the partially explicit splitting idea to improve the performance of CEM-GMsFEM. We split the solution space for pressure into two parts: the coarse-grid part and the correction part. The basis functions for pressure by CEM-GMsFEM serve as the coarse-grid part, which are used to capture the main features of the pressure. A special function space is carefully designed to serve as the correction part. This space is constructed in the complement space (complement to the coarse space) to capture missing pressure information. For displacement in both the coarse-grid part and the correction part, we used the basis functions by CEM-GMsFEM. In other words, the major characteristic of the solution is captured by the CEM-GMsFEM basis and simulated implicitly. The correction part of the solution is simulated explicitly in the complement space. 

This paper is organized as follows. The preliminaries about the poroelasticity problem and the related notations are introduced in Section \ref{sec:prelim}. In Section \ref{sec:scheme}, we describe the partially explicit method and its stability. Section \ref{sec:construction} is devoted to the construction of coarse spaces for pressure and displacement and the construction of additional function space for pressure. We present the numerical results in Section \ref{sec:numerical}. Conclusions are presented in Section \ref{sec:conclusion}.
\section{Preliminaries}\label{sec:prelim}
\subsection{Model problem}
Let $\Omega \subset \mathbb{R}^d$ ($d \in \{ 2, 3 \}$) be a bounded and polyhedral Lipschitz domain and $T > 0$ be a fixed time. We consider the problem of linear poroelasticity: 
Find the pressure 
$p\colon [0,T]\times \Omega\to \mathbb{R}$ and the displacement field $u\colon [0,T]\times \Omega\to \mathbb{R}^d$ such that 
\begin{subequations}
\label{eq:model}
\begin{alignat}{3}
\label{eq:model1}
-\nabla\cdot \sigma(u) +  \nabla (\alpha p)  &= 0\phantom{f}\qquad\text{in } (0,T] \times \Omega, \\
\label{eq:model2}
\partial_t \bigg( \alpha \nabla\cdot  u + \frac{1}{M}  p \bigg)  - \nabla \cdot \bigg( \frac{\kappa}{\nu} \nabla p\bigg) &= f\phantom{0}\qquad\text{in } (0,T] \times \Omega, 
\end{alignat}
\end{subequations}	
with boundary and initial conditions
\begin{subequations}
\begin{alignat}{3}
\label{eq:init1}
u&=0\phantom{p^0}\qquad \text{on }(0,T]\times \partial\Omega,\\
\label{eq:init2}
p&=0\phantom{p^0}\qquad \text{on }(0,T]\times \partial\Omega,\\
p(0, \cdot)&=p^0\phantom{0}\qquad\text{in }\Omega.\label{eq:init3}
\end{alignat}
\end{subequations}
In the interest of parsimony, the present investigation restricts its focus to a homogeneous Dirichlet boundary condition exclusively. However, the possibility of extending this analysis to incorporate other types of boundary conditions is acknowledged, and interested readers are referred to prior studies \cite{henning2014localized, wang2022multiscale, ye2023constraint}.
In this model, the primary sources of heterogeneity are the stress tensor $\sigma$, the permeability $\kappa$, and the Biot-Willis fluid-solid coupling coefficient $\alpha\in[0,1]$. We denote by $M$ the Biot modulus and by $\nu$ the fluid viscosity. Both are assumed to be constant. Moreover, $f$ is a source term representing injection or production processes. Body forces, such as gravity, are neglected. In the case of a linear elastic stress-strain constitutive relation, the stress and strain tensors are expressed as
 \begin{equation*}
 {\sigma(u)} := 2\mu  {\varepsilon(u)} + \lambda (\nabla\cdot  {u}) \,  {\mathcal{I}} \quad \text{and} \quad
 {\varepsilon}( {u}) := \frac{1}{2} \Big( \nabla  {u} + (\nabla  {u})^T \Big),
 \end{equation*}
 where $\mathcal{I}$ is the identity tensor and $\lambda,\, \mu>0$ are the Lam\'e coefficients, 
 which can also be expressed in terms of Young's modulus $E>0$
 and the Poisson ratio $\nu_p\in(-1, 1/2)$,
 \begin{equation*}
 \lambda=\frac{\nu_p}{(1-2\nu_p)(1+\nu_p)}E, \quad
 \mu=\frac{1}{2(1+\nu_p)}E.
 \end{equation*}
In the considered case of heterogeneous media, the coefficients $\mu$, $\lambda$, $\kappa$, and $\alpha$ may be highly oscillatory.
 
 \subsection{Function spaces}
In this subsection, we clarify the notation used throughout the article. 
We write $(\cdot,\cdot)$ to denote the inner product in $L^2(\Omega)$ and $\norm{\cdot}$ for the corresponding norm. 
Let $H^1(\Omega)$ be the classical Sobolev space with norm $\norm{v}_1 := \big( \norm{v}^2 + \norm{\nabla v}^2 \big)^{1/2}$ for all $v \in H^1(\Omega)$ and $H_0^1(\Omega)$ the subspace of functions having a vanishing trace. We denote the corresponding dual space of $H_0^1(\Omega)$ by $H^{-1}(\Omega)$. 
Moreover, we write $L^r(0,T; X)$ for the Bochner space with the norm 
$$ \norm{v}_{L^r(0,T;X)} := \bigg( \int_0^T \norm{v}_X^r \dt \bigg)^{1/r}, \quad 1\leq r < \infty, \quad \norm{v}_{L^\infty(0,T;X)} := \sup_{0 \leq t \leq T} \norm{v}_X,$$
where $(X,\norm{\cdot}_X)$ is a Banach space. Also, we define $H^1(0,T;X) := \{ v \in L^2(0,T;X) : \partial_t v \in L^2(0,T;X) \}$.
To shorten the notation, we define the spaces for the displacement $u$ and the pressure $p$ by
  \begin{equation*}
  V:=[H_0^1(\Omega)]^d,\quad Q:=H^1_0(\Omega).
  \end{equation*}
  
 \subsection{Variational formulation and discretization}
In this subsection, we provide the variational formulation corresponding to the system \eqref{eq:model}. We first multiply the equations \eqref{eq:model1} and \eqref{eq:model2} with test functions from $V$ and $Q$, respectively. Then, applying Green's formula and making use of the boundary conditions \eqref{eq:init1} and \eqref{eq:init2}, we obtain the following variational problem: find $u(\cdot,t)\in V$ and $p(\cdot,t)\in Q$ such that
\begin{subequations}\label{eq:weak}
	\begin{alignat}{3}
   a(u,v) - d(v,p) &= 0, \label{eqn:v1} \\
   d(\partial_t u,q) + c(\partial_t p,q) + b(p,q) &= (f,q) \label{eqn:v2}      
\end{alignat} 
for all $v\in V$, $q\in Q$ and  
\begin{alignat}{3}
p(\cdot,0)=p^0 \in Q. \label{eqn:v3}
\end{alignat}
\end{subequations}
The bilinear forms are defined by 
   \begin{eqnarray*}
   && a(u,v) := \int_{\Omega} \sigma(u) : \varepsilon(v)\dx, \qquad\quad b(p,q) := \int_{\Omega} \frac{\kappa}{\nu}\, \nabla p\cdot \nabla q \dx, \\
   && c(p,q) := \int_{\Omega} \frac{1}{M}\, p\, q\dx, \qquad\qquad\quad\,\, d(u,q) := \int_{\Omega} \alpha\,  (\nabla \cdot u)q\dx.
   \end{eqnarray*}   
 Here, we use : to denote the Frobenius inner product of matrices.
   
Note that \eqref{eqn:v1} and \eqref{eqn:v3} can be used to define a consistent initial value $u^0:= u(\cdot,0) \in V$. 
Using Korn's inequality \cite{BrennerScott, ciarlet1988mathematical}, we get

$$ c_\sigma \norm{v}_1^2 \leq \norm{v}^2_a=: a(v,v) \leq C_\sigma \norm{v}_1^2$$
for all $v \in V$, where $c_\sigma$ depends on $\operatorname{essinf}_{x\in \Omega}\mu(x)$ while  $C_\sigma$ depends on $\operatorname{esssup}_{x\in \Omega} \mu(x)$ and $\operatorname{esssup}_{x\in \Omega} \lambda(x)$. Similarly, there exist two positive constants $c_\kappa$ and $C_\kappa$ such that
$$ c_\kappa \norm{q}_1^2 \leq \norm{q}^2_b =: b(q,q) \leq C_\kappa \norm{q}_1^2$$ 
for all $q \in Q$. Here, $c_\kappa$ depends on $\operatorname{essinf}_{x\in \Omega}\kappa(x)$ and $C_\kappa$ depends on $\operatorname{esssup}_{x\in \Omega} \kappa(x)$. 
The proof of existence and uniqueness of solutions $u$ and $p$ to \eqref{eq:weak} can be found in \cite{showalter2000diffusion}. 

Here are some notations that would be used later in this article.  Let $V^\prime$ (resp., $Q^\prime$) be the dual space of $V$ (resp., $Q$) with the duality product denoted by $\langle \cdot, \cdot\rangle_a$ (resp., $\langle \cdot, \cdot\rangle_b$) and norm $\|\cdot\|_{a^\prime}=\sup_{0\not=v\in V}|\langle \cdot, v\rangle_a|/\|v\|_a$ (resp., $\|\cdot\|_{b^\prime}=\sup_{0\not=q\in Q}|\langle \cdot, q\rangle_b|/\|q\|_b$). 
$d(\cdot, \cdot) $ is a continuous bilinear form with continuity constant $\beta$, i.e., for all $(v,q)\in V\times Q$, $|d(v,q)|\leq \beta \|v\|_a\|q\|_c$, where $\beta =\max_{x\in\Omega}\{\alpha(M/\lambda)^{\frac{1}{2}}\}$ and $\|\cdot\|_c$ is induced by $c(\cdot,\cdot)$. For all $q\in Q$, it holds that $\|q\|_c\leq \gamma \|q\|_b$, where $\gamma=\frac{C_p}{M}$ and $ C_p$ is the  Poincar\'e constant.

We use a temporal discretized approach as a reference solution. For time discretization, let $\tau$ be a uniform time step and define $t_n=n\tau$ for $n =  0,1,\cdots,N$ and $T=N\tau$. The semi-discretization in time by the backward Euler method yields the following semi-discrete problem: find $u^{n+1}\in V$ and $p^{n+1}\in Q$ such that
\begin{subequations}\label{eq:semi}
	\begin{alignat}{3}
	a(u^{n+1},v) - d(v,p^{n+1}) &= 0, \label{eq:semi_a}\\
	d(D_{\tau}u^{n+1},q) + c(D_{\tau}p^{n+1},q) + b(p^{n+1},q) &= ( f^{n+1},q ),      \label{eq:semi_b}
	\end{alignat} 
\end{subequations}
for all $v\in V$, $q\in Q$, and  
\begin{alignat}{3}
p(\cdot,0)=p^0 \in Q. \label{eq:semiinitial}
\end{alignat}
Here,
$D_{\tau}$ denotes the discrete time derivative, i.e., $D_{\tau}u^{n+1}:=(u^{n+1}-u^n)/\tau$ and $f^{n+1}:=f(t_{n+1})$.

To fully discretize the variational problem \eqref{eq:weak}, we introduce a conforming partition $\mathcal{T}^h$ for the computational domain 
$\Omega$ with (local) grid sizes $h_{K}:=\text{diam}(K)$ for $K\in \mathcal{T}^h$ and ${h:=\max_{K\in \mathcal{T}^h}h_K}$. We remark that $\mathcal{T}^h$ is referred to as the \textit{fine grid}. Next, let $V_h$ and $Q_h$ be the standard finite element spaces of first order with respect to the fine grid $\mathcal{T}^h$, i.e.,
$$ V_h := \{ v = (v_i)_{i=1}^d \in V: \text{each} ~ v_i \lvert_K \text{ is a polynomial of degree} \leq 1 \text{ for all } K \in \mathcal{T}^h \}, $$
$$ Q_h := \{ q \in Q: q\lvert_K \text{ is a polynomial of degree} \leq 1 \text{ for all } K \in \mathcal{T}^h \}. $$

The fully discretization of \eqref{eq:weak} read as follows: for $n = 1,\cdots,N$ and given $p_h^0,\ u_h^0$, find $u_h^{n+1}\in V_h$ and $p_h^{n+1}\in Q_h$ such that
\begin{subequations}\label{eq:weak1}
	\begin{alignat}{3}
	a(u_h^{n+1},v) - d(v,p_h^{n+1}) &= 0, \label{eq:weak1_a}\\
	d(D_{\tau}u_h^{n+1},q) + c(D_{\tau}p_h^{n+1},q) + b(p_h^{n+1},q) &= (f^{n+1},q),      \label{eq:weak1_b}
	\end{alignat} 
\end{subequations}
for all $v\in V_h$ and $q\in Q_h$. 
Here, the initial value $p_h^0\in Q_h$ is set to be the $L^2$
projection of $p^0\in Q$. The initial value $u_h^0$ for the displacement can be obtained by solving 
\begin{equation} \label{eq:initial_u}
a(u_h^0,v)=d(v,p_h^0)
\end{equation}
for all $v \in V_h$.

\section{Partially Explicit Temporal Splitting Scheme}\label{sec:scheme}
In order to make the multiscale spaces~$V_{H}$ and~$Q_{H}$ suitable for computations, we need finite-dimensional analogons. To achieve this, we follow the construction from the previous chapter and use the finite element space corresponding to the coarse grid $\mathcal{T}^H$. This then yields the following fully discrete implicit scheme: for $n=1,2,\cdots,N$, find $(u_{H}^{n+1},p_{H}^{n+1}) \in V_{H} \times Q_{H}$ such that
\begin{subequations}
	\begin{alignat}{2}
	a(u_{H}^{n+1},v) - d(v,p_{H}^{n+1}) &= 0, \label{eq:weak2_1}\\
	d(D_{\tau}u_{H}^{n+1},q) + c(D_{\tau}p_{H}^{n+1},q) + b(p_{H}^{n+1},q) &= (f^{n+1},q),   \label{eq:weak2_2}    
	\end{alignat} 
\end{subequations}
for all $(v,q) \in V_{H} \times Q_{H}$ with initial condition $p_{H}^0\in Q_{H}$ defined by
$$b(p^0 - p_{H}^0,q) = 0 $$
for all $q\in Q_{H}$. 

With the aim of improving the accuracy of the numerical solution by the CEM-GMsFEM method and inspired by the partially explicit method for wave problems \cite{chung2022contrast}, we develop a partially explicit method for the poroealsticity problem. We use the $Q_{H,1}$ generated by the CEM-GMsFEM method to capture the major feature of pressure and carefully design the complement subspace $Q_{H,2}$ to capture the missing details of pressure.
We consider that $Q_H$ can be decomposed into two subspaces $Q_{H,1}$ and $Q_{H,2},$ namely,
$$
Q_H=Q_{H,1}+Q_{H,2}.
$$
With decomposition of $Q_H$, we consider the following partially explicit scheme:
\begin{align}
&a(u_{H,1}^{n+1},v)-d(v,p_{H,1}^{n+1})  =0, \label{scheme1}\\
&a(u_{H,2}^{n+1},v)-d(v,p_{H,2}^{n+1})  =0,  \label{scheme2}\\
&d(D_{\tau} u_{H,1}^{n+1} + D_{\tau} u_{H,2}^{n},q_{1})+c(D_{\tau} p_{H,1}^{n+1}+ D_{\tau} p_{H,2}^{n},q_{1})+b(p_{H,1}^{n+1}+p_{H,2}^{n},q_{1})  =(f^{n},q_{1}), \label{scheme3}\\
&d(D_{\tau} u_{H,2}^{n+1}+ D_{\tau} u_{H,1}^{n},q_{2})+c( D_{\tau} p_{H,2}^{n+1}+ D_{\tau} p_{H,1}^{n},q_{2})+b(p_{H,1}^{n+1}+p_{H,2}^{n},q_{2})  =(f^{n},q_{2}). \label{scheme4}
\end{align}
for all $q_1\in Q_{H,1}$, $q_2\in Q_{H,2}$, and $v\in V_H$. Then the multiscale solution are $u_H^n=u_{H,1}^n+u_{H,2}^n$ and $p_H^n=p_{H,1}^n+p_{H,2}^n$.

\begin{remark}
We remark that with the proposed partially explicit scheme, we need to solve $(u_{H,1},p_{H,1})$ and $(u_{H,2},p_{H,2})$ separately,
and thus, compared to CEM-GMsFEM, we need to solve one more elasticity equation in each time
step. 
\end{remark}

We found that the scheme is stable if the  $b$-norm and the $c$-norm satisfies a CFL-type condition for all functions in $Q_{H,2}$.  The stability of the scheme then implicates the convergence of the scheme. It is worthy to note that this CFL-type condition characterize the explicit space $Q_{H,2}$ and we construct such a subspace to satisfy this CFL-type condition in the next section.

\begin{theorem}
The partially explicit scheme \eqref{scheme1}-\eqref{scheme2} is stable if  \begin{equation}\label{stability_condition_2}
\|q_{2}\|_{b}^{2}\leq\cfrac{(1-\gamma_c)}{\tau}\|q_{2}\|_{c}^{2}\end{equation} for all $q_{2}\in Q_{H,2}$. Moreover, we have the following stability estimate
\begin{align*}
&~~\sum_{i=1,2}\Big(\gamma_c\|p_{H,i}^{n+1}-p_{H,i}^{n}\|_{c}^{2}+\gamma_{a}\|u_{H,i}^{n+1}-u_{H,i}^{n}\|_{a}^{2}\Big)+\tau\|p_H^{n+1}\|_{b}^{2}\Big)\\
&\leq \sum_{i=1,2}\Big(\gamma_c\|p_{H,i}^{n}-p_{H,i}^{n-1}\|_{c}^{2}+\gamma_{a}\|u_{H,i}^{n}-u_{H,i}^{n-1}\|_{a}^{2}\Big)+\tau\|p_H^{n}\|_{b}^{2}\Big)+\cfrac{2\tau^2}{(1-\gamma_c)}\norm{M^{1/2}f^n}^2.
\end{align*}
\end{theorem}

\begin{proof}
We have 
\begin{align*}
d(\cfrac{u_{H,i}^{n+1}-u_{H,i}^{n}}{\tau},p_{H,i}^{n+1}-p_{H,i}^{n}) & =\cfrac{1}{\tau}\|u_{H,i}^{n+1}-u_{H,i}^{n}\|_{a}^{2}\\
d(\cfrac{u_{H,i}^{n}-u_{H,i}^{n-1}}{\tau},p_{H,j}^{n+1}-p_{H,j}^{n}) & =\cfrac{1}{\tau}a(u_{H,i}^{n}-u_{H,i}^{n-1},u_{H,j}^{n+1}-u_{H,j}^{n})
\end{align*}

We define $0<\gamma_{a}<1$ and $0<\gamma_c<1$ to be the constants used in strengthened Cauchy-Schwarz inequalities \cite{aldaz2013strengthened} with 
\begin{align*}
c(p_{1},p_{2}) & \leq\gamma_c\|p_{1}\|_{c}\|p_{2}\|_{c}\\
a(u_{1},u_{2}) & \leq\gamma_{a}\|u_{1}\|_{a}\|u_{2}\|_{a}.
\end{align*}
Thus, we have 
\begin{align*}
d(\cfrac{u_{H,i}^{n+1}-u_{H,i}^{n}+u_{H,j}^{n}-u_{H,j}^{n-1}}{\tau},p_{H,i}^{n+1}-p_{H,i}^{n})\geq\cfrac{1}{\tau}\Big(\|u_{H,i}^{n+1}-u_{H,i}^{n}\|_{a}^{2}-\gamma_{a}\|u_{H,i}^{n}-u_{H,i}^{n-1}\|_{a}\|u_{H,j}^{n+1}-u_{H,j}^{n}\|_{a}\Big)\\
\geq \cfrac{1}{\tau}\Big(\|u_{H,i}^{n+1}-u_{H,i}^{n}\|_{a}^{2}-\cfrac{\gamma_a}{2}\|u_{H,i}^{n}-u_{H,i}^{n-1}\|_{a}^2-\cfrac{\gamma_a}{2}\|u_{H,j}^{n+1}-u_{H,j}^{n}\|_{a}^2\Big). 
\end{align*}
On the other side,
\begin{align*}
&~~~~~b(p_{H,1}^{n+1}+p_{H,2}^n, p_H^{n+1}-p_H^n) \\
&=b(p_{H,1}^{n+1}+p_{H,2}^{n+1}-p_{H,2}^{n+1}+p_{H,2}^n, p_H^{n+1}-p_H^n)\\
&=b(p_H^{n+1},p_H^{n+1}-p_H^n)-b(p_{H,2}^{n+1}-p_{H,2}^n, p_H^{n+1}-p_H^n)\\
&= \frac{1}{2}\norm{p_H^{n+1}}_b^2-\frac{1}{2}\norm{p_H^{n}}_b^2+\frac{1}{2}\norm{p_H^{n+1}-p_H^n}_b^2-b(p_{H,2}^{n+1}-p_{H,2}^n, p_H^{n+1}-p_H^n),
\end{align*}
and hence,
$$b(p_{H,2}^{n+1}-p_{H,2}^n, p_H^{n+1}-p_H^n)\leq \frac{1}{2}\norm{p_{H,2}^{n+1}-p_{H,2}^n}_b^2+\frac{1}{2}\norm{p_H^{n+1}-p_H^n}_b^2.$$
Also, we have
\begin{align*}
c(\cfrac{p_{H,i}^{n+1}-p_{H,i}^{n}+p_{H,j}^{n}-p_{H,j}^{n-1}}{\tau},p_{H,i}^{n+1}-p_{H,i}^{n}) & =\cfrac{1}{\tau}\|p_{H,i}^{n+1}-p_{H,i}^{n}\|_{c}^{2}+\cfrac{1}{\tau}c(p_{H,i}^{n+1}-p_{H,i}^{n},p_{H,j}^{n}-p_{H,j}^{n-1})\\
 & \geq\cfrac{1}{\tau}\|p_{H,i}^{n+1}-p_{H,i}^{n}\|_{c}^{2}-\cfrac{\gamma_c}{\tau}\|p_{H,i}^{n+1}-p_{H,i}^{n}\|_{c}\|p_{H,j}^{n}-p_{H,j}^{n-1}\|_{c}\\
 &\geq \cfrac{1}{\tau}\|p_{H,i}^{n+1}-p_{H,i}^{n}\|_{c}^{2}-\cfrac{\gamma_c}{2\tau}\|p_{H,i}^{n+1}-p_{H,i}^{n}\|_{c}^2-\cfrac{\gamma_c}{2\tau}\|p_{H,j}^{n}-p_{H,j}^{n-1}\|_{c}^2\\
 &\quad\quad\text{for }i,j=1,2, \;i\neq j.
\end{align*}

Taking $q_1= p_{H,1}^{n+1}-p_{H,1}^n, q_2=p_{H,2}^{n+1}-p_{H,2}^n, v= u_H^{n+1}-u_H^n$, we have 
\begin{align*}
 & \cfrac{1}{2\tau}\Big(\tau\|p_H^{n+1}\|_{b}^{2}\Big)-\cfrac{1}{2}\|p_{H,2}^{n+1}-p_{H,2}^{n}\|_{b}^{2}\\
 & +\cfrac{1}{2\tau}\Big((2-\gamma_c)\sum_{i=1,2}\|p_{H,i}^{n+1}-p_{H,i}^{n}\|_{c}^{2}+(2-\gamma_{a})\sum_{i=1,2}\|u_{H,i}^{n+1}-u_{H,i}^{n}\|_{a}^{2}\Big)\\
\leq & (f^{n},p_H^{n+1}-p_H^{n})+\cfrac{1}{2\tau}\Big(\gamma_c   \sum_{i=1,2}\|p_{H,i}^{n}-p_{H,i}^{n-1}\|_{c}^{2}+\gamma_{a}\sum_{i=1,2}\|u_{H,i}^{n}-u_{H,i}^{n-1}\|_{a}^{2}+\tau\|p_H^{n}\|_{b}^{2}\Big).
\end{align*}
After rearrangement, we get
\begin{align*}
 & \cfrac{1}{2\tau}\sum_{i=1,2}\Big((1-\gamma_c)\|p_{H,i}^{n+1}-p_{H,i}^{n}\|_{c}^{2}+2(1-\gamma_{a})\|u_{H,i}^{n+1}-u_{H,i}^{n}\|_{a}^{2}\Big)\\
 & +\cfrac{1}{2\tau}\Big(\sum_{i=1,2}\Big(\gamma_c\|p_{H,i}^{n+1}-p_{H,i}^{n}\|_{c}^{2}+\gamma_{a}\|u_{H,i}^{n+1}-u_{H,i}^{n}\|_{a}^{2}\Big)+\tau\|p_H^{n+1}\|_{b}^{2}\Big)\\
 &+\frac{1}{2}\Big( \cfrac{1-\gamma_c}{\tau}   \sum_{i=1,2}\|p_{H,i}^{n+1}-p_{H,i}^{n}\|_{c}^{2} - \|p_{H,2}^{n+1}-p_{H,2}^{n}\|_{b}^{2}\Big) \\
\leq & (f^{n},p_H^{n+1}-p_H^n)+\frac{1}{2\tau}\Big(\gamma_c   \sum_{i=1,2}\|p_{H,i}^{n}-p_{H,i}^{n-1}\|_{c}^{2}+\gamma_{a}\sum_{i=1,2}\|u_{H,i}^{n}-u_{H,i}^{n-1}\|_{a}^{2}+\tau\|p_H^n\|_{b}^{2}\Big).
\end{align*}

If $\|q_{2}\|_{b}^{2}\leq\cfrac{(1-\gamma_c)}{\tau}\|q_{2}\|_{c}^{2}$
for all $q_{2}\in V_{2,H}$, we have 
\begin{align*}
 & \cfrac{1}{2\tau}\sum_{i=1,2}\Big((1-\gamma_c)\|p_{H,i}^{n+1}-p_{H,i}^{n}\|_{c}^{2}+2(1-\gamma_{a})\|u_{H,i}^{n+1}-u_{H,i}^{n}\|_{a}^{2}\Big)\\
 & +\cfrac{1}{2\tau}\Big(\sum_{i=1,2}\Big(\gamma_c\|p_{H,i}^{n+1}-p_{H,i}^{n}\|_{c}^{2}+\gamma_{a}\|u_{H,i}^{n+1}-u_{H,i}^{n}\|_{a}^{2}\Big)+\tau\|p_H^{n+1}\|_{b}^{2}\Big)\\
\leq & (f^{n},p_H^{n+1}-p_H^n)+\cfrac{1}{2\tau}\Big(\sum_{i=1,2}\Big(\gamma_c\|p_{H,i}^{n}-p_{H,i}^{n-1}\|_{c}^{2}+\gamma_{a}\|u_{H,i}^{n}-u_{H,i}^{n-1}\|_{a}^{2}\Big)+\tau\|p_H^n\|_{b}^{2}\Big)
\\\leq & \cfrac{\tau}{(1-\gamma_c)}\norm{M^{1/2}f^n}^2+\frac{1-\gamma_c}{4\tau}\sum_{i=1,2}2\|p_{H,i}^{n+1}-p_{H,i}^{n}\|_{c}^{2}+\\
&~~\cfrac{1}{2\tau}\Big(\sum_{i=1,2}\Big(\gamma_c\|p_{H,i}^{n}-p_{H,i}^{n-1}\|_{c}^{2}+\gamma_{a}\|u_{H,i}^{n}-u_{H,i}^{n-1}\|_{a}^{2}\Big)+\tau\|p_H^n\|_{b}^{2}\Big).
\end{align*}
This gives us the stability estimate for the partially explicit scheme:
\begin{align*}
&~~\sum_{i=1,2}\Big(\gamma_c\|p_{H,i}^{n+1}-p_{H,i}^{n}\|_{c}^{2}+\gamma_{a}\|u_{H,i}^{n+1}-u_{H,i}^{n}\|_{a}^{2}\Big)+\tau\|p_H^{n+1}\|_{b}^{2}\Big)\\
&\leq \sum_{i=1,2}\Big(\gamma_c\|p_{H,i}^{n}-p_{H,i}^{n-1}\|_{c}^{2}+\gamma_{a}\|u_{H,i}^{n}-u_{H,i}^{n-1}\|_{a}^{2}\Big)+\tau\|p_H^n\|_{b}^{2}\Big)+\cfrac{2\tau^2}{(1-\gamma_c)}\norm{M^{1/2}f^n}^2.
\end{align*}
\end{proof}

\section{Space Construction}\label{sec:construction}

Finite element spaces constructed by CEM-GMsFEM are good choices for $Q_{H,1}$ and $V_H$ since the previous study shows that these finite element spaces have the ability to capture main characteristics of the solution.   
In this section, we first present the multiscale finite element space $Q_{H,1}$ and $V_H$ that is constructed by CEM-GMsFEM. 
Next, we present the construction of the space $Q_{H,2}$ which satisfies the explicit stability condition and helps in reducing errors. The choice of $Q_{H,2}$ is inspired by \cite{chung2022contrast}. With the selected $Q_{H,2}$,
the scheme achieves a convergence rate that is independent of the high contrast permeability value, denoted as $\kappa$. The proof of convergence is omitted in this work; however, readers who are interested in the motivation of the choices of $Q_{H,2}$ and the proof of convergence can find detailed information in \cite{chung2022contrast}.  
Recalling that $Q=H^1_0(\Omega).$ For any set $S$, we let $Q(S)=H^1(S)$ and $Q_0(S)=H^1_0(S).$
\subsection{The CEM-GMsFEM method}
The CEM-GMsFEM starts with the auxiliary basis functions by solving spectral problems on each coarse element $K_i$ over the spaces $V(K_i):= {V \vert}_{ K_i}$ and $Q(K_i):= {Q \vert}_{ K_i}$. More precisely, we consider local eigenvalue problems (of Neumann type): find $(\lambda_j^i,v_j^i)\in \mathbb{R}\times V(K_i)$ such that
\begin{equation}\label{eq:eig1}
a_i(v_j^i,v)=\lambda_j^i s^1_i(v_j^i,v)
\end{equation}
for all $v \in V(K_i)$ and find 
$(\zeta_j^i,q_j^i)\in \mathbb{R}\times Q(K_i)$ such that 
\begin{equation}\label{eq:eig2}
b_i(q_j^i,q)=\zeta_j^i s^2_i(q_j^i,q)
\end{equation}
for all $q \in Q(K_i)$, where 
$$a_i(u,v) := \int_{K_i} \sigma(u) : \varepsilon(v)\dx, \quad b_i(p,q) := \int_{K_i} \frac{\kappa}{\nu} \nabla p\cdot \nabla q \dx,$$ 
$$s^1_i(u,v):=\int_{K_i} \tilde{\sigma}u\cdot v\dx, \quad s^2_i(p,q):=\int_{K_i} \tilde{\kappa}pq\dx,$$ 
\begin{equation*}
\tilde\sigma := \sum_{i=1}^{\Nv}(\lambda+2\mu) | \nabla \chi_i^1 |^2,\qquad
\tilde\kappa := \sum_{i=1}^{\Nv}\frac{\kappa}{\nu} | \nabla \chi_i^2 |^2.
\end{equation*}
The functions ${\chi_i^1}$ and ${\chi_i^2}$ are neighborhood-wise defined
partition of unity functions \cite{bm97} on the coarse grid. To be more precise, for $k  =1,2$ the function $\chi_i^{k}$ satisfies $H \abs{\nabla \chi_i^{k}} = O(1)$, $0 \leq \chi_i^{k+1} \leq 1$, and $\sum_{i=1}^{\Nv} \chi_i^{k} = 1$. One can take $\{ \chi_i^k \}_{i=1}^{\Nv}$ to be the set of standard multiscale basis functions or the standard piecewise linear functions. 
 
Assume that the eigenvalues $\{\lambda_j^i\}$ (resp. $\{ \zeta_j^i \}$) are arranged in ascending order and that the eigenfunctions satisfy the normalization condition $s_i^1(v_j^i,v_j^i)=1$ as well as $s_i^2(q_j^i,q_j^i)=1$ for any $i$ and $j$. 
Next, choose $J_i^1\in \mathbb{N}^+$ and define the local auxiliary space $V_{\text{aux}}(K_i):= \text{span} \{v_j^{i}:1\leq j \leq J_i^1 \}$. 
Similarly, we choose $J_i^2 \in \mathbb{N}^+$ and define $Q_{\text{aux},1}(K_i) := \text{span} \{q_j^i: 1 \leq j \leq J_i^2\}$. 
Based on these local spaces, we define the global auxiliary spaces $V_{\text{aux}}$ and $Q_{\text{aux},1}$ by
$$V_{\text{aux}} := \bigoplus_{i=1}^{\Ne} V_{\text{aux}}(K_i) \quad \text{and} \quad Q_{\text{aux},1} := \bigoplus_{i=1}^{\Ne} Q_{\text{aux},1}(K_i).$$
The corresponding inner products for the global auxiliary multiscale spaces are defined by
\begin{eqnarray*}
s^1(u,v):=\sum_{i=1}^{\Ne}s_i^1(u,v), \qquad 
s^2(p,q):=\sum_{i=1}^{\Ne}s_i^2(p,q) 
\end{eqnarray*}
for all $u,v \in V_{\text{aux}}$ and $p,q\in Q_{\text{aux},1}$.

Further, we define projection operators $\pi_1: V \to V_{\text{aux}}$ and $\pi_2: Q \to Q_{\text{aux},1}$ such that for all $v \in V,\ q \in Q$ it holds that
\begin{eqnarray*}
\pi_1(v):=\sum_{i=1}^{\Ne}\sum_{j=1}^{J_i^1}s^1_i(v,v_j^i)v_j^i, \qquad
\pi_2(q):=\sum_{i=1}^{\Ne}\sum_{j=1}^{J_i^2}s^2_i(q,q_j^i)q_j^i.
\end{eqnarray*}
We also denote the kernel of the operator $\pi_1$ restricted to $V$ and the kernel of the operator $\pi_2$ restricted to $Q$ as 
$$
\tilde{V}=\{w\in V\mid \pi_1(w)=0\},\qquad
\tilde{Q}=\{q\in Q\mid \pi_2(q)=0\}.
$$

Next, we construct the multiscale spaces for the practical computations. 
For each coarse element $K_i$, we define the oversampled region $K_{i,\ell}\subset\Omega$ obtained by enlarging $K_i$ by $\ell$ layers, i.e.,
\begin{equation}
\label{eq:K_il}
 K_{i,0} := K_i, \quad K_{i,\ell} := \bigcup \left\{ K\in \mathcal{T}^H : K \cap  K_{i,\ell-1} \neq \emptyset \right\}, \quad \ell \in \mathbb{N}^+.
 \end{equation}
We define $V(K_{i,\ell}):=[H_0^1(K_{i,\ell})]^d$ and $Q(K_{i,\ell}):=H^1_0(K_{i,\ell})$. Then, for each pair of auxiliary functions $v_j^i\in V_{\text{aux}}$ and $q_j^i\in Q_{\text{aux},1}$,
we solve the following  minimization problems:
find $\psi_{j,\text{ms}}^{(i)}\in V(K_{i,\ell})$ such that
\begin{equation}\label{eq:mineq1}
\psi_{j,\text{ms}}^{(i)}=\text{argmin}\Big\{a(\psi,\psi)+s^1\big(\pi_1(\psi)-v_j^i,\pi_1(\psi)-v_j^i\big):\,\psi\in V(K_{i,\ell})\Big\}
\end{equation}
and 
find $\phi_{j,\text{ms}}^{(i)} \in Q(K_{i,\ell})$ such that
\begin{equation}\label{eq:mineq2}
\phi_{j,\text{ms}}^{(i)}=\text{argmin}\Big\{b(\phi,\phi)+s^2\big(\pi_2(\phi)-q_j^i,\pi_2(\phi)-q_j^i\big):\,\phi\in Q(K_{i,\ell})\Big\}.
\end{equation}
Note that problem (\ref{eq:mineq1}) is equivalent to the local problem 
\begin{equation}\label{eq:mineq1_loc}
a(\psi_{j,\text{ms}}^{(i)},v)+s^1\big(\pi_1(\psi_{j,\text{ms}}^{(i)}),\pi_1(v)\big)=s^1\big(v_j^i,\pi_1(v)\big)
\end{equation}
for all $v\in V(K_{i,\ell})$, whereas problem (\ref{eq:mineq2}) is equivalent to
\begin{equation}\label{eq:mineq2_loc}
b(\phi_{j,\text{ms}}^{(i)},q)+s^2\left(\pi_2(\phi_{j,\text{ms}}^{(i)}),\pi_2(q)\right)=s^2\big(q_j^i,\pi_2(q)\big)
\end{equation}
for all $q\in Q(K_{i,\ell})$.
Finally, for fixed parameters $\ell$, $J_i^1$, and $J_i^2$, the multiscale spaces $V_{H}$ and $Q_{H,1}$ are defined by 
 $$V_{H} := \text{span}\{ \psi_{j,\text{ms}}^{(i)}: 1\leq j \leq J_i^1,\ 1\leq i \leq \Ne \} \quad \text{and} \quad Q_{H,1} :=\text{span} \{ \phi_{j,\text{ms}}^{(i)}: 1 \leq j \leq J_i^2,\ 1\leq i \leq \Ne\}.$$

 The multiscale basis functions can be interpreted as approximations to global multiscale basis functions $\psi_j^{(i)}\in V$ and $\phi_j^{(i)}\in Q$, similarly defined by
 
\begin{align}\label{eq:minglo1}
\psi_j^{(i)} &= \text{argmin}\Big\{a(\psi,\psi)+s^1\big(\pi_1(\psi)-v_j^i,\pi_1(\psi)-v_j^i\big): \,\psi\in V\Big\}, \\\label{eq:minglo2}
\phi_j^{(i)}&= \text{argmin}\Big\{b(\phi,\phi)+s^2\big(\pi_2(\phi)-q_j^i,\pi_2(\phi)-q_j^i\big): \,\phi\in Q\Big\}.
\end{align}
This is equivalent to solve the global problems
\begin{align}\label{eq:minglo1eqn}
a(\psi_{j}^{(i)},v)+s^1\big(\pi_1(\psi_{j}^{(i)}),\pi_1(v)\big)&=s^1\big(v_j^i,\pi_1(v)\big), \quad \forall  v\in V,\\
\label{eq:minglo2eqn}
b(\phi_{j}^{(i)},q)+s^2\left(\pi_2(\phi_{j}^{(i)}),\pi_2(q)\right)&=s^2\big(q_j^i,\pi_2(q)\big),\quad \forall  q\in Q.
\end{align}
These basis functions have global support in the domain $\Omega$ but, as shown in \cite{chung2017constraint}, decay exponentially outside some local (oversampled) region. This property plays a vital role in the convergence analysis of CEM-GMsFEM and justifies the use of local basis functions in $V_{H}$ and $Q_{H,1}$. 

The discussions related to CEM-GMsFEM as presented in \cite{chung2017constraint} show that $Q_{H,1}$ is a good approximation of the complement of $\tilde{Q}$. Thus, aiming to enhance the simulation accuracy of the pressure term, we construct a space $Q_{H,2}$ in $\tilde{Q}$.

\subsection{$Q_{H,2}$ Construction}
In this section, we construct a subspace $Q_{H,2}$ that satisfies the CFL-type condition \eqref{stability_condition_2} such that the scheme is stable. The motivation of the construction and the proof of the proposed $Q_{H,2}$ that meets the CFL-type condition can be referred to \cite{chung2022contrast}. 
One of the constructions for $Q_{H,2}$ is based on a CEM (Constrained Energy Minimization) type structure, which is similar to CEM-GMsFEM. 
For each coarse element $K_i$, we solve an eigenvalue problem to get the second type of auxiliary basis. More precisely, we find eigenpairs $(\xi_j^{(i)},\gamma_j^{(i)})\in(Q(K_i)\cap \tilde{Q})\times \mathbb{R}$ by solving
$$
b_i(\xi_j^{(i)},v)= \gamma^{(i)}_j c(\xi^{(i)}_j, v), \forall v\in V(K_i)\cap \tilde{V}\times \mathbb{R}
$$

For each $K_i$, we rearrange the eigenvalues in an ascending order and choose the first few $J_i$ eigenfunctions corresponding to the smallest $J_i$ eigenvalues. The resulting auxiliary space is 
$$
Q_{\text{aux},2} := \text{span}\{\xi_j^{(i)} : 1\leq i\leq N_e, 1\leq j\leq J_i\}. 
$$

With auxiliary space $V_{\text{aux},2}$, we construct the basis for $V_{H,2}$ by solving the following problem on each oversampling region $K_{i,l}$, which is defined in \eqref{eq:K_il}. For each auxiliary basis function $\xi_j^{(i)}$, we define the basis function $\phi_{j,\ms 2}^{(i)} \in Q(K_{i,l})$ such that for some $\mu_j^{(i),1}\in Q_{\text{aux},1}$ and $\mu_j^{(i),2}\in Q_{\text{aux},2}$, we have
\begin{subequations}
    \begin{alignat}{3}
       b\left(\phi_{j,\ms 2}^{(i)} , q\right)+s^2\left(\mu_{j}^{(i), 1}, q\right)+c\left(\mu_{j}^{(i), 2}, q\right) &=0, \quad \forall q \in Q\left(K_{i,l}\right) \\ 
       s^2\left(\phi_{j,\ms 2}^{(i)} , q\right) &=0, \quad \forall q \in Q_{\text{aux},1} \\
       c\left(\phi_{j,\ms 2}^{(i)} , q\right) &=c\left(\xi_{j}^{(i)}, q\right), \quad \forall q \in Q_{\text{aux},2}
    \end{alignat}
\end{subequations}

We define 
$$
Q_{H,2}= \text{span}\{ \phi_{j,\ms 2}^{(i)} \mid 1\leq i \leq N_e, 1\leq j \leq J_i\}.
$$

For any $q\in Q_{H,2}$, we can derive as in \cite{chung2022contrast} that
$$
\norm{q}_b^2\leq C_1^2 H^{-2} \norm{q}_c^2
$$
and hence we get a more explicit stability condition based on \eqref{stability_condition_2}:
$$
\tau \leq C_1^{-2} H^2(1-\gamma_c).
$$

\section{Numerical Results}\label{sec:numerical}
In this section, we present some numerical results obtained by using the proposed partially explicit method. 
The (relative) energy errors between the multiscale and the fine-scale solutions are defined below 
by$$e_{L^2}^n:=\frac{\norm{p_{\ms}^{n}-p_h^n}}{\norm{p_h^n}} \tand e_{energy}^n :=\frac{\norm{p_{\ms}^{n}-p_h^n}_b}{\norm{p_h^n}_b};$$
they will serve as indicators for evaluating the performance of algorithm. 
\begin{figure}[htbp!]
\centering
\includegraphics[width = 1.7in]{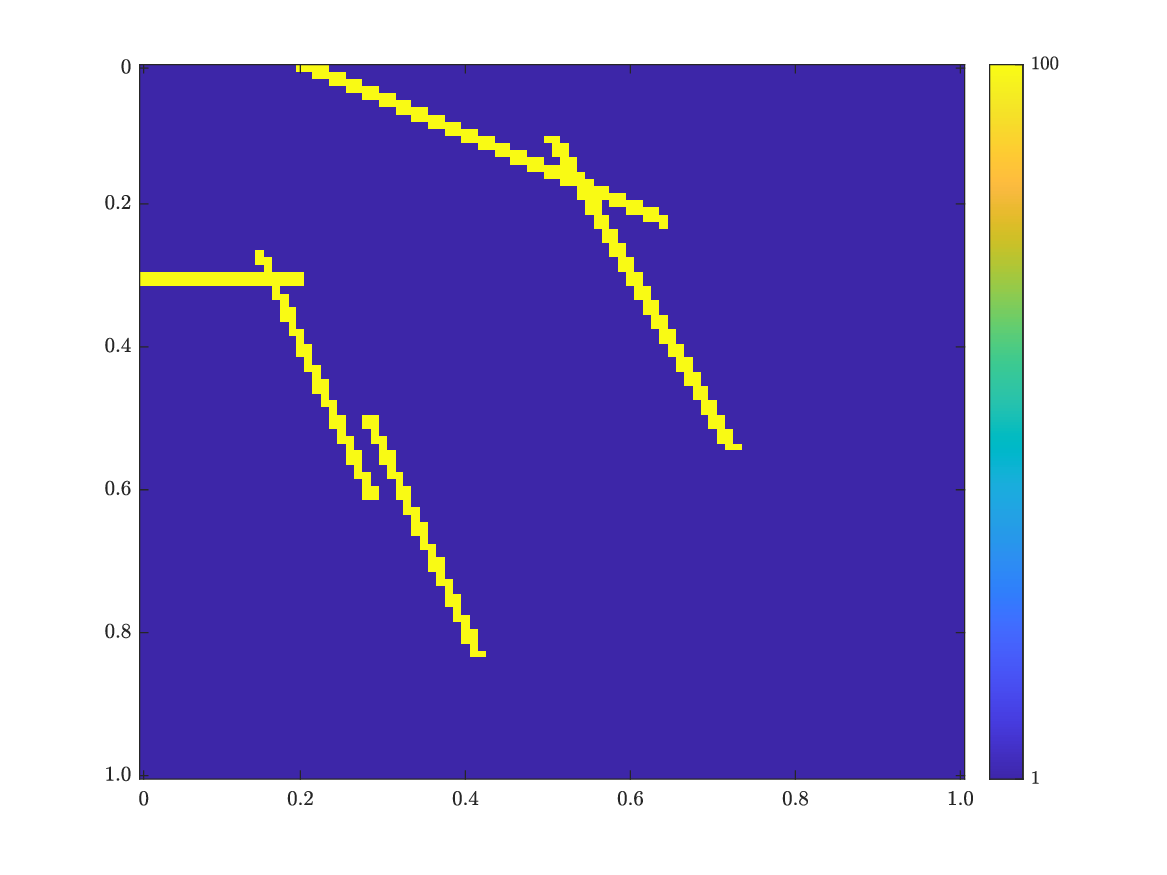}\quad \includegraphics[width = 1.7in]{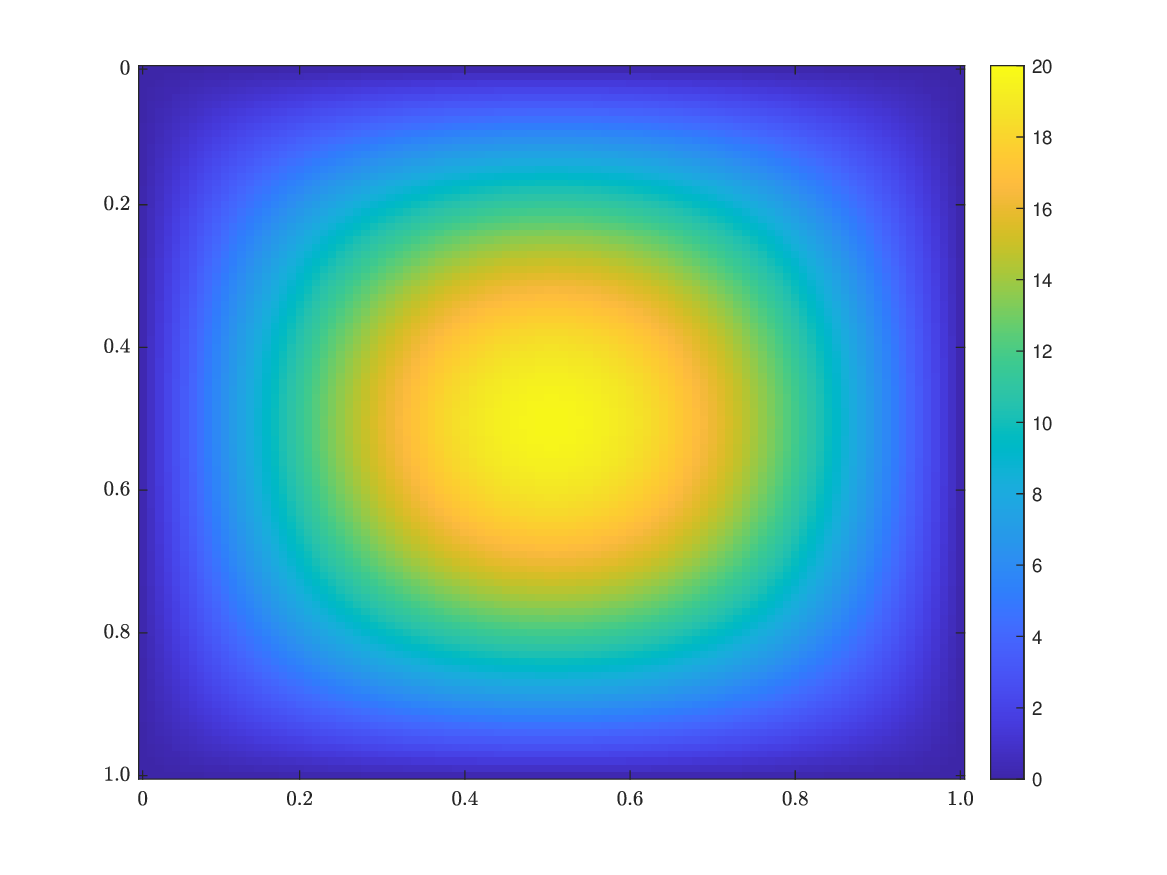}
\includegraphics[width = 1.7in]{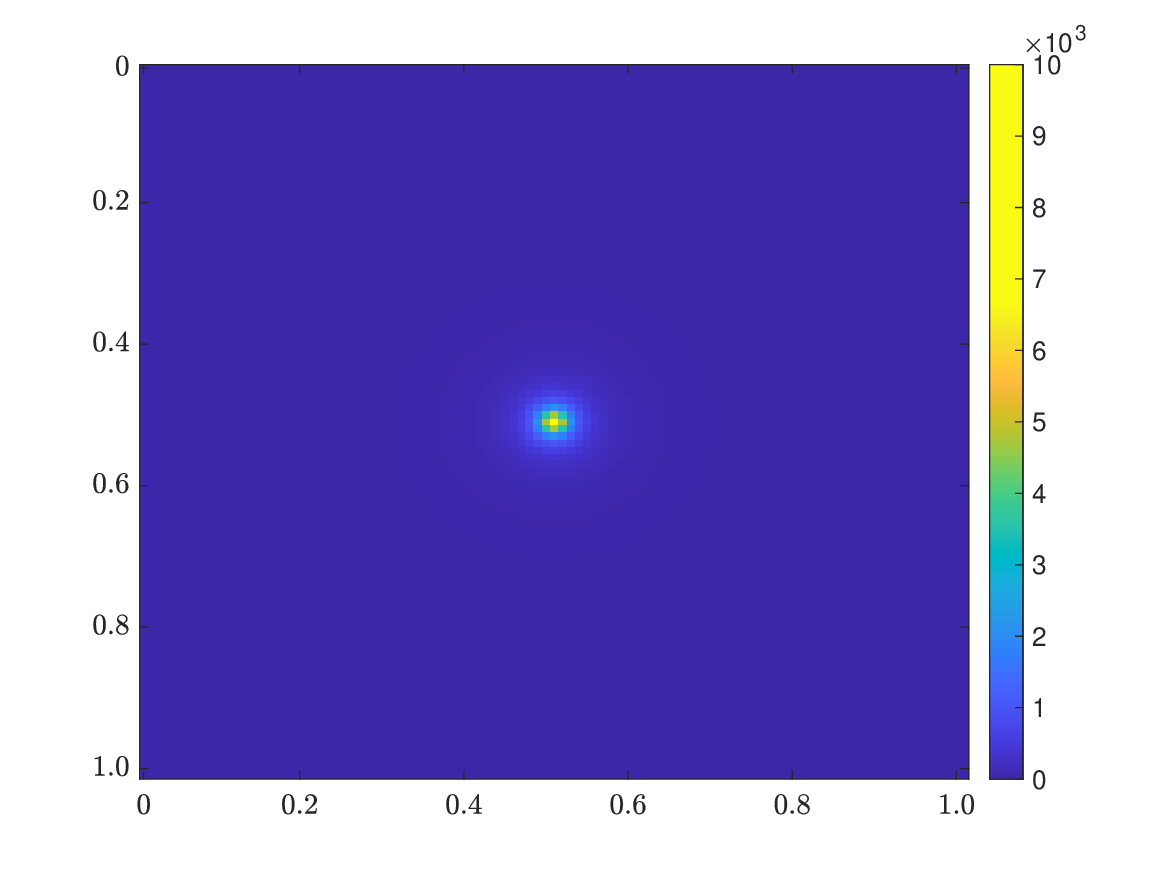}
\caption{Left: Young's Modulus for Example \ref{exp:1}. Middle: A smooth source term $f_1$. Right: A nearly singular source term $f_2$.}
\label{fig:ym_exp1_partial}
\end{figure}

For each example, we test and compare three different schemes
\begin{itemize}
\item the implicit method with CEM bases.
\item the implicit method with CEM bases and additional degrees of freedom from $Q_{H,2}$. 
\item the partially explicit splitting scheme with CEM bases and additional degrees of freedom from $Q_{H,2}$.
\end{itemize}
\begin{example}\label{exp:1}
The problem setting in this example is as follows: 
\begin{enumerate}
\item We set $\Omega = (0,1)^2$, $\alpha = 0.9$, $M = 1$, $\nu_p = 0.2$ (Poisson ratio), and $\nu = 1$. 
\item  The terminal time is $T = 1/100$. The time step size is $\tau = 10^{-4}$, i.e., $N = 100$. 
\item The Young's modulus $E$ is depicted in Figure \ref{fig:ym_exp1_partial}. We set the permeability to be $\kappa = E$. The permeability field is heterogeneous with high contrast streaks. 
\item In this example, we test the model with two source functions separately: 

$f_1(x_1, x_2) \equiv 2 \pi^2\sin(\pi x_1) \sin(\pi x_2)$ and $f_2(x_1,x_2) \equiv 1/\left((x_1-0.5)^2+(x_2-0.5)^2+10^{-4}\right)$. 

The initial condition is $p(x_1, x_2) = 100x_1(1-x_1)x_2(1-x_2)$. 
\item The number of (local) offline (pressure and displacement) basis functions is $J = J_i^1 = J_i^2 = 2$ and the number of oversampling layers is $\ell = 2$. 
\item The number of (local) $Q_{H,2}$ basis functions is $J_i=2$. 
\item The coarse mesh is $10 \times 10$ and the overall fine mesh is $100 \times 100$. 
\end{enumerate}
\end{example}

Figures \ref{fig:5_1_soln_profile_partial} and \ref{fig:5_1_0soln_profile_partial} show the solution profiles at the terminal time $T=0.01$ for source term $f_1$ and $f_2$ respectively. In Tables \ref{exp:5_1_energy} - \ref{exp:5_1_1l2}, we present the energy errors and $L_2$ errors of the pressure term with time. We also employ illustrative diagrams Figures \ref{fig:error_partial} and \ref{fig:error_partial_2} to provide an intuitive visualization of the aforementioned errors. The red curve denotes the error due to CEM without additional basis functions. 
The additional degrees of freedom are treated both implicitly (blue curve) and explicitly (yellow
curve). As we see, these two curves coincide. This indicates that our proposed partial explicit
method with the time stepping which is chosen independent of contrast provides as accurate solution as full backward Euler. Consequently, this backs up our discussions on the partially explicit method. The numerical results shows that with the additional space $Q_{H,2}$, the approximation for displacement $u$ does not improve prominently and hence the data of errors has not been included in this paper. With additional basis functions in $Q_{H,2}$,  the approximation for pressure $p$ is improved in terms of energy norm but the improvement in terms of $L_2$ norm is not obvious. The CEM basis can effectively reflects details of the solution due to the high-contrast features in the media and the $Q_{H,2}$ serves as a supplement for missing information. In the simulation, when the source term $f_1$ is smooth, CEM-GMsFEM without additional basis functions produces results similar to those of CEM-GMsFEM with additional basis functions, whether treated explicitly or implicitly. However, for the nearly singular source term $f_2$, the implicit CEM-GMsFEM with $Q_{H,2}$ and the partially explicit method outperforms the implicit CEM-GMsFEM without $Q_{H,2}$.

\begin{figure}[htbp!]
\centering
\includegraphics[width = 1.7in]{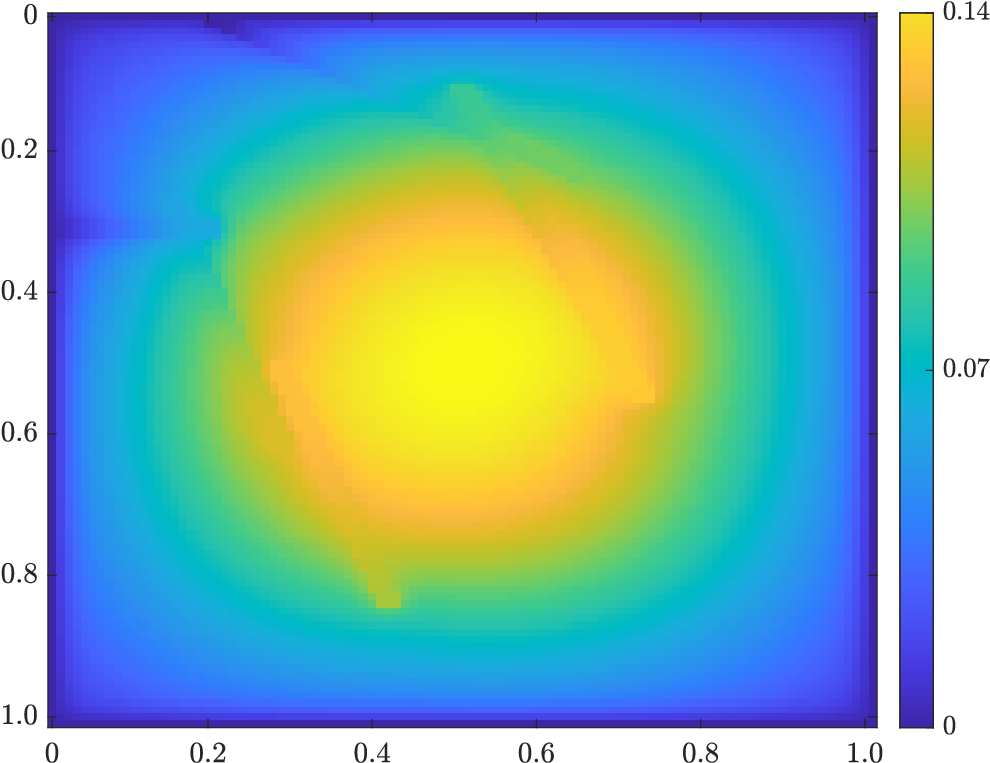} \quad
\includegraphics[width = 1.7in]{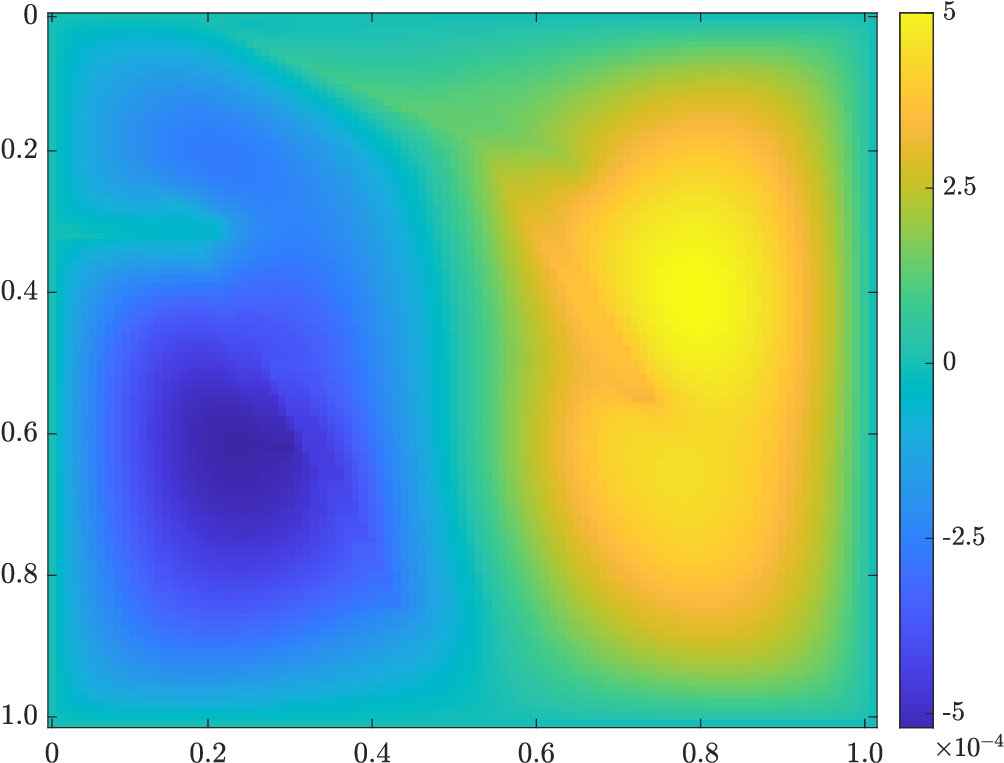}\quad
\includegraphics[width = 1.7in]{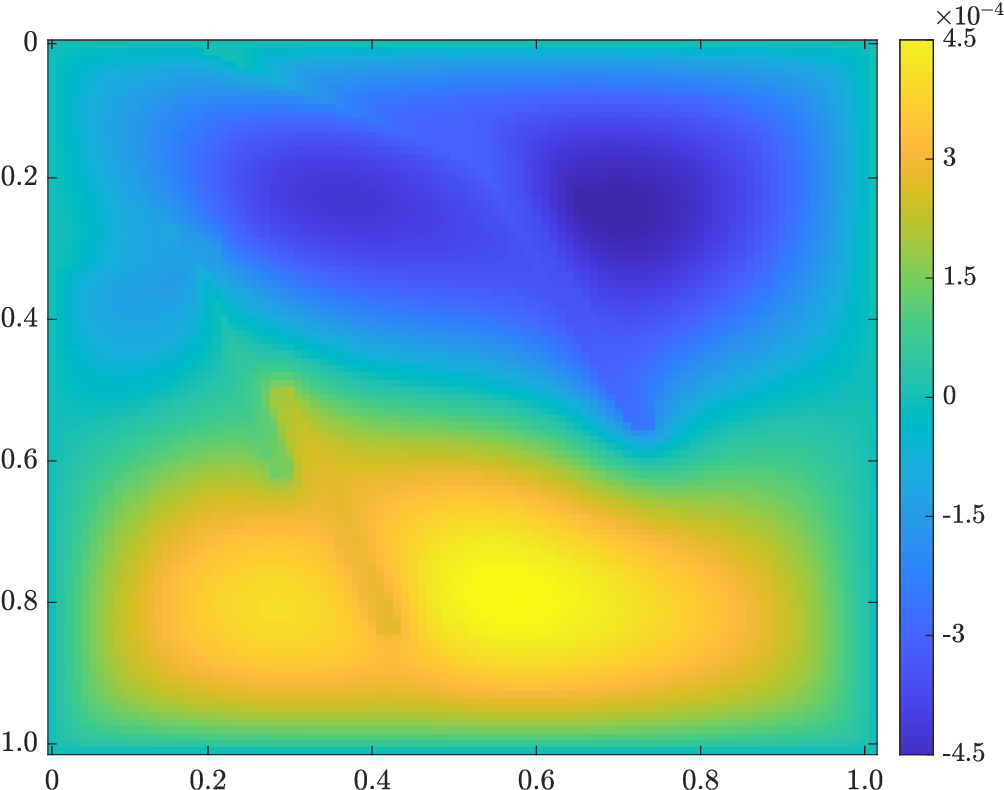}
\includegraphics[width = 1.7in]{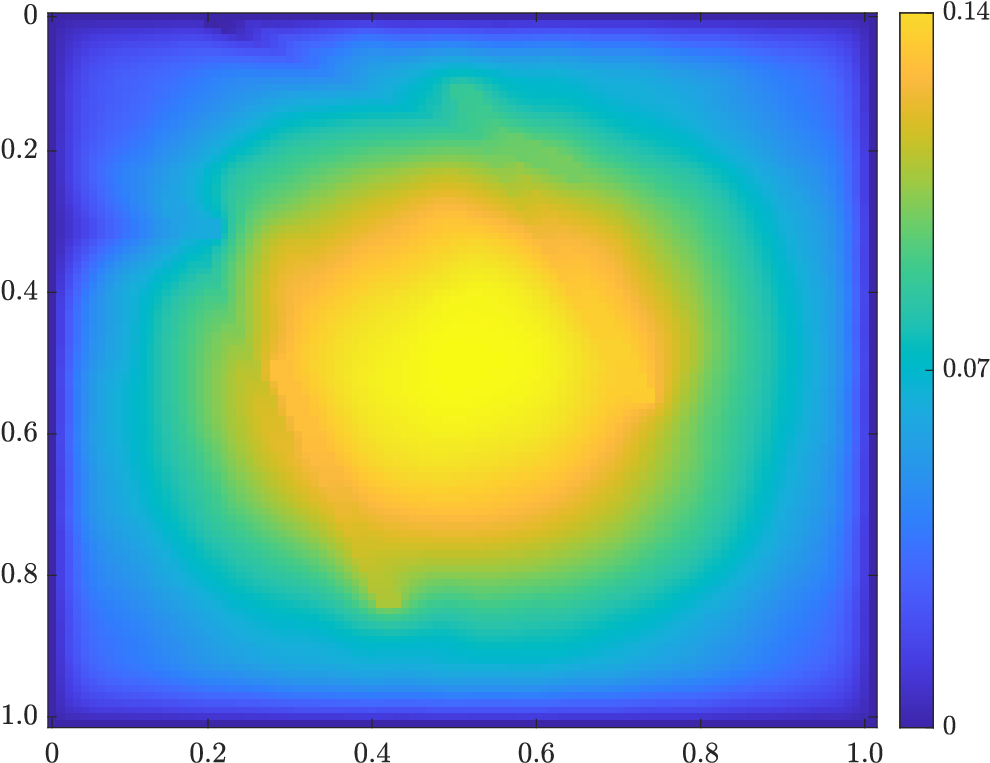}\quad
\includegraphics[width = 1.7in]{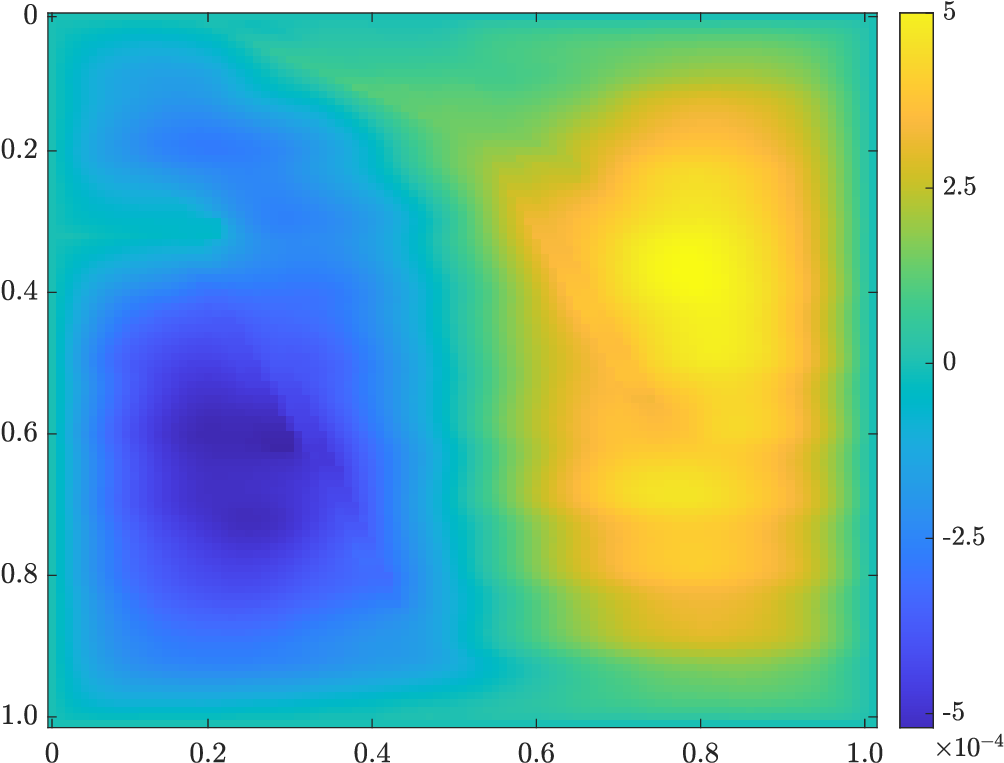}\quad
\includegraphics[width = 1.7in]{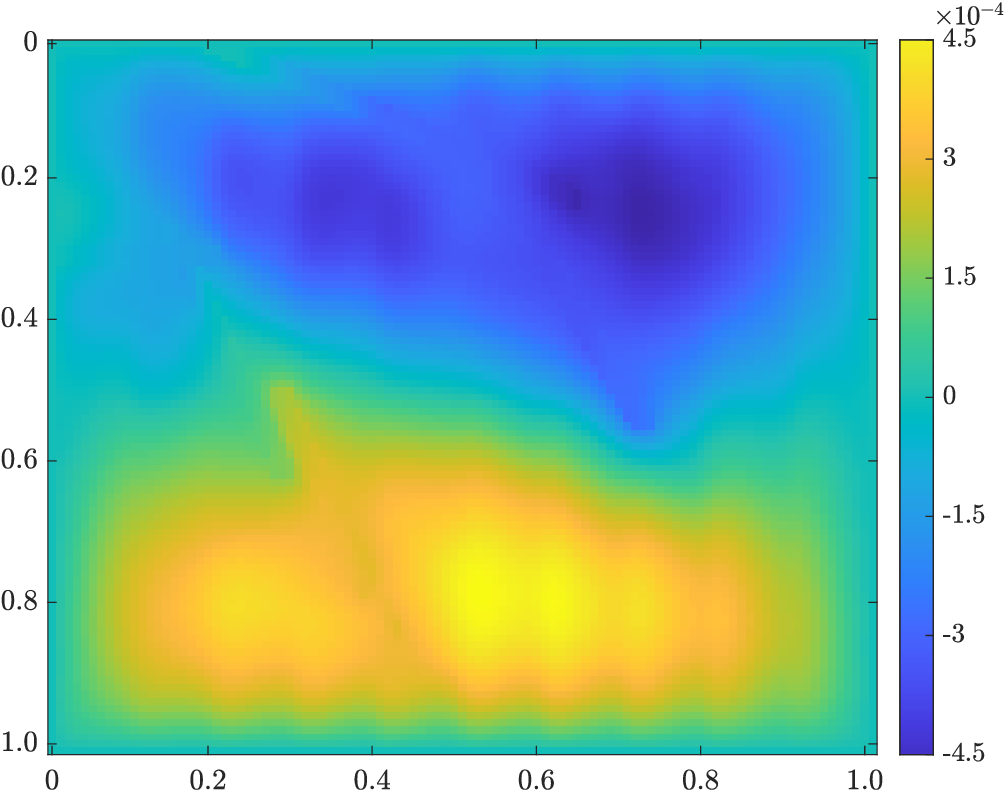}
\includegraphics[width = 1.7in]{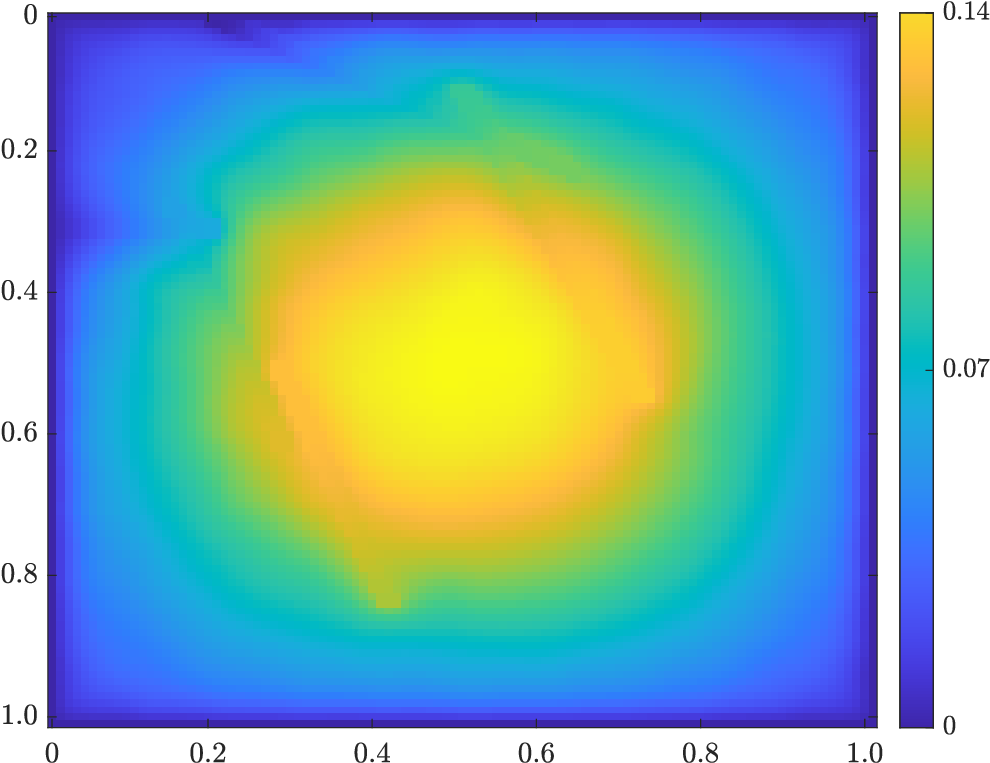}\quad
\includegraphics[width = 1.7in]{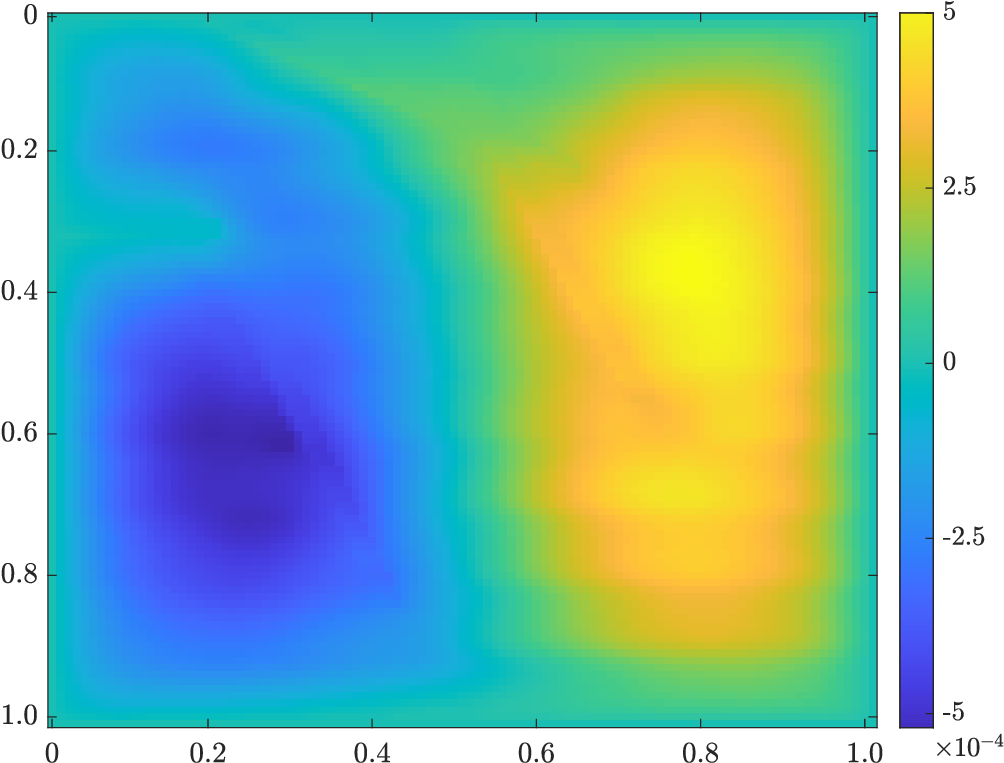}\quad
\includegraphics[width = 1.7in]{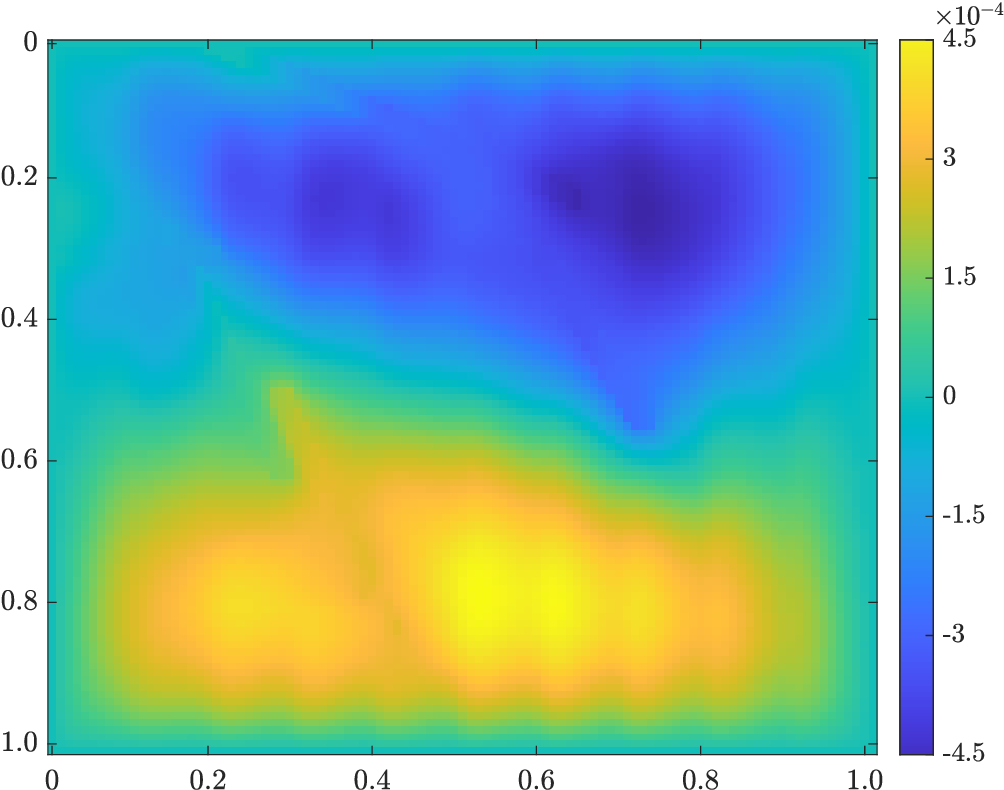}
\includegraphics[width = 1.7in]{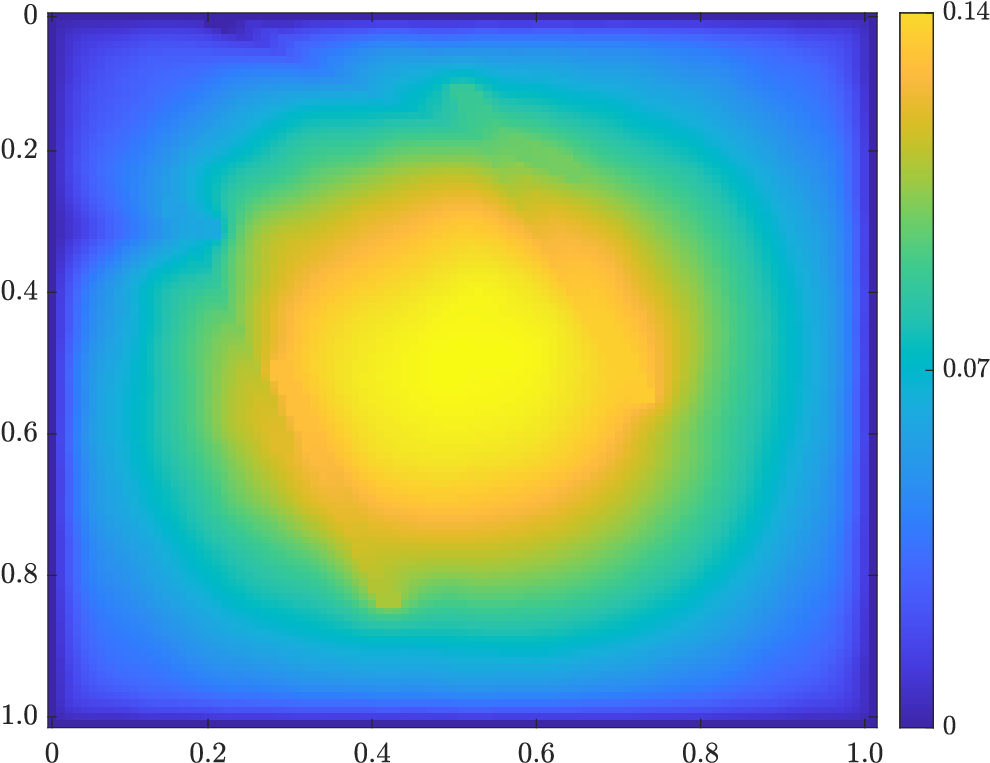}\quad
\includegraphics[width = 1.7in]{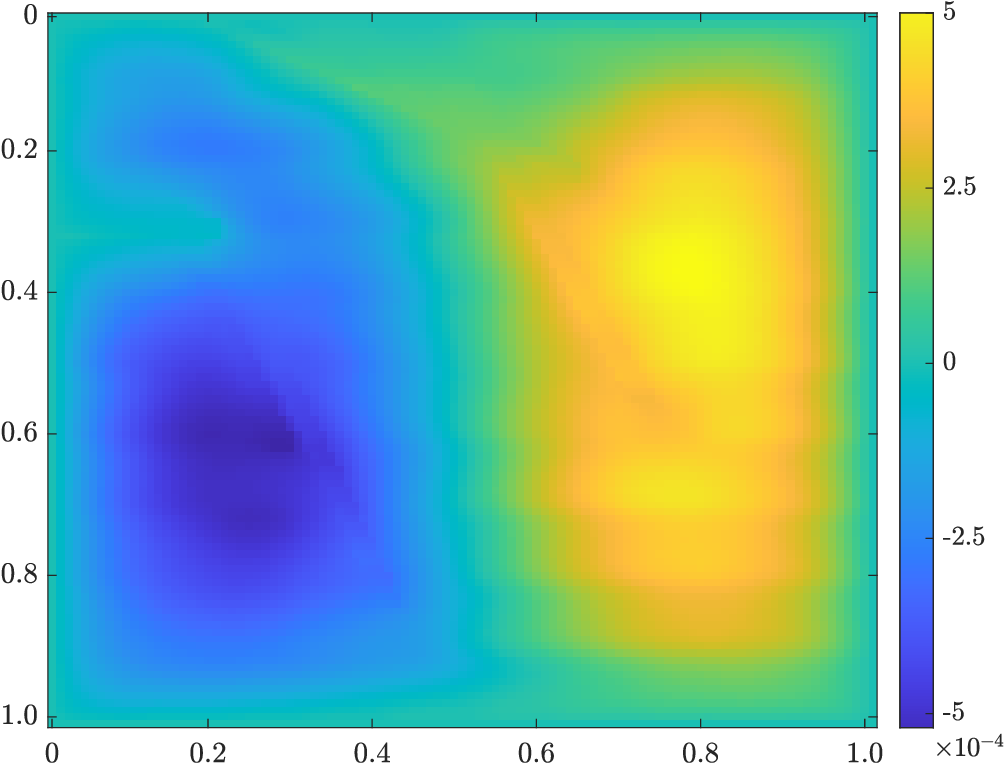}\quad
\includegraphics[width = 1.7in]{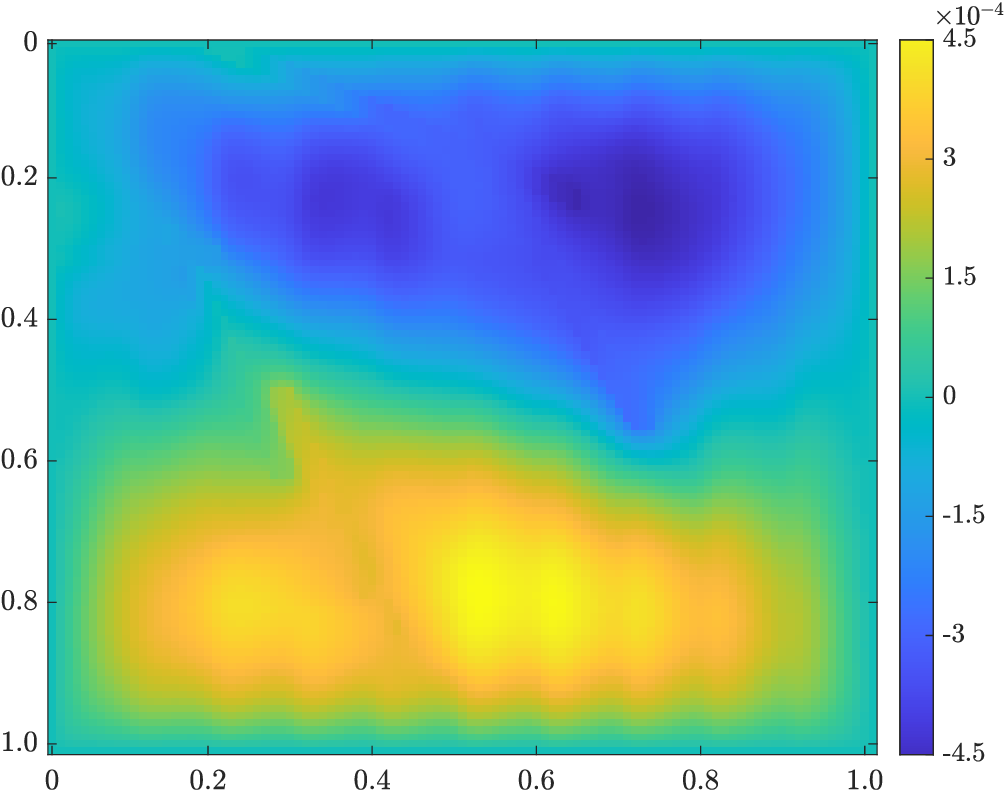}
\caption{Solution profiles for (starting from left to right) $p$, $u_1$, $u_2$ at $T=0.01$ of Example \ref{exp:1} with source function $f_1$. Top Row: Reference solutions with fully implicit fine-scale approach. Second Row: Implicit CEM-GMsFEM.  Third Row: Implicit CEM-GMsFEM with additional basis $Q_{H,2}$. Bottom Row: Proposed splitting method with additional basis $Q_{H,2}$ .}
\label{fig:5_1_soln_profile_partial}
\end{figure}

\begin{figure}[H]
\centering
\includegraphics[width=3.05 in]{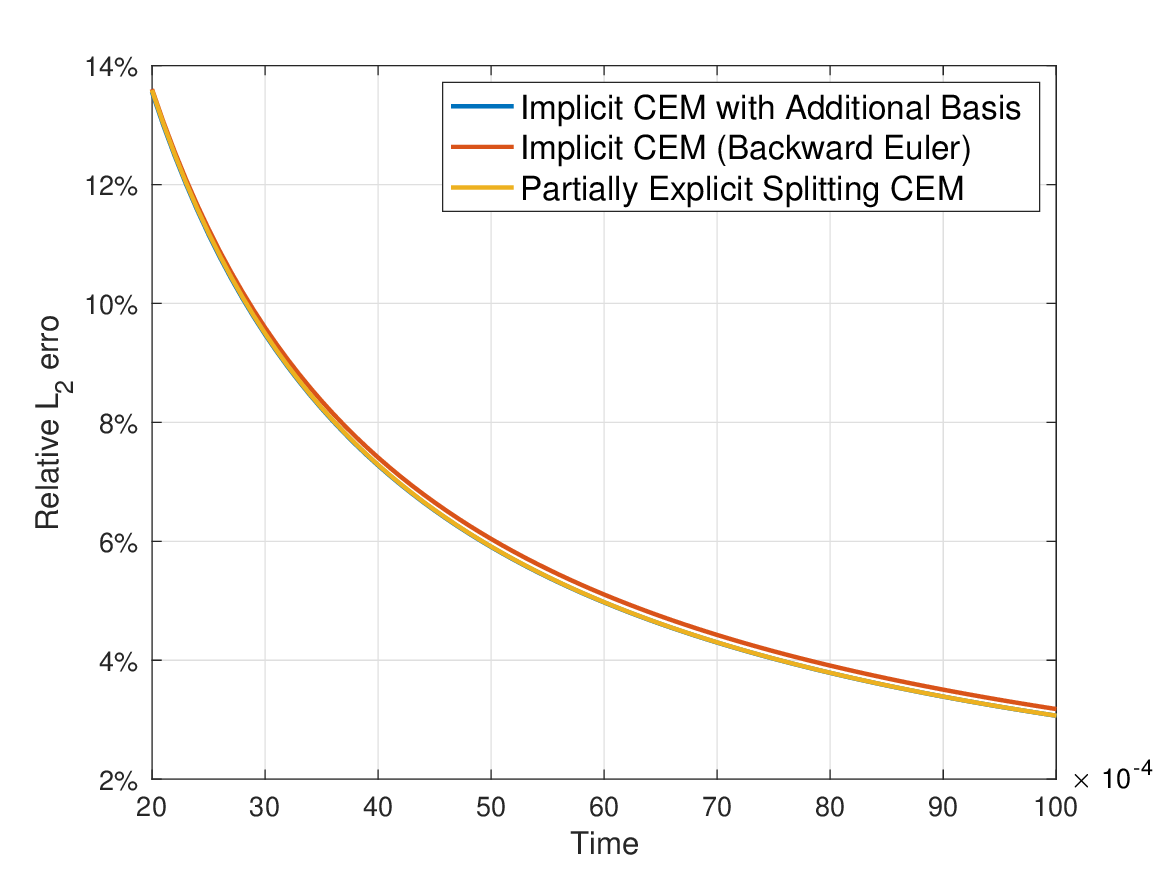} \quad
\includegraphics[width=3.05in]{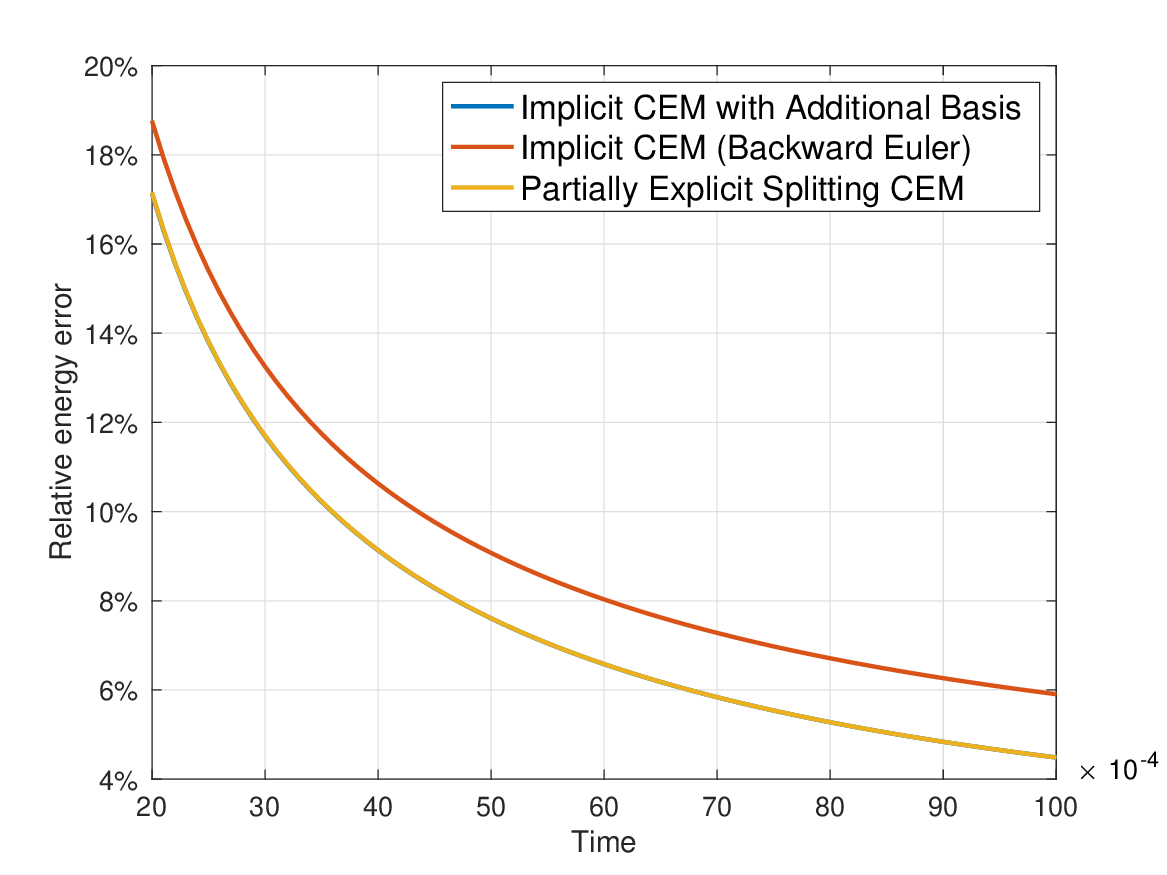}
\caption{Example \ref{exp:1} with source function $f_1$. Left: $L_2$ error against time.  Right: Energy error against time. }
\label{fig:error_partial}
\end{figure}

\begin{table}[H]
\centering
\resizebox{\textwidth}{!}{
\begin{tabular}{c||c|c|c}
Time Step $n$& Implicit CEM-GMsFEM with $Q_{H,2}$  & Implicit CEM-GMsFEM & Partially explicit method \\ 
\hline
   $ 1   $ &  $ 569.97 \% $  & $ 603.18 \% $  & $ 620.32 \% $  \cr
   $ 21   $ &  $ 15.58 \% $  & $ 17.21 \% $  & $ 15.59 \% $  \cr
   $ 41   $ &  $ 8.77 \% $  & $ 10.26 \% $  & $ 8.77\% $  \cr
   $ 61   $ &  $ 6.41 \% $  & $ 7.86 \% $  & $ 6.41 \% $  \cr
   $ 81   $ &  $ 5.18 \% $  & $ 6.61 \% $  & $ 5.18 \% $  \cr
   $ 100  $ &  $ 4.45 \% $  & $ 5.87 \% $  & $ 4.46 \% $ 
\end{tabular}
}
\caption{Energy error $e_{energy}^n$ of pressure $p$ against time step for Example \ref{exp:1} with source function $f_1$. . }
\label{exp:5_1_energy}
\end{table}

\begin{table}[H]
\centering
\resizebox{\textwidth}{!}{
\begin{tabular}{c||c|c|c}
Time Step $n$& Implicit CEM-GMsFEM with $Q_{H,2}$ & Implicit CEM-GMsFEM & Partially explicit method \\ 
\hline
   $ 1   $ &  $102.71 \% $  & $ 102.81 \% $  & $ 104.06 \% $  \cr
   $ 21   $ &  $ 12.48 \% $  & $ 12.55 \% $  & $ 12.51 \% $  \cr
   $ 41   $ &  $6.95 \% $  & $ 7.09 \% $  & $ 6.96\% $  \cr
   $ 61   $ &  $ 4.82 \% $  & $ 4.95 \% $  & $ 4.82 \% $  \cr
   $ 81  $ &  $ 3.70 \% $  & $ 3.82 \% $  & $ 3.70 \% $  \cr
   $ 100   $ &  $ 3.04 \% $  & $ 3.15 \% $  & $ 3.04 \% $
\end{tabular}
}
\caption{$L_2$ error  $e_{L_2}^n$ of pressure $p$ against time step for Example \ref{exp:1} with source function $f_1$. }
\label{exp:5_1_l2}
\end{table}

\begin{figure}[htbp!]
\centering
\includegraphics[width = 1.7in]{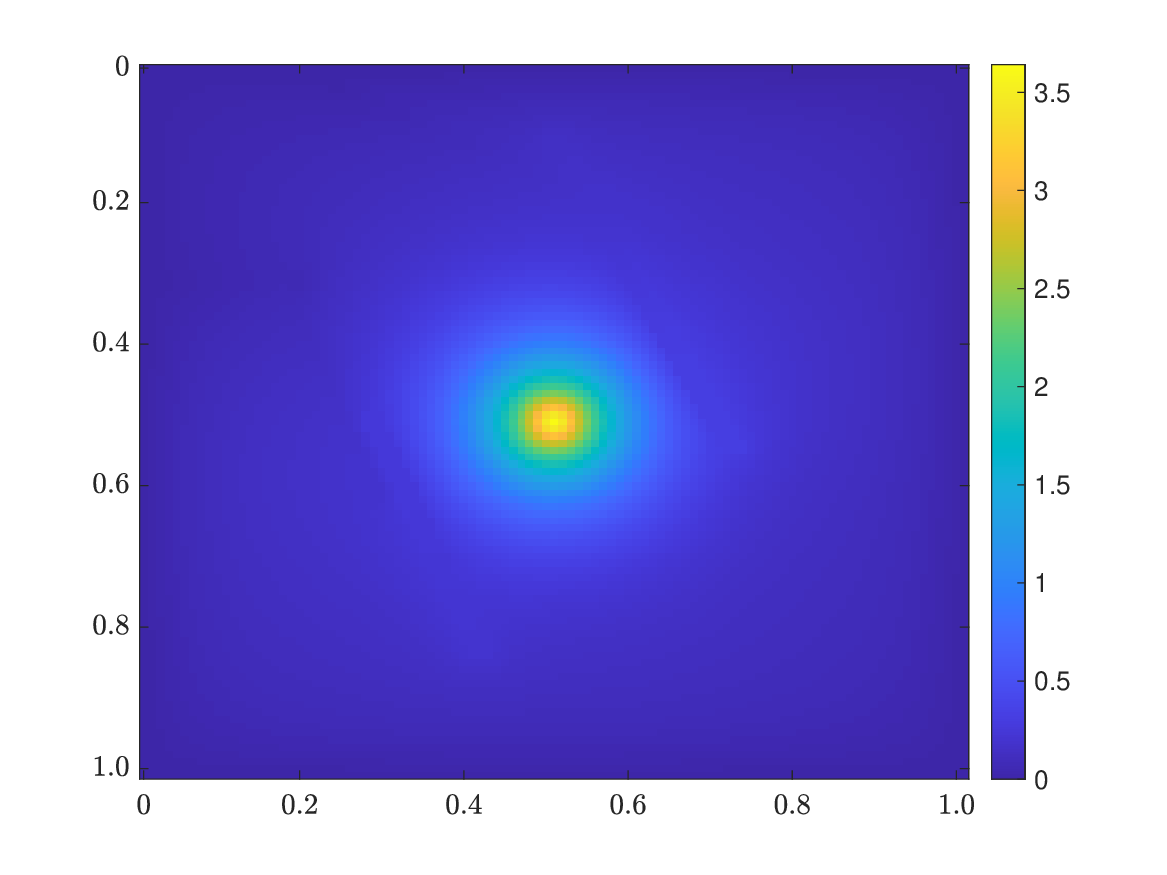} \quad
\includegraphics[width = 1.7in]{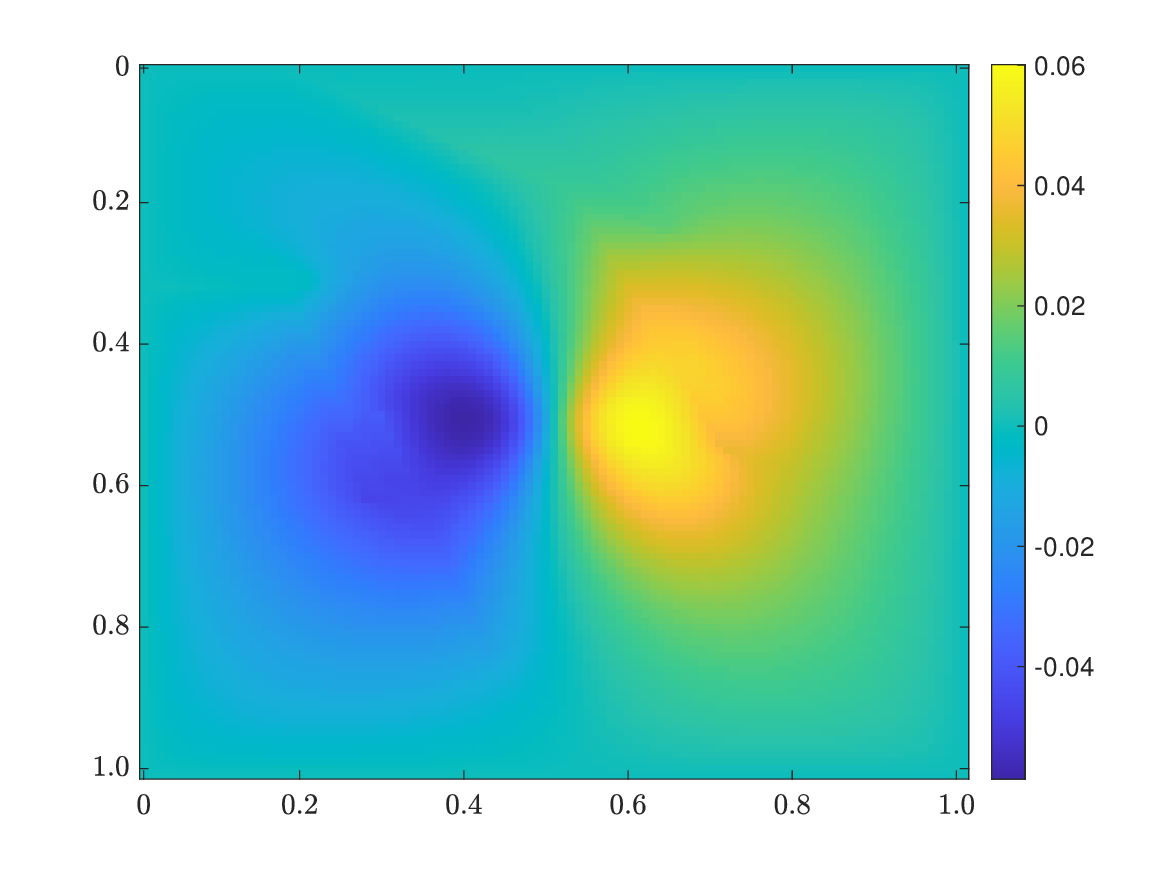}\quad
\includegraphics[width = 1.7in]{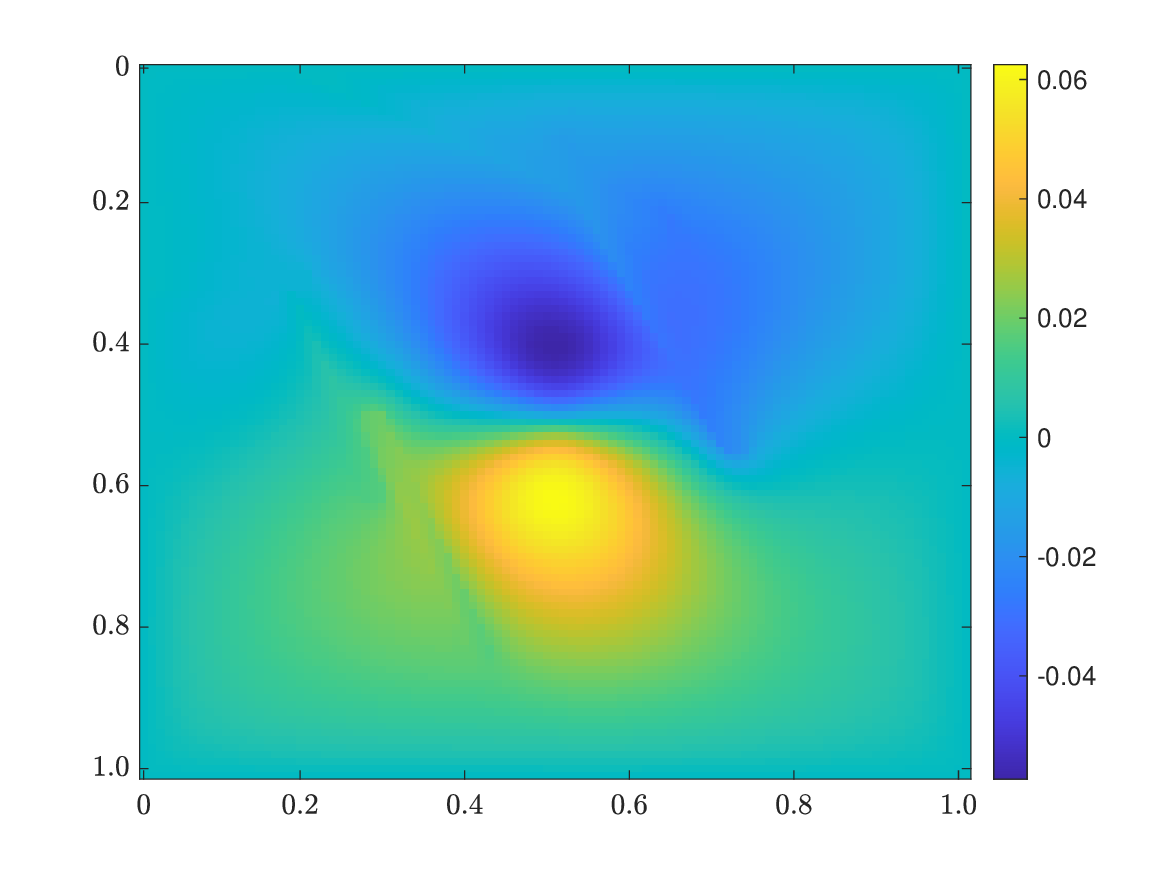}
\includegraphics[width = 1.7in]{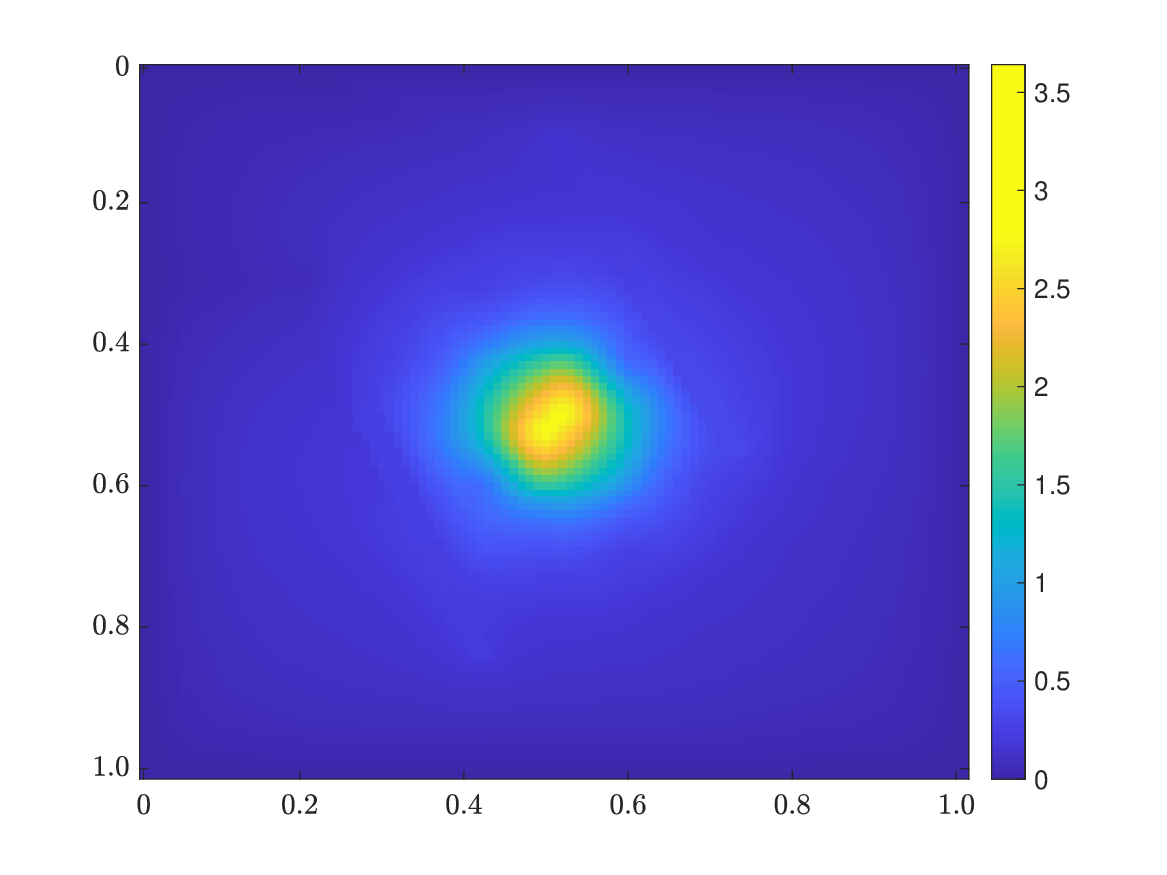}\quad
\includegraphics[width = 1.7in]{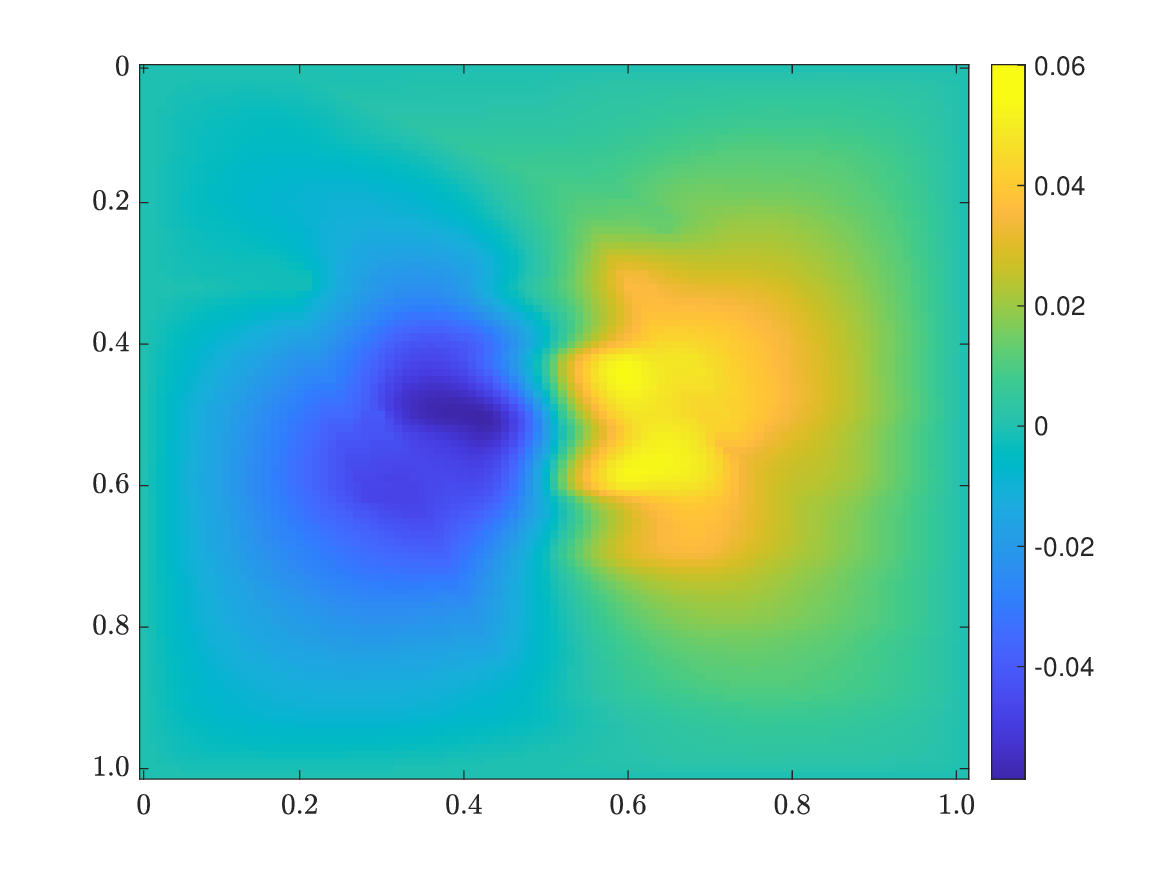}\quad
\includegraphics[width = 1.7in]{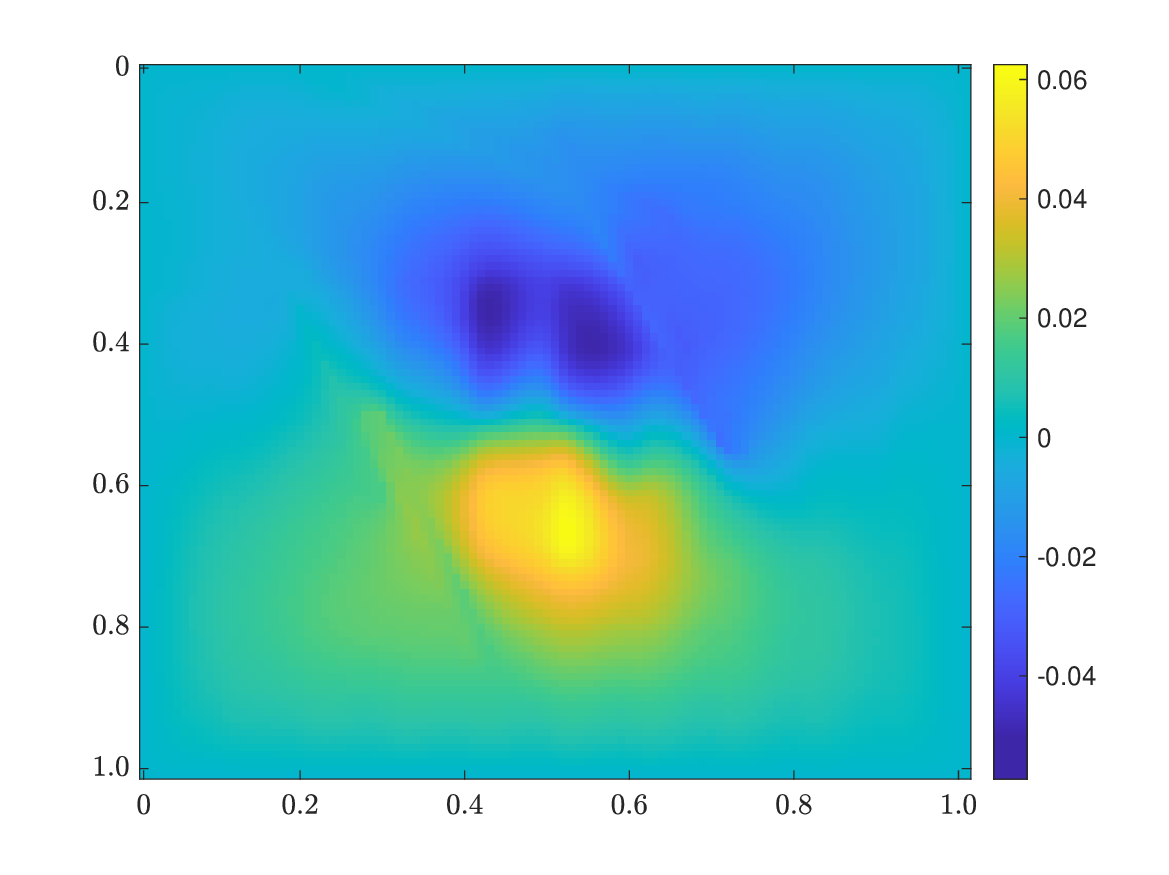}
\includegraphics[width = 1.7in]{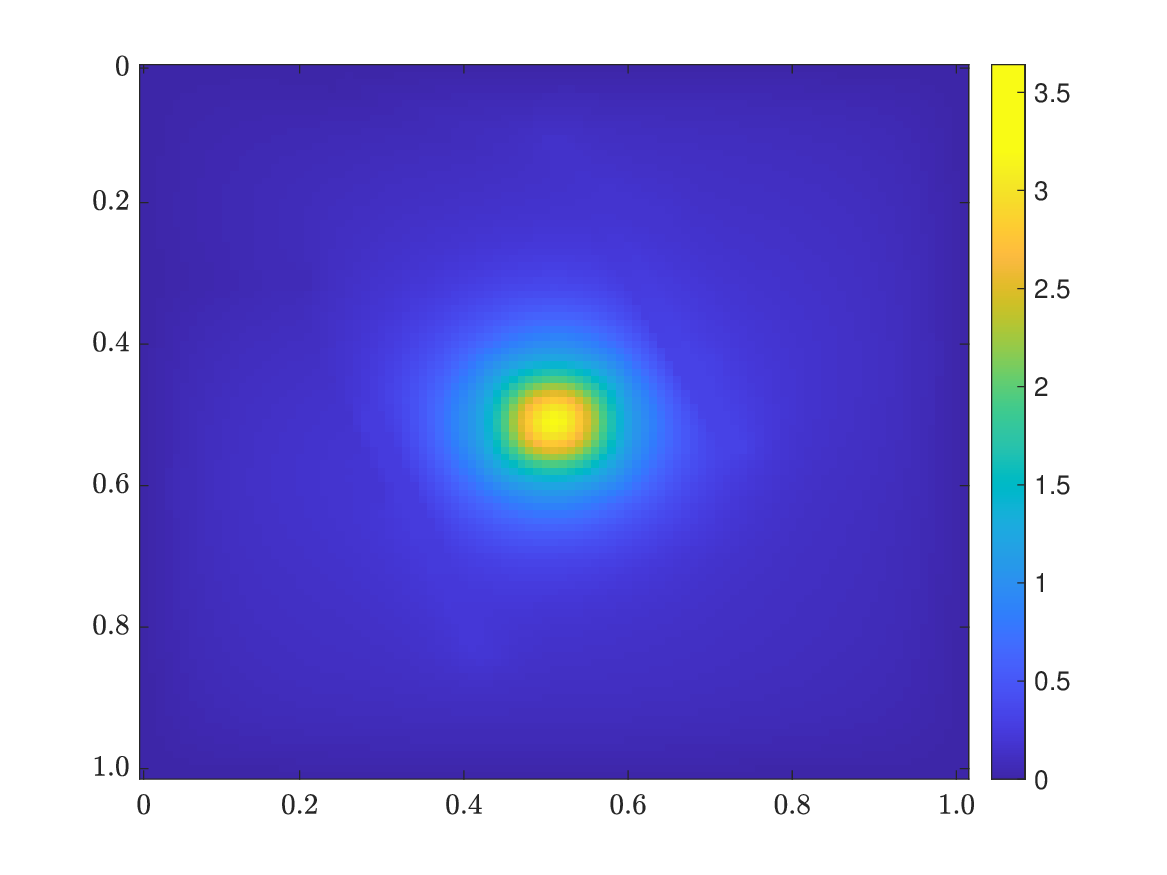}\quad
\includegraphics[width = 1.7in]{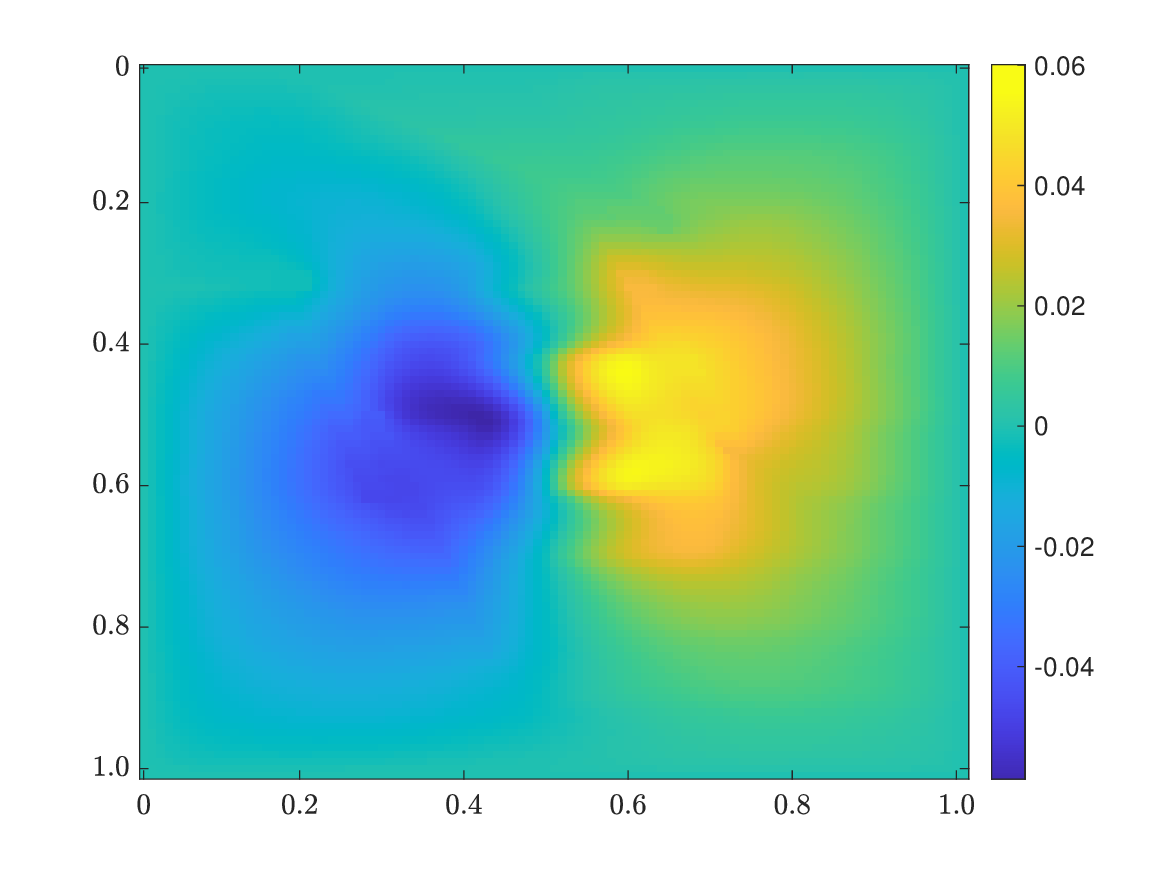}\quad
\includegraphics[width = 1.7in]{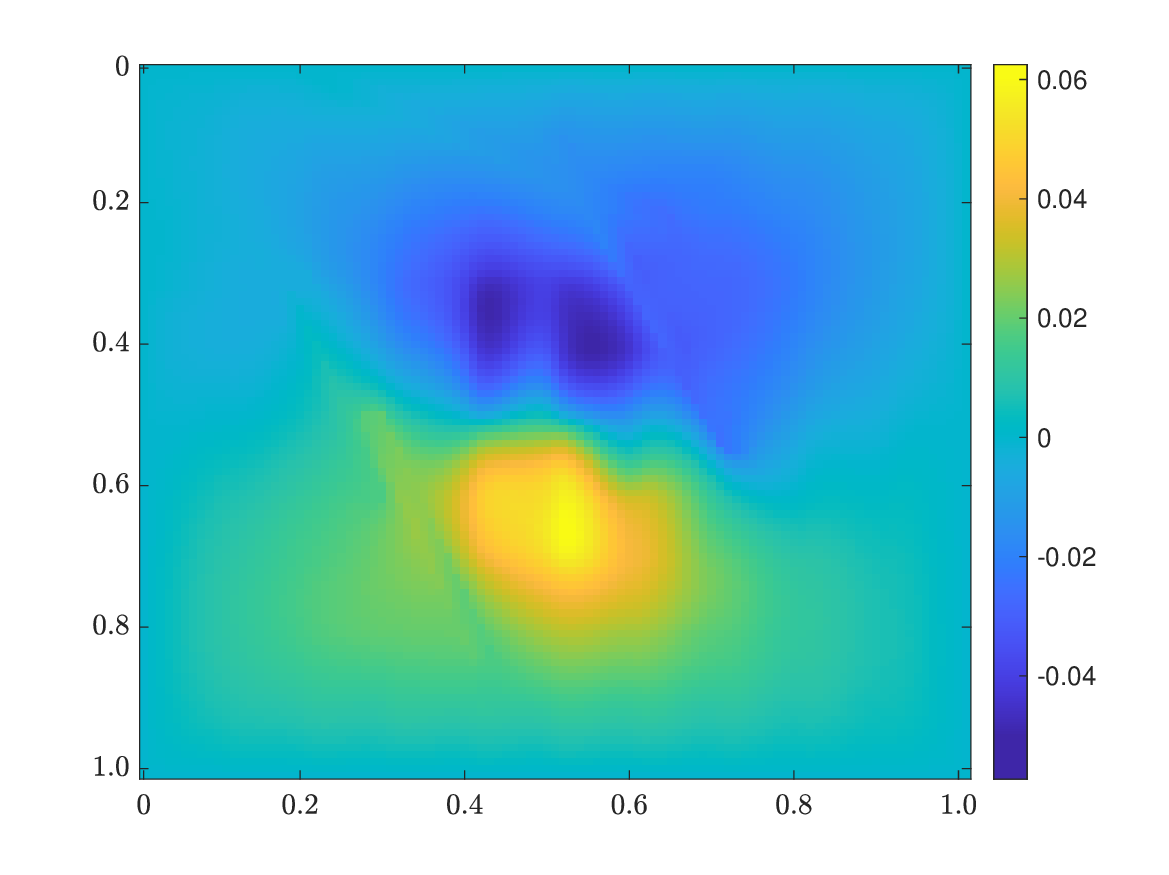}
\includegraphics[width = 1.7in]{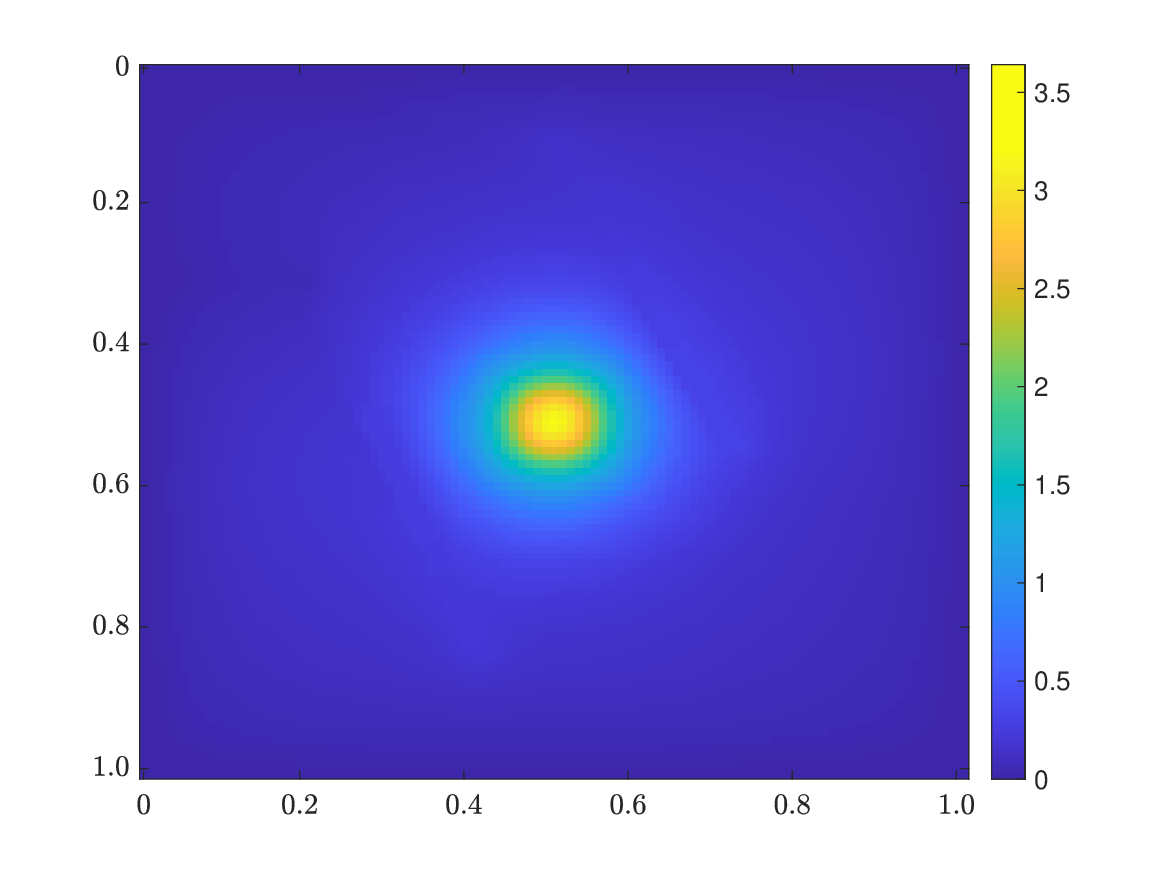}\quad
\includegraphics[width = 1.7in]{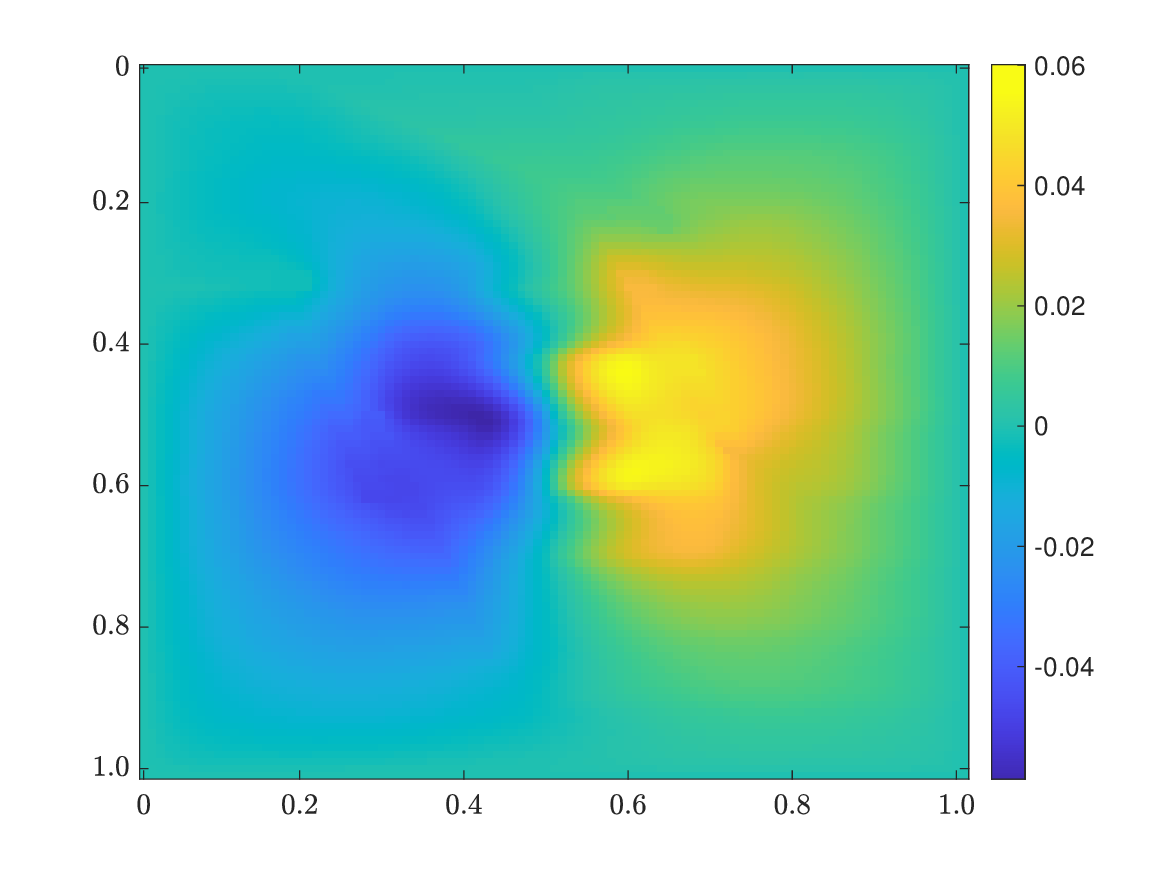}\quad
\includegraphics[width = 1.7in]{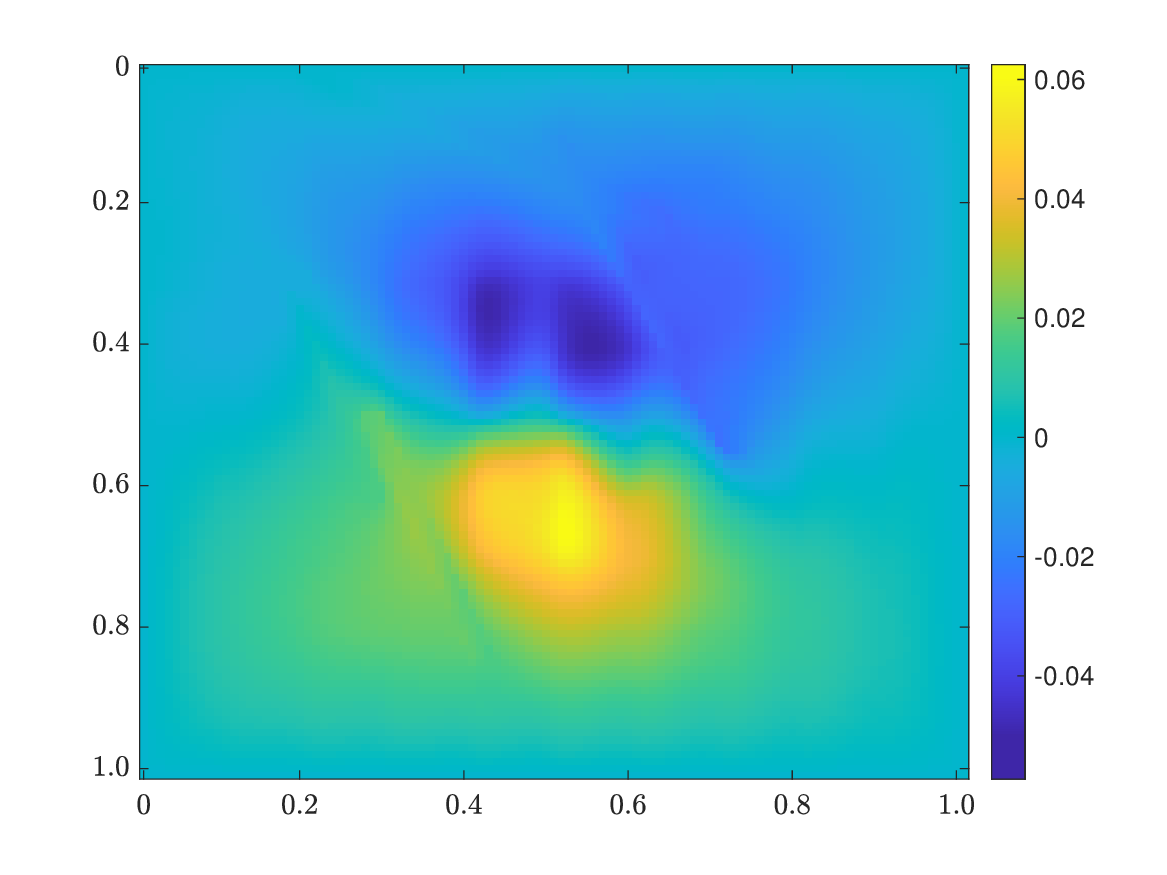}
\caption{Solution profiles for (starting from left to right) $p$, $u_1$, $u_2$ at $T=0.01$ of Example \ref{exp:1} with source function $f_2$. Top Row: Reference solutions with fully implicit fine-scale approach. Second Row: Implicit CEM-GMsFEM.  Third Row: Implicit CEM-GMsFEM with additional basis $Q_{H,2}$. Bottom Row: Proposed splitting method with additional basis $Q_{H,2}$ .}
\label{fig:5_1_0soln_profile_partial}
\end{figure}

\begin{figure}[H]
\centering
\includegraphics[width=3.05 in]{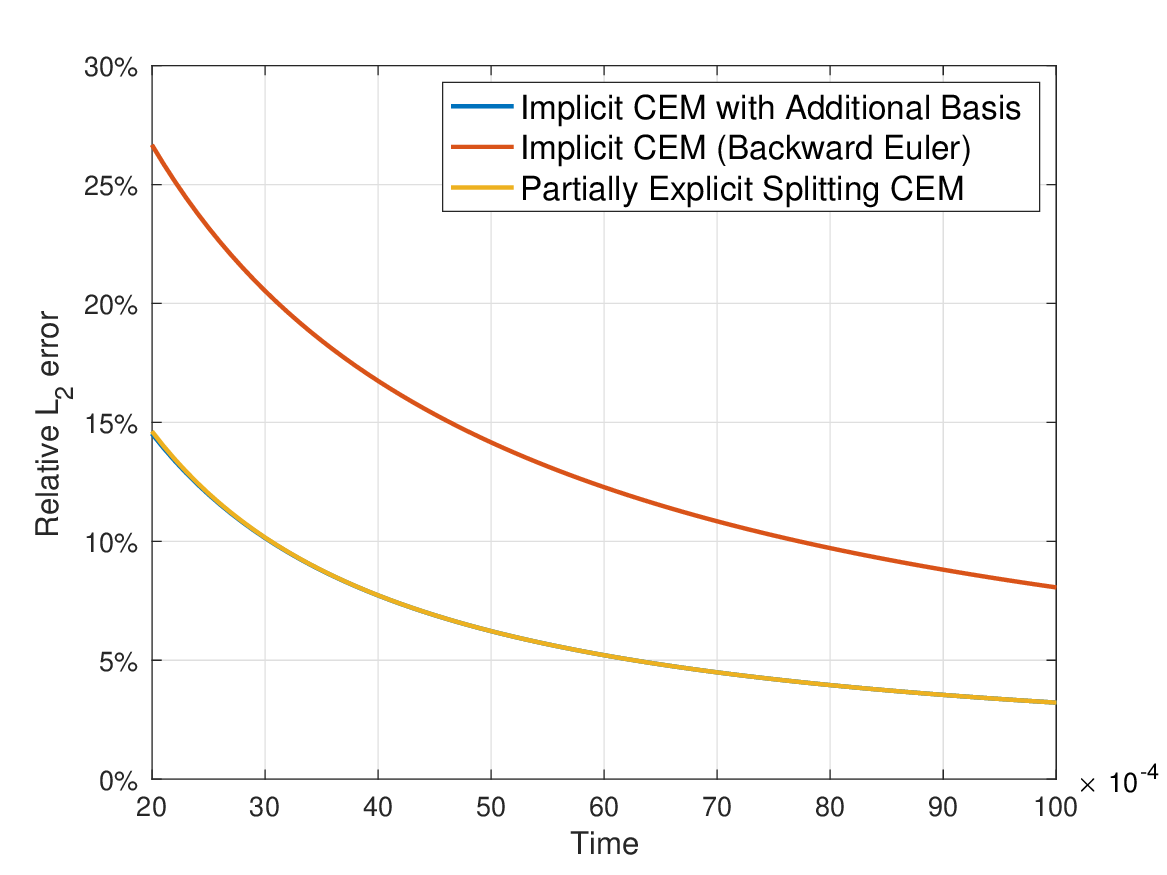} \quad
\includegraphics[width=3.05in]{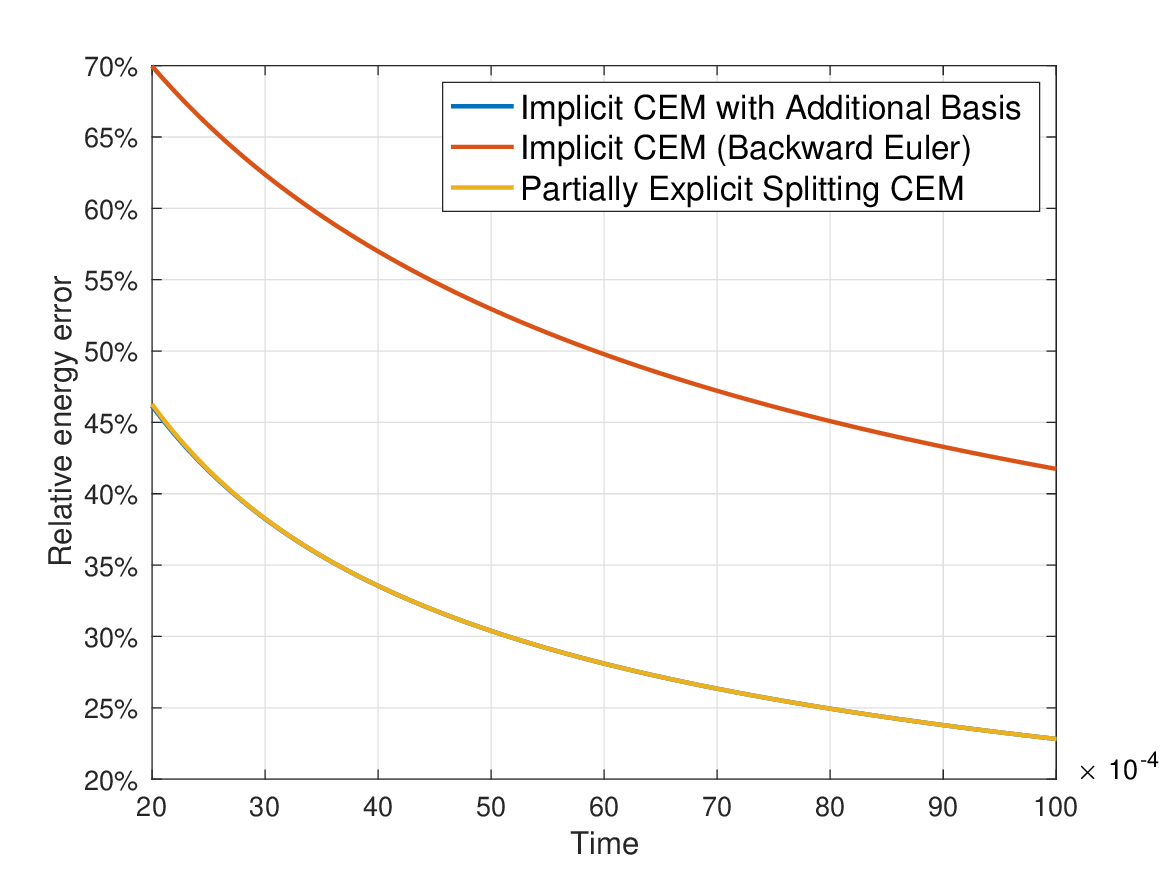}
\caption{Example \ref{exp:1} with source function $f_2$. Left: $L_2$ error against time  Right: Energy error against time}
\label{fig:error_partial_2}
\end{figure}

\begin{table}[H]
\centering
\resizebox{\textwidth}{!}{
\begin{tabular}{c||c|c|c}
Time Step $n$& Implicit CEM-GMsFEM with $Q_{H,2}$  & Implicit CEM-GMsFEM & Partially explicit method \\ 
\hline
   $ 1   $ &  $ 91.31 \% $  & $ 96.82 \% $  & $ 100.76 \% $  \cr
   $ 21   $ &  $ 44.15 \% $  & $ 68.19 \% $  & $ 44.23 \% $  \cr
   $ 41   $ &  $ 32.82 \% $  & $ 56.09 \% $  & $ 32.82\% $  \cr
   $ 61   $ &  $ 27.71 \% $  & $ 49.22 \% $  & $ 27.71 \% $  \cr
   $ 81   $ &  $ 24.69 \% $  & $ 44.70 \% $  & $ 24.69 \% $  \cr
   $ 100  $ &  $ 22.73 \% $  & $ 41.60 \% $  & $ 22.73\% $ 
\end{tabular}
}
\caption{Energy error $e_{energy}^n$ of pressure $p$ against time step for Example \ref{exp:1} with source function $f_2$. }
\label{exp:5_1_1energy}
\end{table}

\begin{table}[H]
\centering
\resizebox{\textwidth}{!}{
\begin{tabular}{c||c|c|c}
Time Step $n$& Implicit CEM-GMsFEM with $Q_{H,2}$ & Implicit CEM-GMsFEM & Partially explicit method \\ 
\hline
   $ 1   $ &  $53.73 \% $  & $ 64.84 \% $  & $ 70.03 \% $  \cr
   $ 21   $ &  $ 13.39 \% $  & $ 25.14 \% $  & $ 13.46 \% $  \cr
   $ 41   $ &  $7.37 \% $  & $ 16.15 \% $  & $ 7.37\% $  \cr
   $ 61   $ &  $ 5.05 \% $  & $ 11.96 \% $  & $ 5.05 \% $  \cr
   $ 81  $ &  $ 3.86 \% $  & $ 9.52 \% $  & $ 3.86 \% $  \cr
   $ 100   $ &  $ 3.19 \% $  & $ 7.99 \% $  & $ 3.19 \% $
\end{tabular}
}
\caption{$L_2$ error  $e_{L_2}^n$ of pressure $p$ against time step for Example \ref{exp:1} with source function $f_2$. }
\label{exp:5_1_1l2}
\end{table}

For the next two examples, we will employ the same format of tables and figures as in Example \ref{exp:1} to present  numerical results.
\begin{example}\label{chap3_exp:2}
The problem setting in this example reads as follows: 
\begin{enumerate}
\item We set $\Omega = (0,1)^2$, $\alpha = 0.9$, $M = 1$, $\nu_p = 0.2$ (Poisson ratio), and $\nu = 1$. 
\item The terminal time is $T = 1/100$. The time step size is $\tau = 10^{-4}$, i.e., $N = 100$. 
\item The Young's modulus $E$ is depicted in Figure \ref{fig:ym_exp3}. We set the permeability to be $\kappa = E$. The permeability field is heterogeneous with high contrast streaks. 
\item The source function is $f(x_1,x_2) \equiv 100 \exp(-800((x_1-0.5)^2+(x_2-0.5)^2))$; the initial condition is $p(x_1, x_2) = 100x_1^2(1-x_1)x_2^2(1-x_2)$. 
\item The number of (local) offline (pressure and displacement) basis functions is $J = J_i^1 = J_i^2 = 2$ and the number of oversampling layers is $\ell = 2$. 
\item The number of (local) $Q_{H,2}$ basis function is $J_i=2$.
\item The coarse mesh is $10 \times 10$ and the overall fine mesh is $100 \times 100$. 
\end{enumerate}
\end{example}

\begin{figure}[htbp!]
\centering
\includegraphics[width = 3 in]{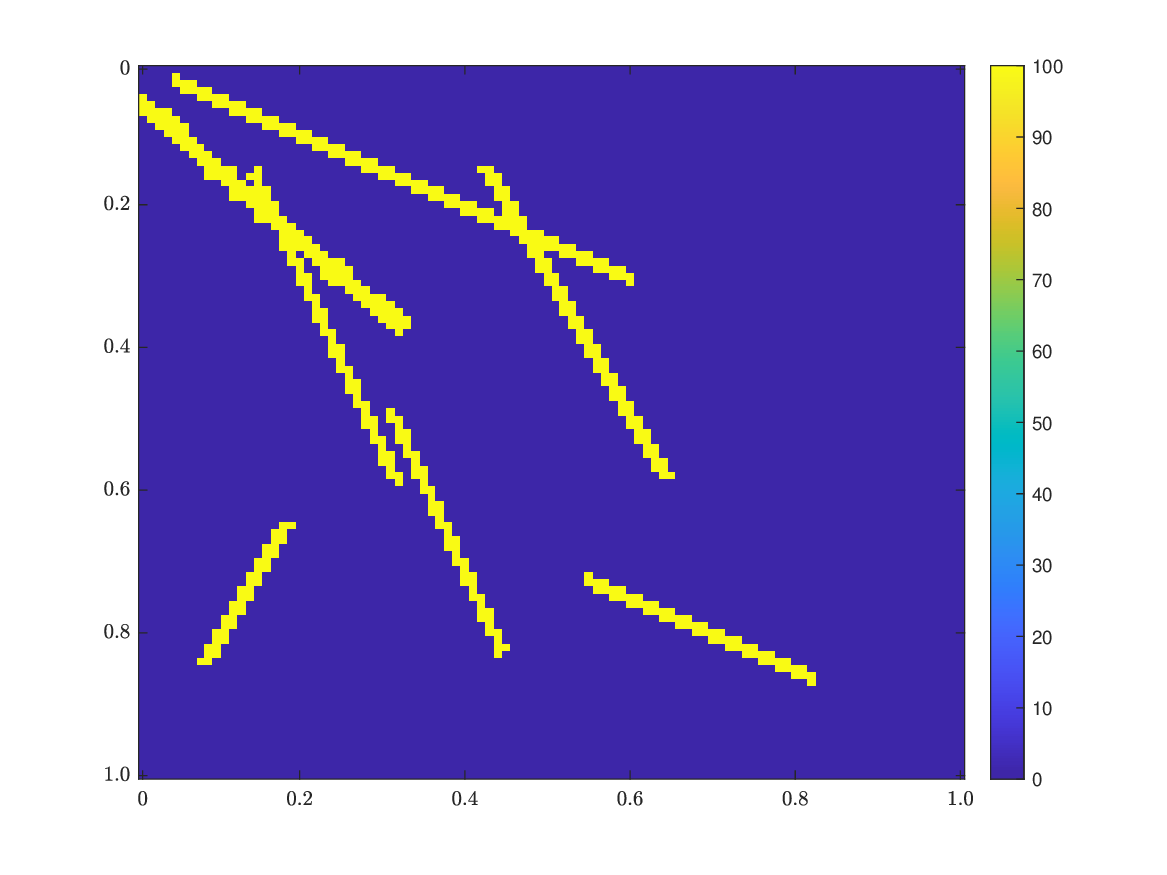}\quad \includegraphics[width = 3 in]{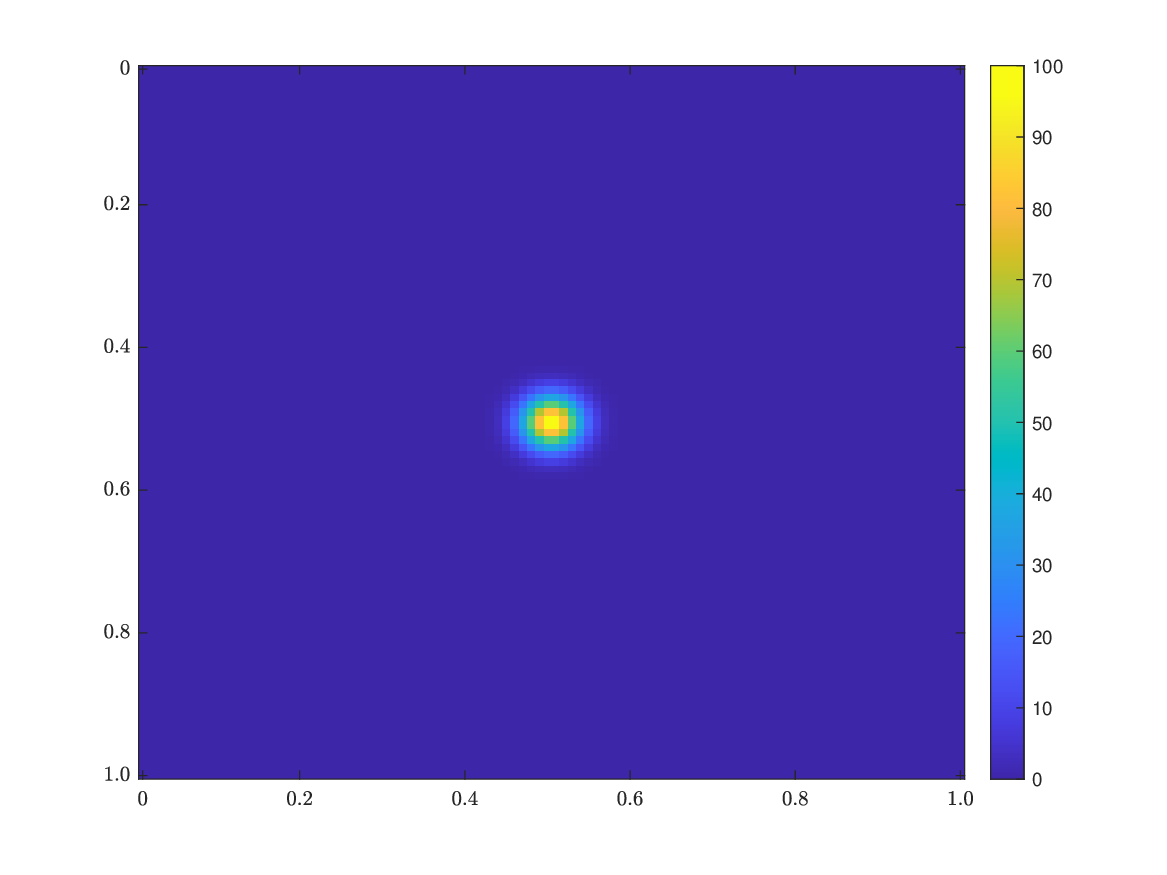}
\caption{Left: Young's Modulus for Example \ref{chap3_exp:2}. Right: Source term $f$} 
\label{fig:ym_exp3}
\end{figure}

\begin{figure}[htbp!]
\centering
\includegraphics[width = 1.7in]{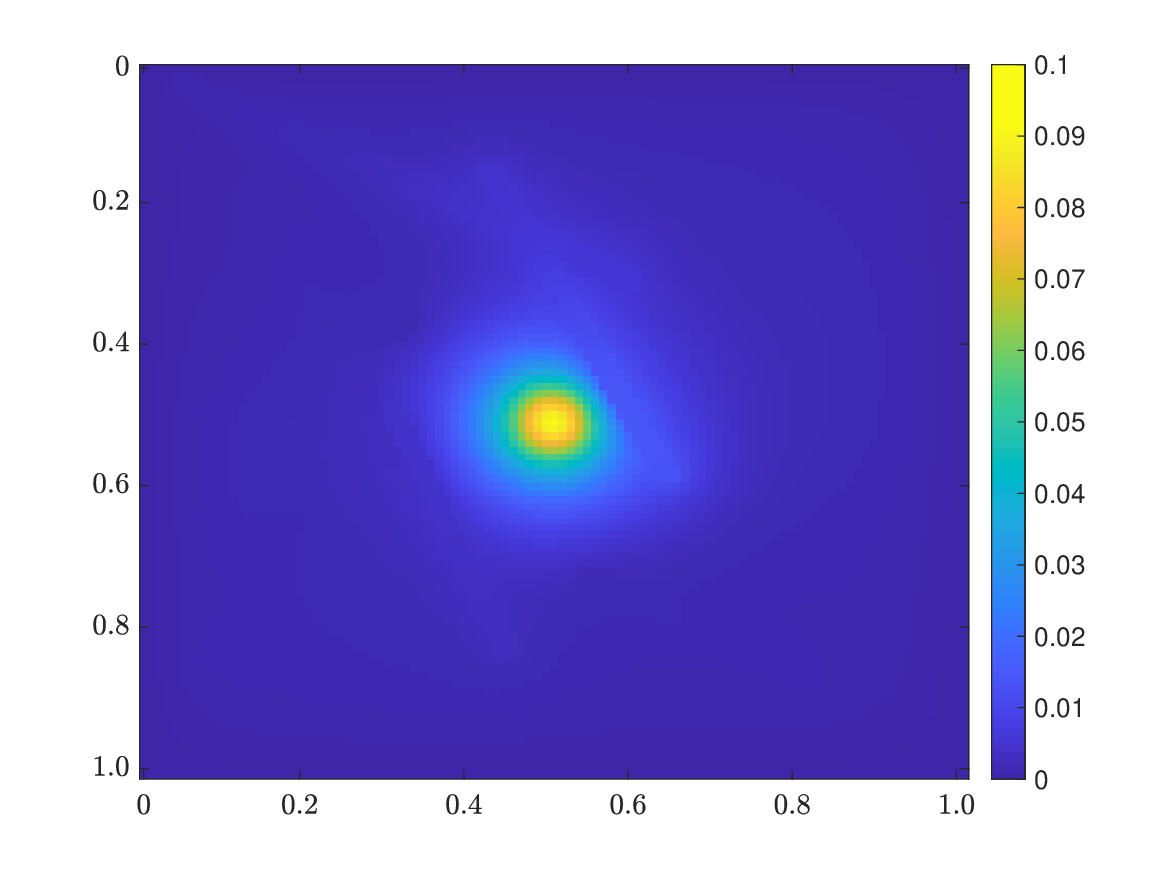} \quad
\includegraphics[width = 1.7in]{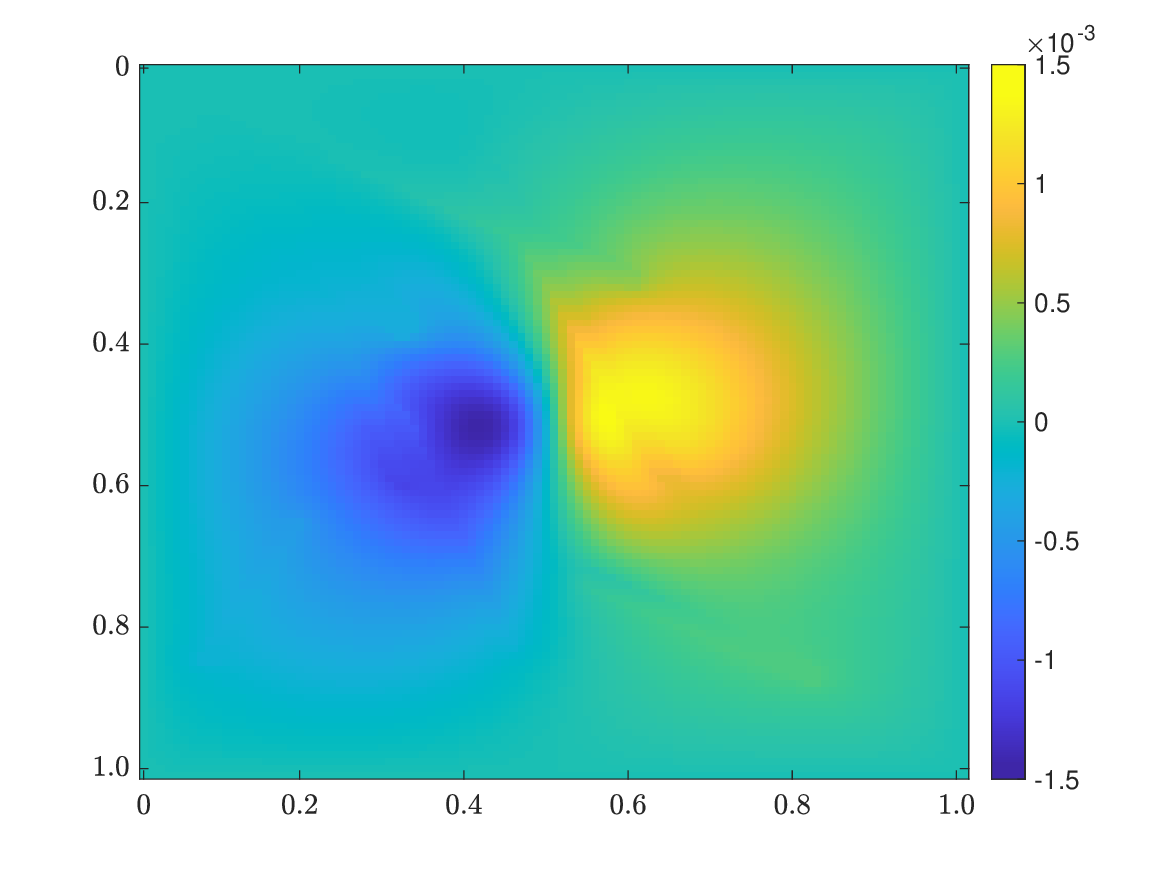}\quad
\includegraphics[width = 1.7in]{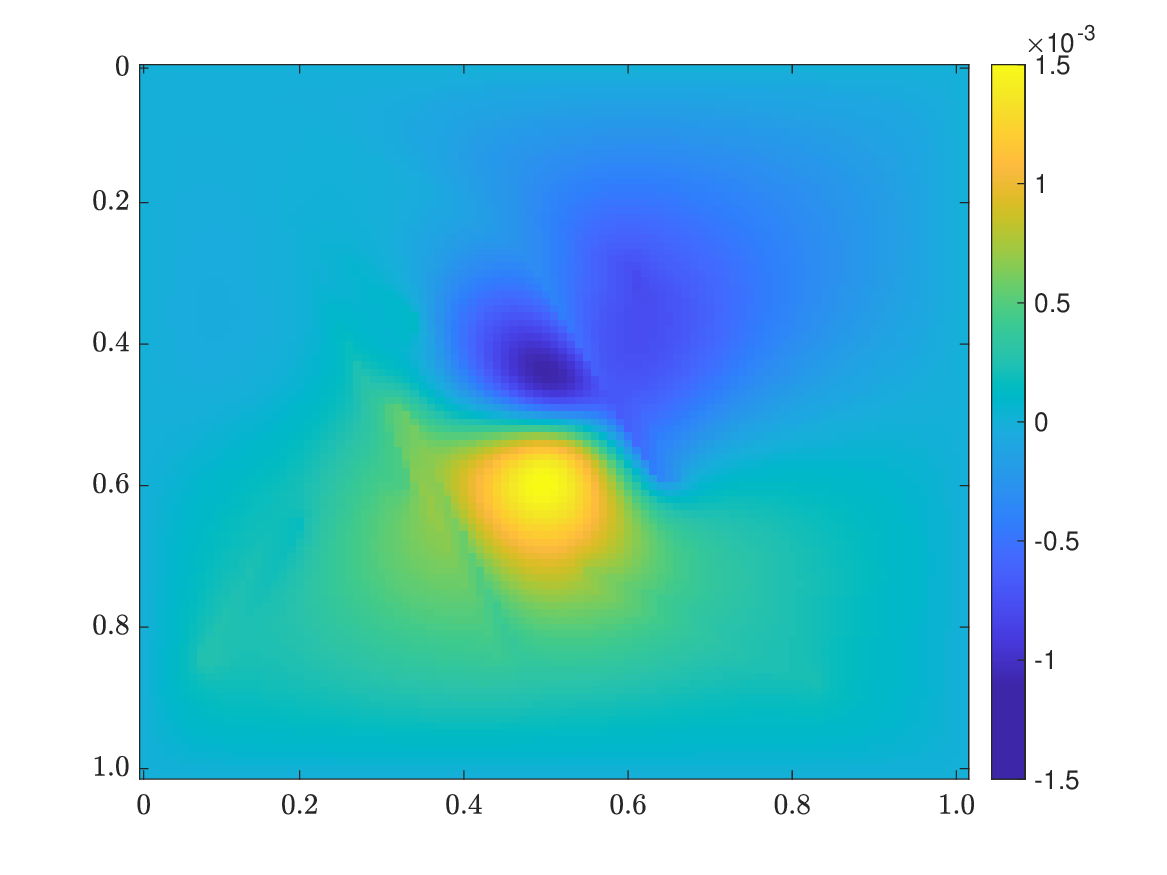}
\includegraphics[width = 1.7in]{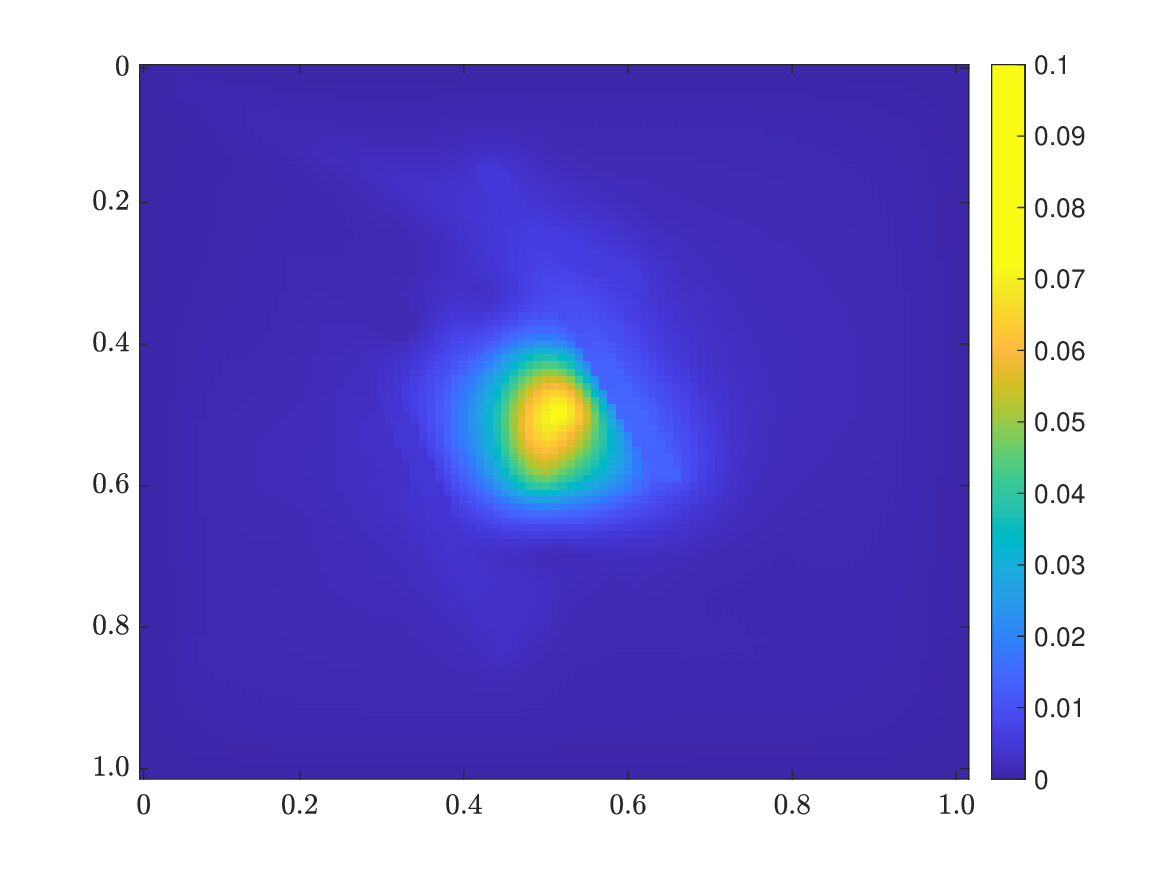}\quad
\includegraphics[width = 1.7in]{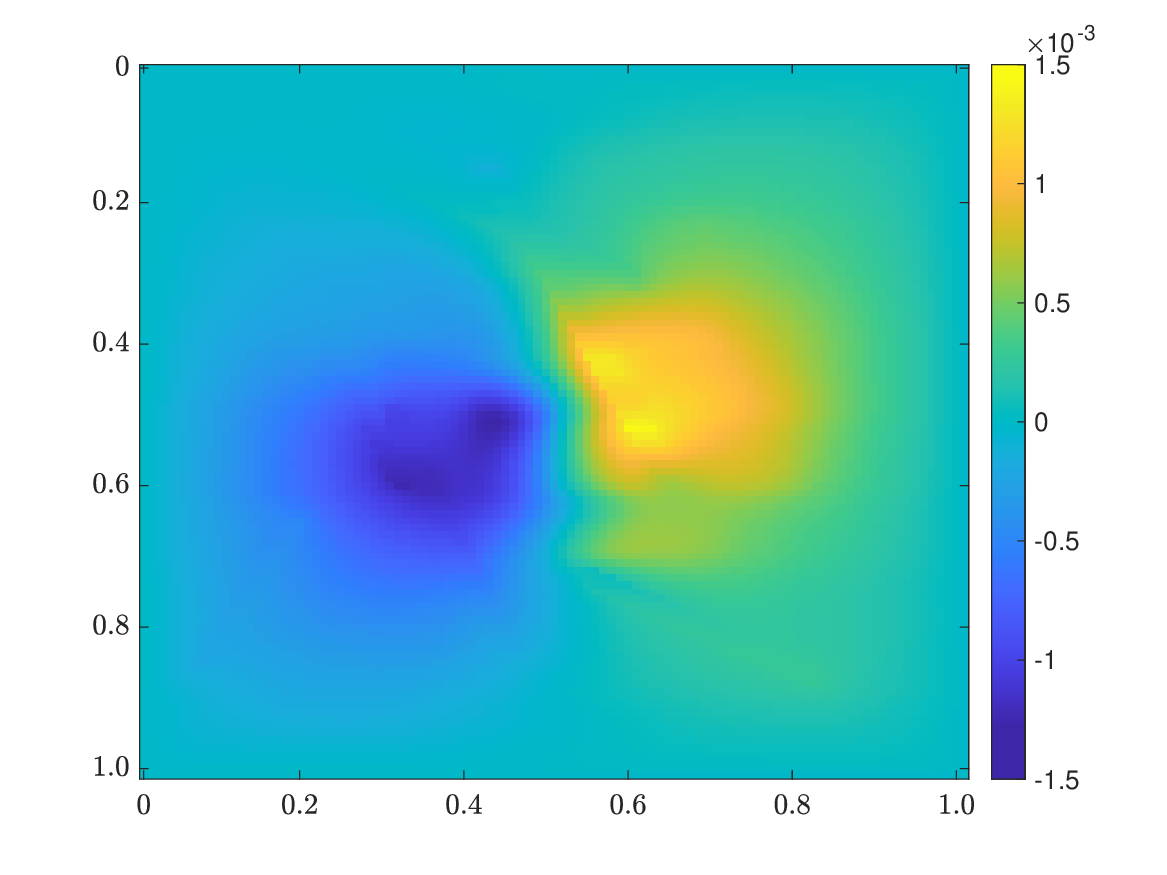}\quad
\includegraphics[width = 1.7in]{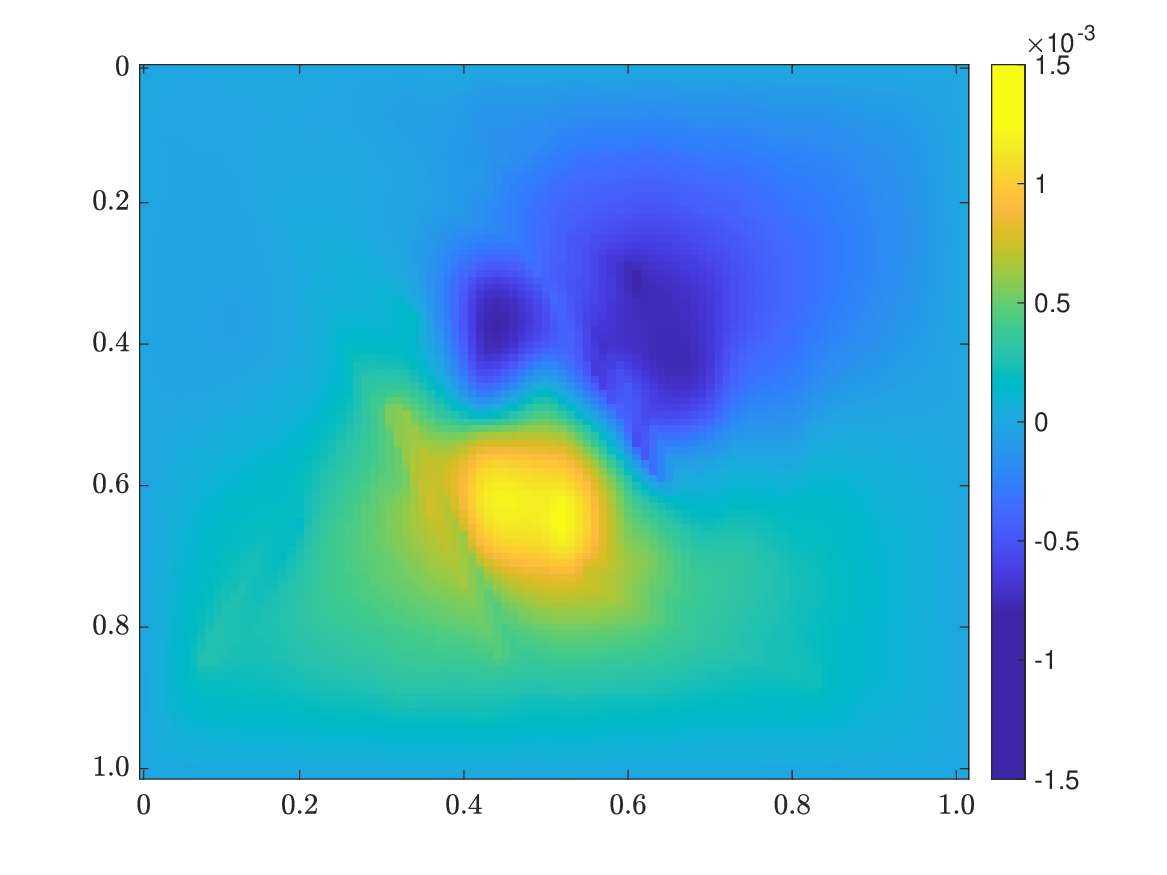}

\includegraphics[width = 1.7in]{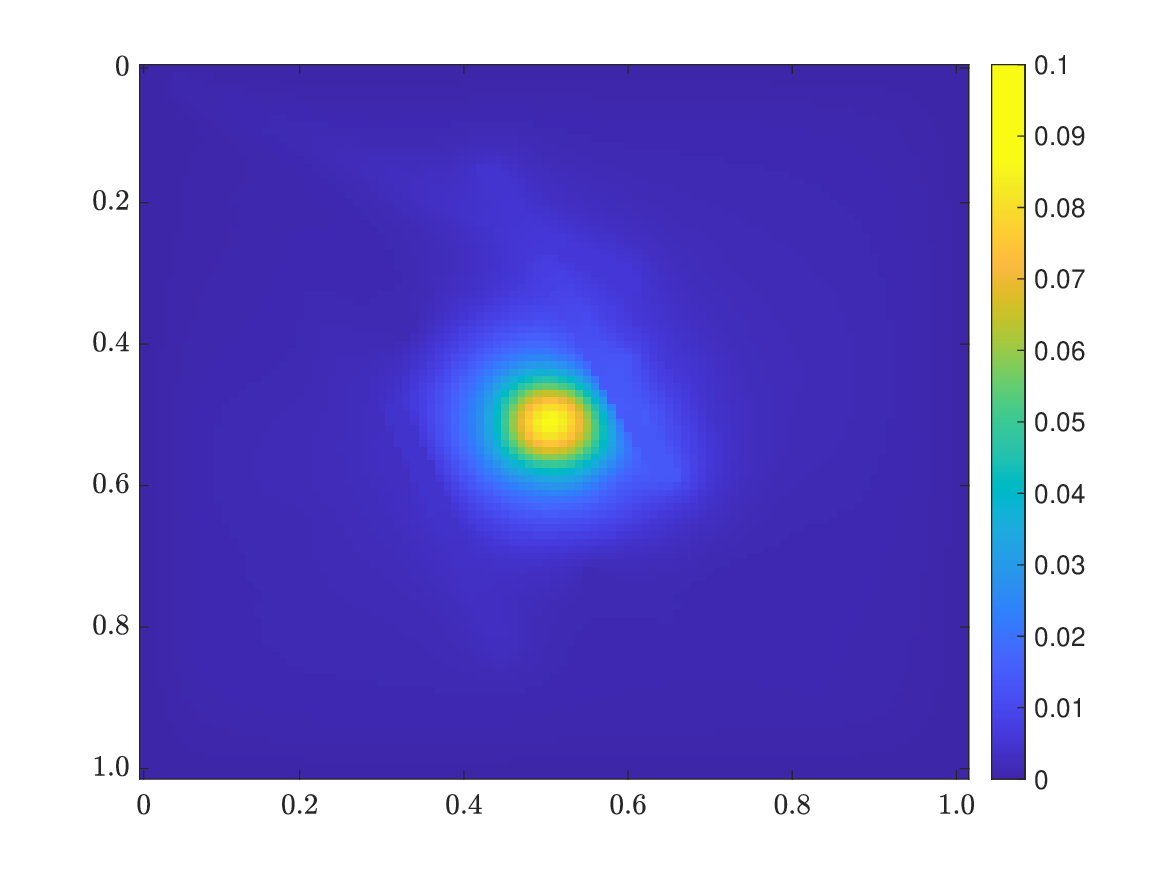}\quad
\includegraphics[width = 1.7in]{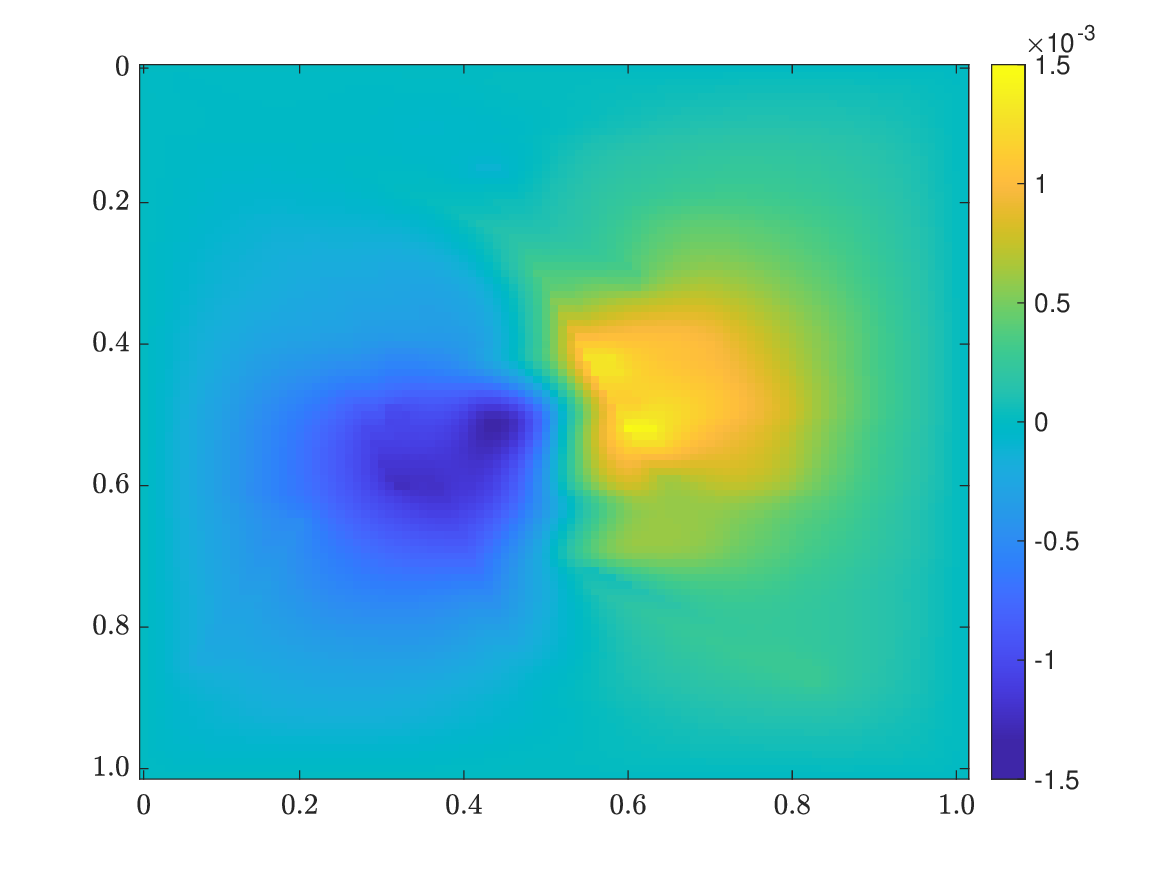}\quad
\includegraphics[width = 1.7in]{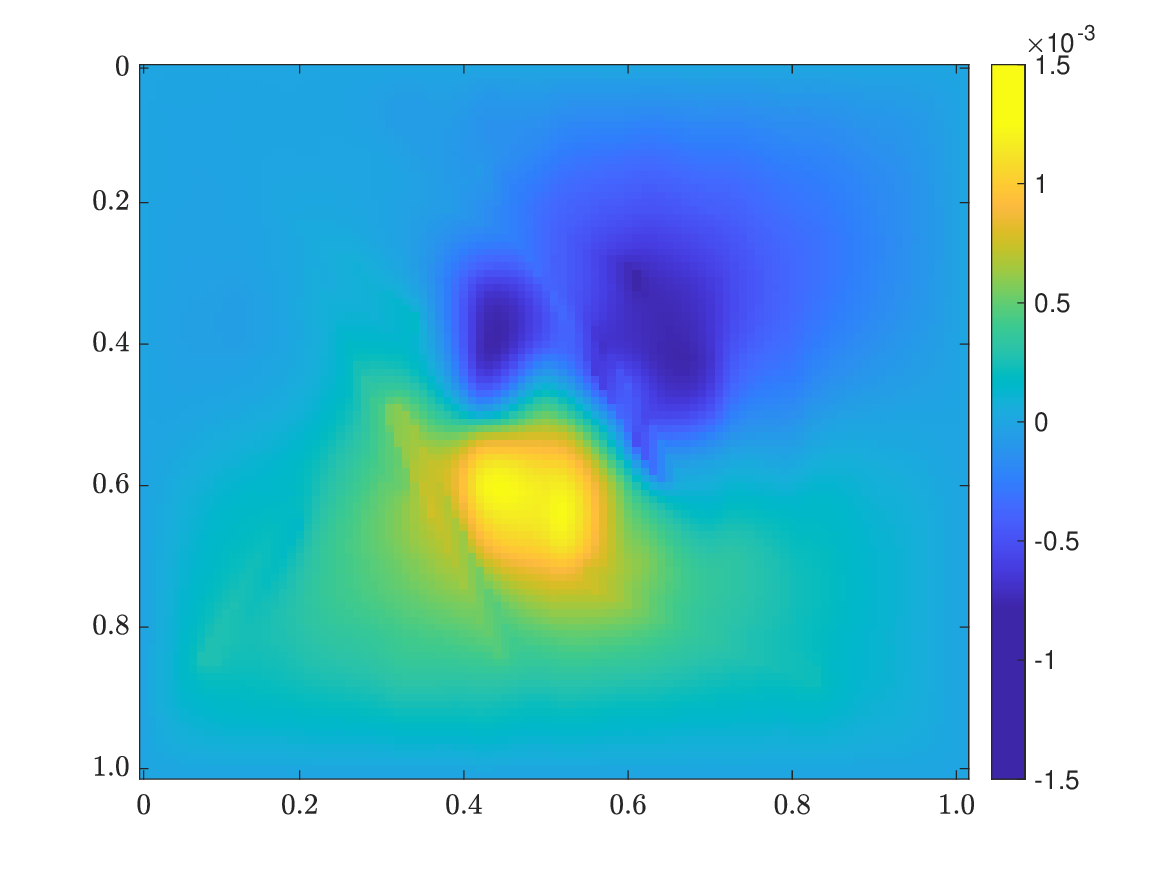}

\includegraphics[width = 1.7in]{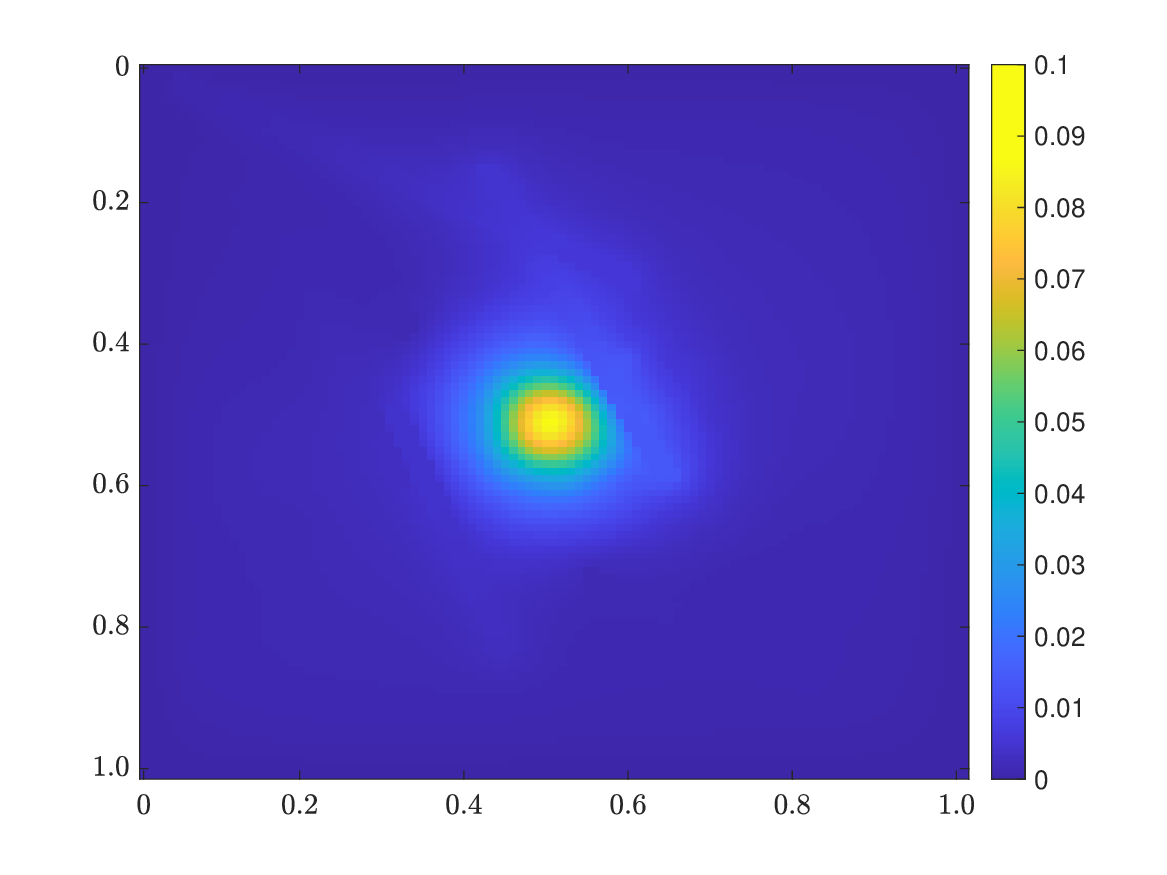}\quad
\includegraphics[width = 1.7in]{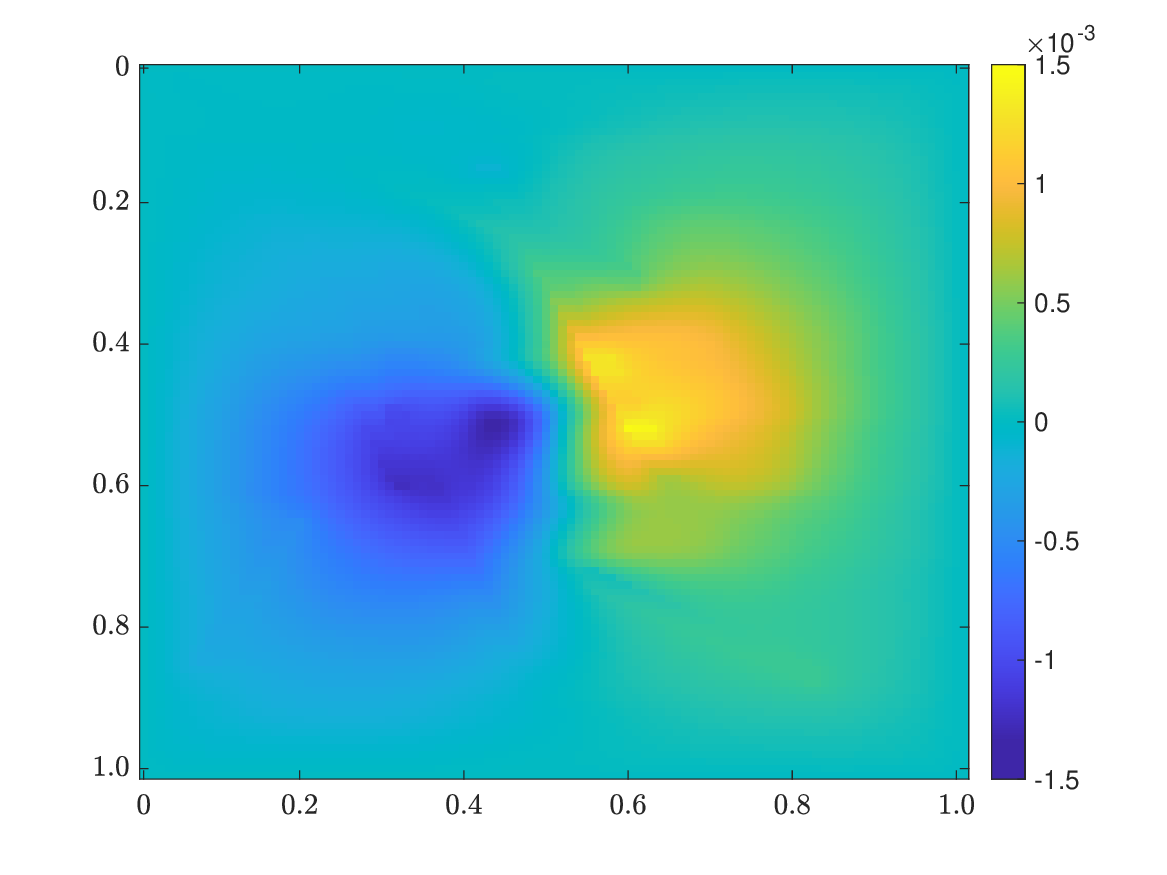}\quad
\includegraphics[width = 1.7in]{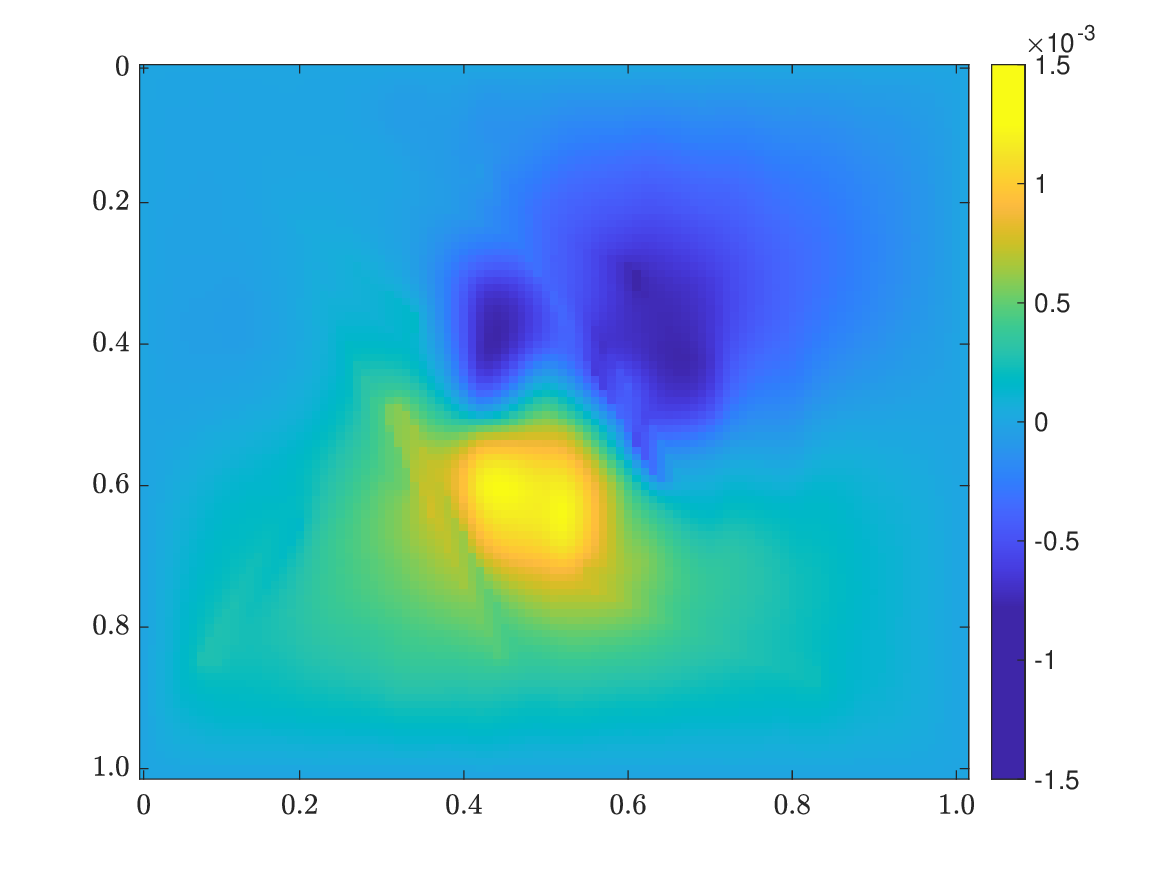}
\caption{Solution profiles for (starting from left to right) $p$, $u_1$, $u_2$ at $T=0.01$ of Example \ref{chap3_exp:2}. Top Row: Reference solutions with fully implicit fine-scale approach. Second Row: Implicit CEM-GMsFEM. Third Row: Implicit CEM-GMsFEM with additional basis $Q_{H,2}$. Bottom Row: Proposed splitting method with additional basis $Q_{H,2}$ .}
\label{fig:5_2_1_soln_profile_partial}
\end{figure}

\begin{figure}[H]
\centering
\includegraphics[width=3.05 in]{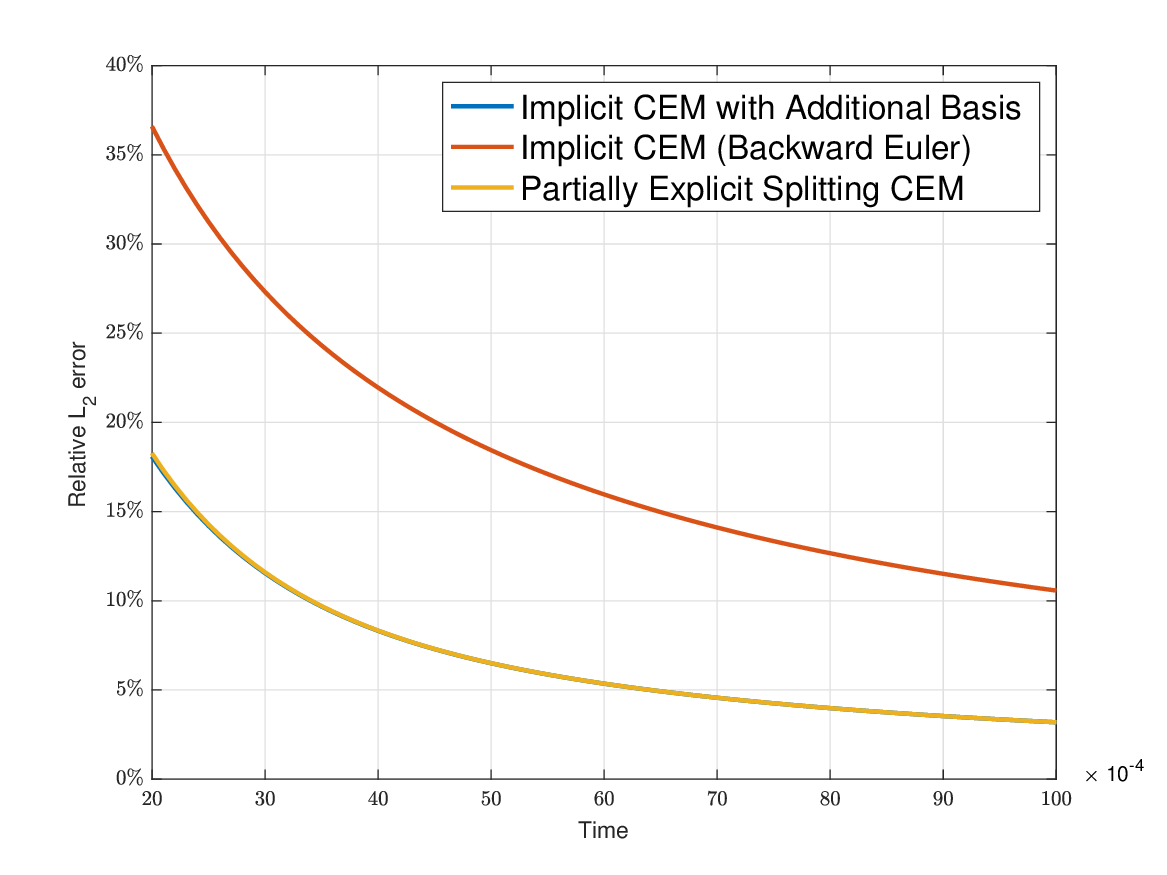} \quad
\includegraphics[width=3.05in]{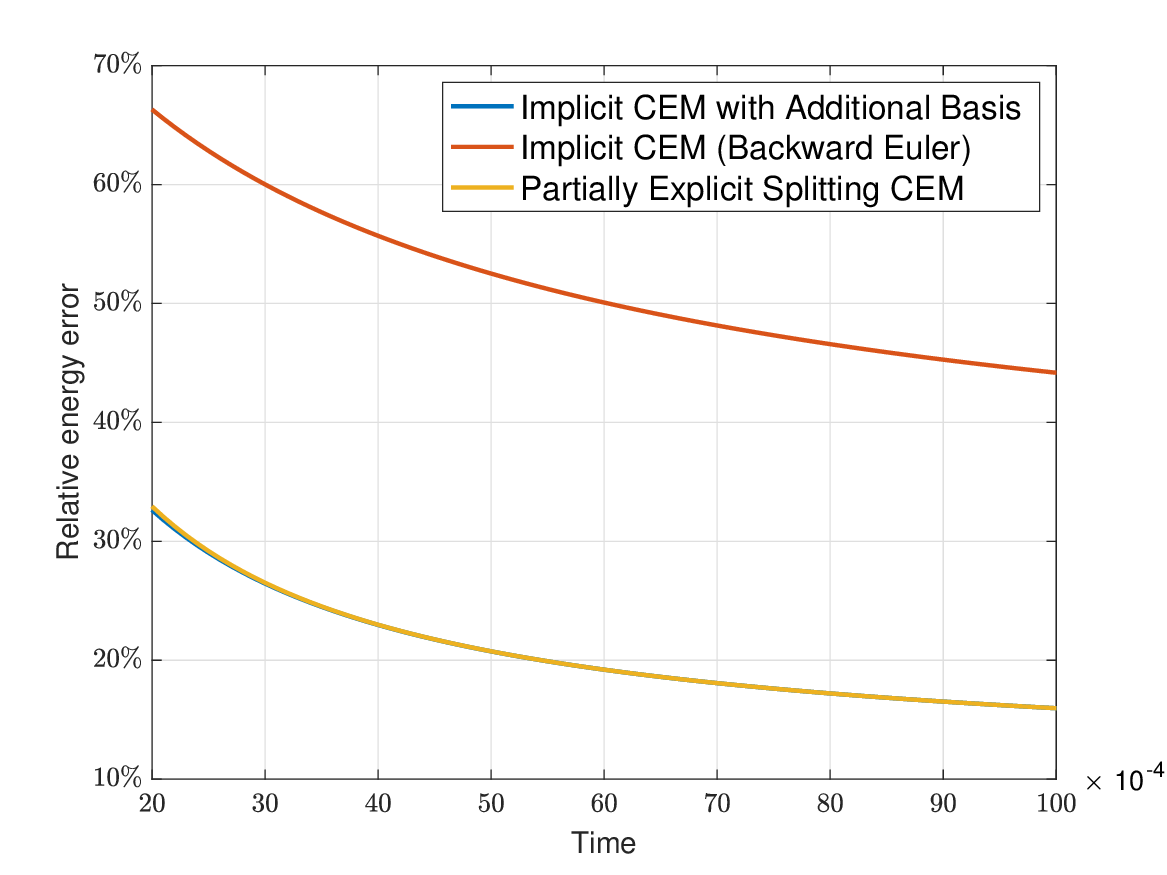}
\caption{Left: $L_2$ error against time  Right: Energy error against time}
\label{fig:error2_partial}
\end{figure}

\begin{table}[H]
\centering
\resizebox{\textwidth}{!}{
\begin{tabular}{c||c|c|c}
Time Step $n$& Implicit CEM-GMsFEM with $Q_{H,2}$  & Implicit CEM-GMsFEM & Partially explicit method \\ 
\hline
   $ 1   $ &  $ 85.57 \% $  & $ 99.81 \% $  & $ 102.23 \% $  \cr
   $ 21   $ &  $ 31.03 \% $  & $ 64.82 \% $  & $ 31.26 \% $  \cr
   $ 41  $ &  $ 22.45 \% $  & $ 54.98 \% $  & $ 22.46 \% $  \cr
   $ 61   $ &  $ 18.95 \% $  & $ 49.66 \% $  & $ 18.95 \% $  \cr
   $ 81   $ &  $ 17.06 \% $  & $ 46.30 \% $  & $ 17.06 \% $  \cr
   $100   $ &  $ 15.91 \% $  & $ 44.07 \% $  & $ 15.91 \% $ 
 \end{tabular}
 }
\caption{Energy error $e_{energy}^n$ of pressure $p$ against time step. }
\label{tabel:DG_error_example3}
\end{table}

  \begin{table}[H]
\centering
\resizebox{\textwidth}{!}{
\begin{tabular}{c||c|c|c}
Time Step $n$& Implicit CEM-GMsFEM with $Q_{H,2}$ & Implicit CEM-GMsFEM & Partially explicit method \\ 
\hline
   $ 1   $ &  $ 74.79 \% $  & $ 99.73 \% $  & $ 100.79 \% $  \cr
   $ 21  $ &  $ 16.34 \% $  & $ 34.23 \% $  & $ 16.47 \% $  \cr
   $ 41   $ &  $ 7.87 \% $  & $ 21.13 \% $  & $ 7.88 \% $  \cr
   $ 61   $ &  $ 5.17 \% $  & $ 15.55 \% $  & $ 5.18 \% $  \cr
   $ 81   $ &  $ 3.88 \% $  & $ 12.41 \% $  & $ 3.88 \% $  \cr
   $ 100   $ &  $ 3.16 \% $  & $ 10.49 \% $  & $ 3.16 \% $ 
\end{tabular}
}
\caption{$L_2$ error  $e_{L_2}^n$ of pressure $p$ against time step. }
\label{tabel:L2_error_example3}
\end{table}

As shown in Figure \ref{fig:error2_partial}, Table \ref{tabel:DG_error_example3} and Table \ref{tabel:L2_error_example3}, with the singular source term in this example, CEM with additional basis functions give notable improvements. Compared to the blue curve (CEM without additional basis functions), the other two curves coincide and are much lower than the blue curve, especially for the energy error. The reason for this is the original multiscale CEM basis functions do not take singular source term into account, while the additional basis function can significantly correct the solution in cases with a singular source term.

\begin{example}\label{chap3_exp:3}
We consider a time-dependent source function in this example. Specifically, we set the source function to be  $ f(t, x_1, x_2)=100\exp(-800((x_1-0.5)^2+(x_2-0.5)^2))\exp(-(100 t-1)^2)$. The value of the source term increases as time propagates. The rest of the settings are the same as Example \ref{chap3_exp:2}. It demonstrates the capability of the partially explicit splitting method for the case with a time-dependent source term. 
\end{example}

\begin{figure}[htbp!]
\centering
\includegraphics[width = 1.7in]{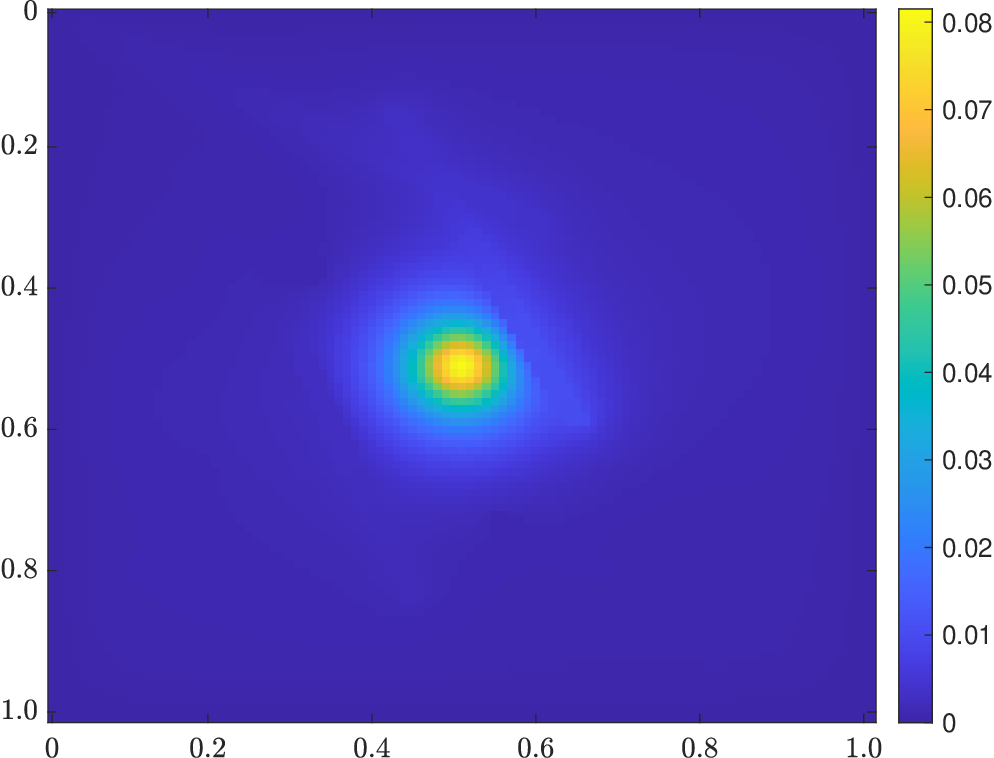} \quad
\includegraphics[width = 1.7in]{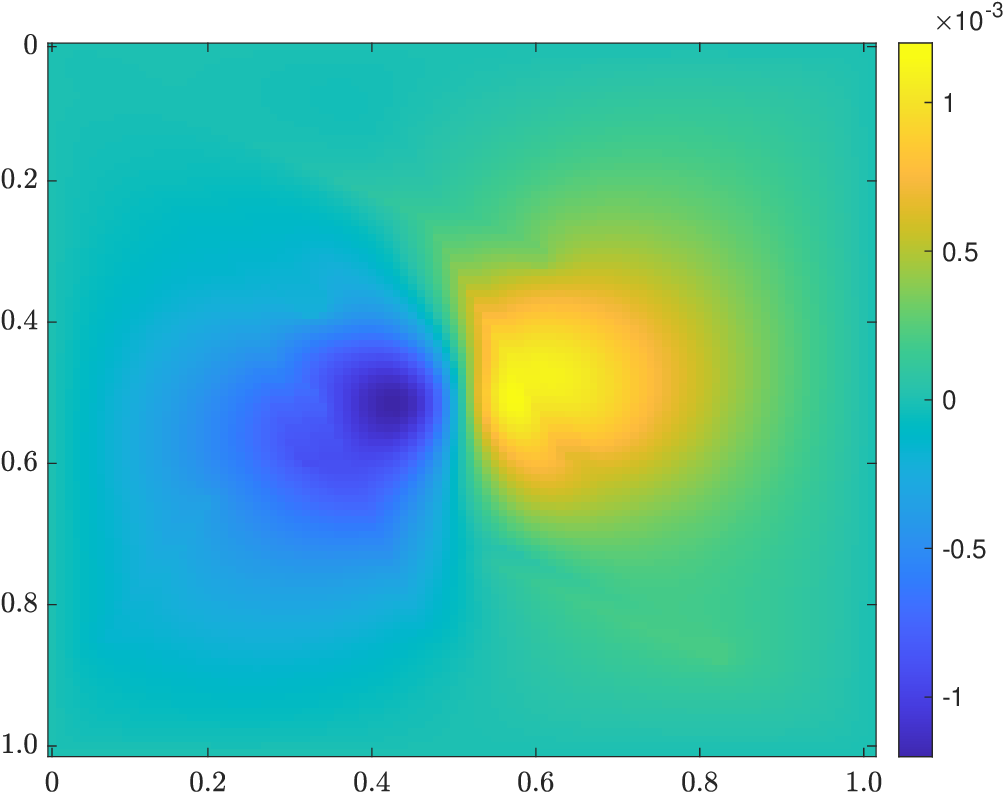}\quad
\includegraphics[width = 1.7in]{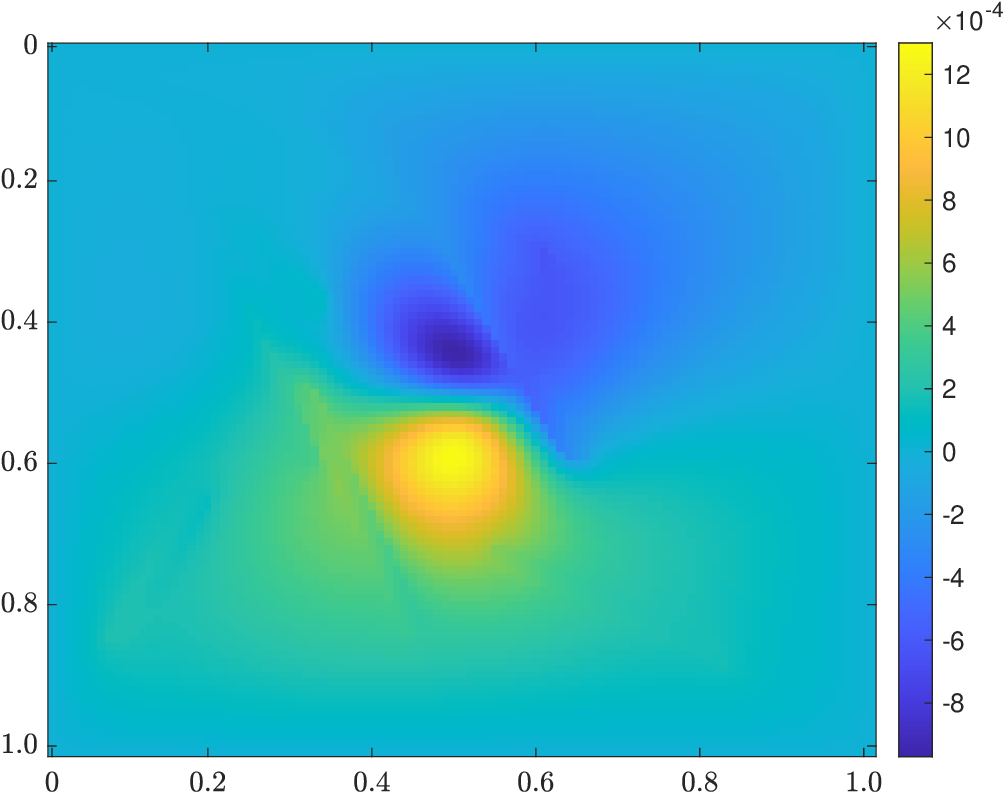}
\includegraphics[width = 1.7in]{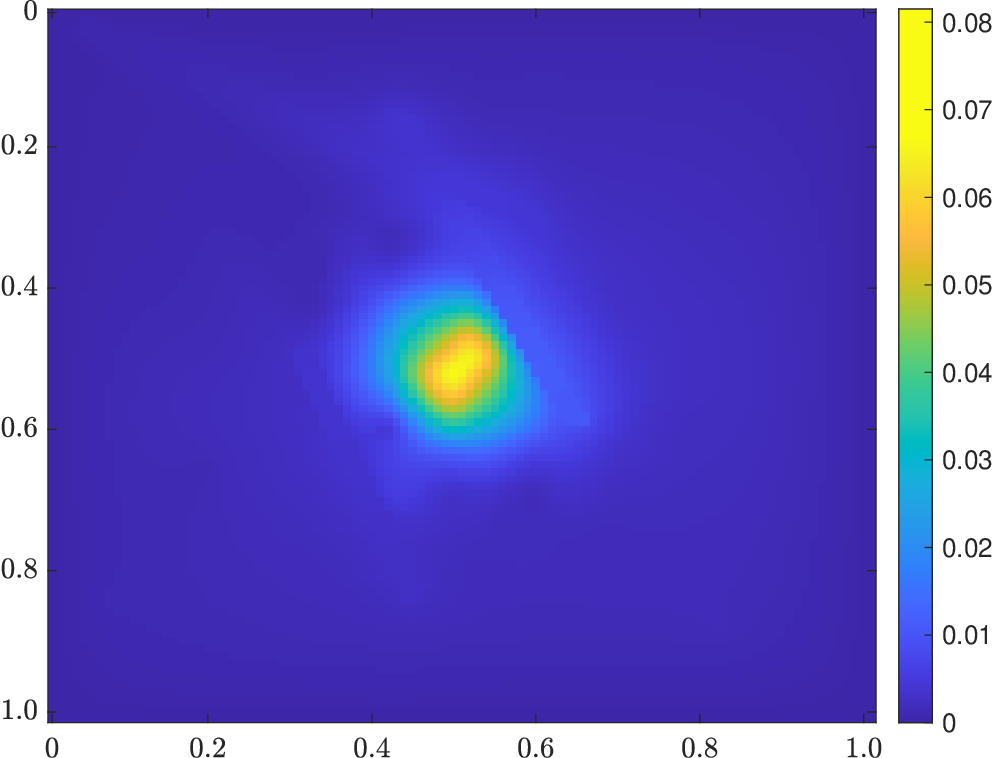}\quad
\includegraphics[width = 1.7in]{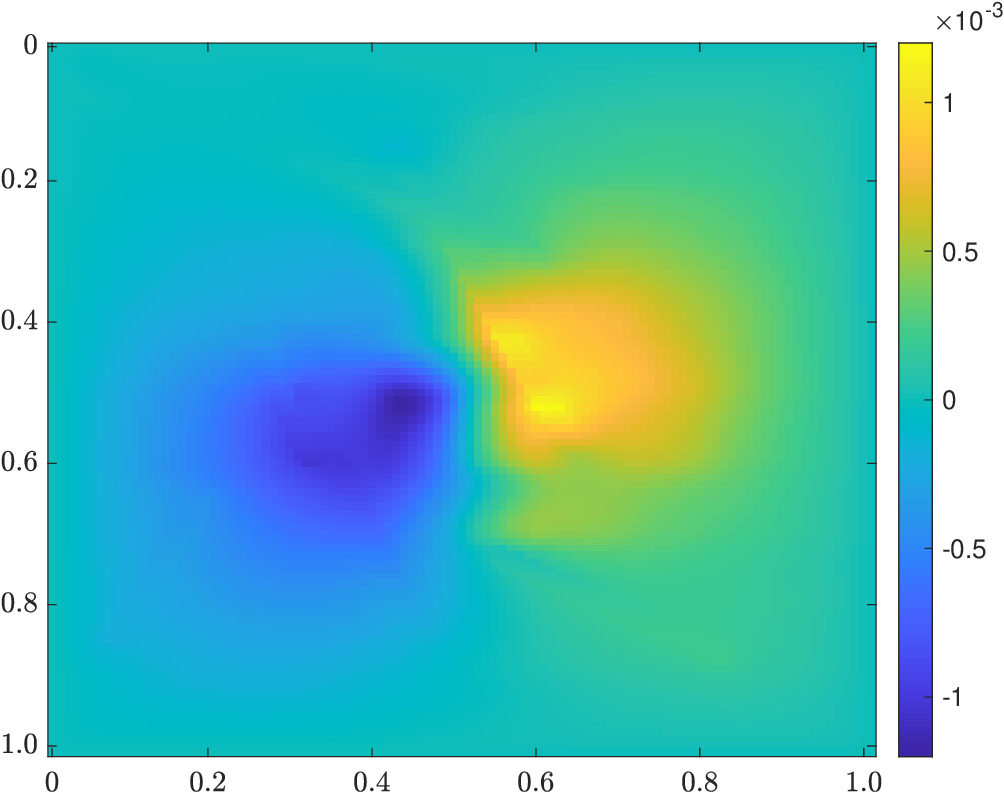}\quad
\includegraphics[width = 1.7in]{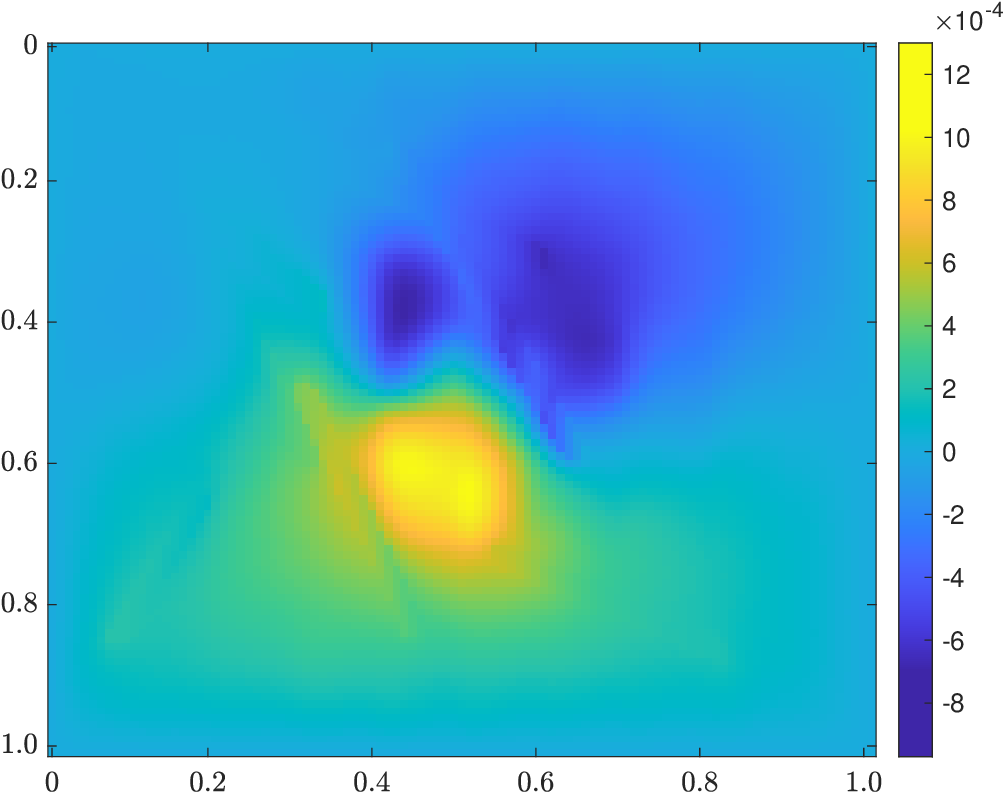}

\includegraphics[width = 1.7in]{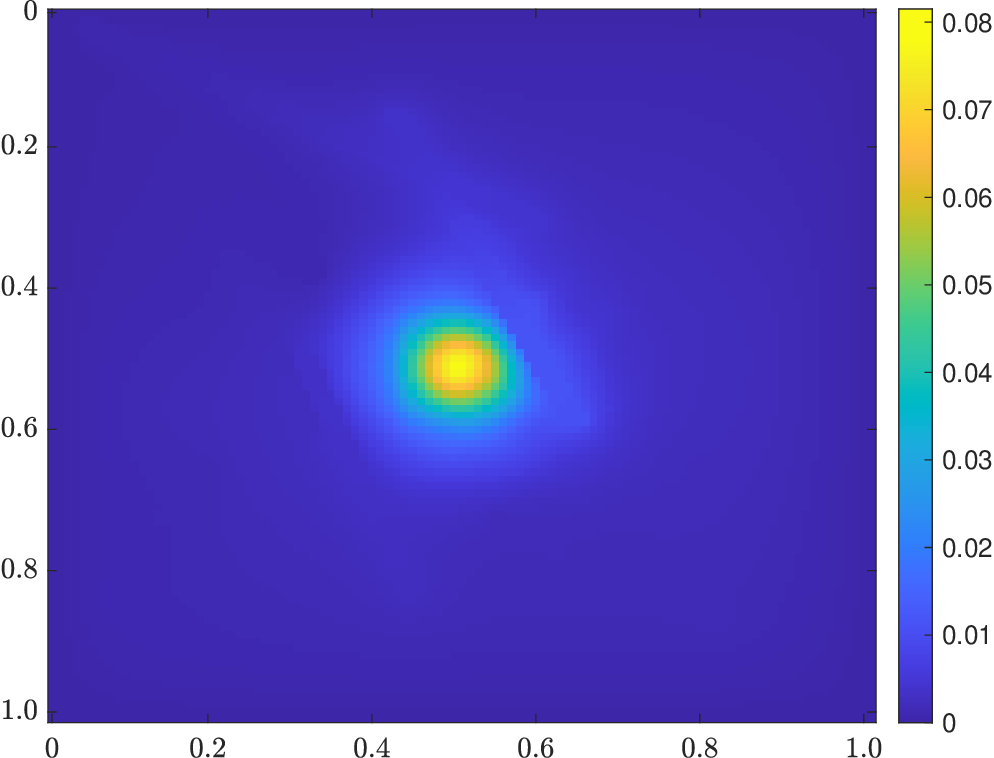}\quad
\includegraphics[width = 1.7in]{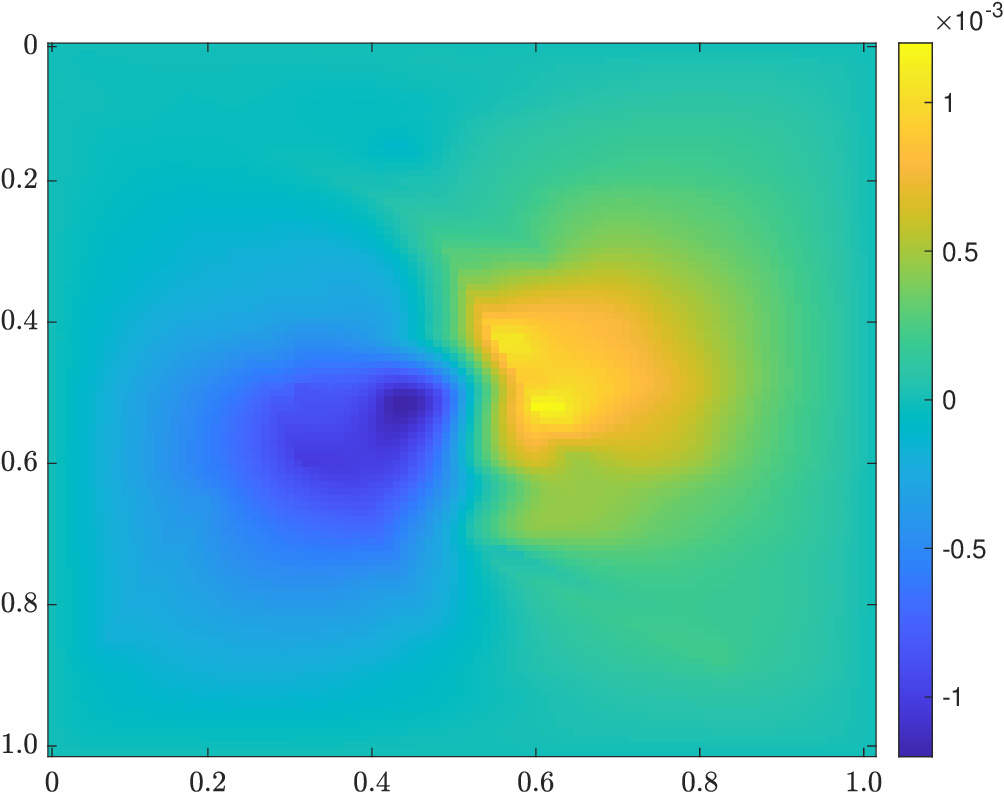}\quad
\includegraphics[width = 1.7in]{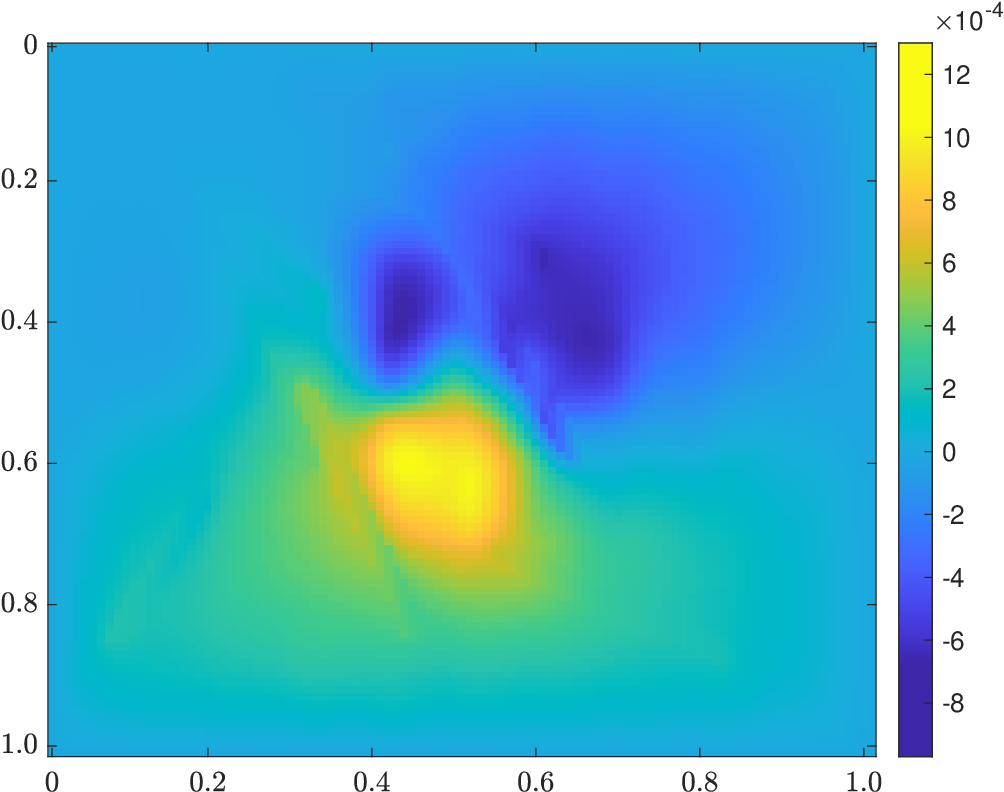}

\includegraphics[width = 1.7in]{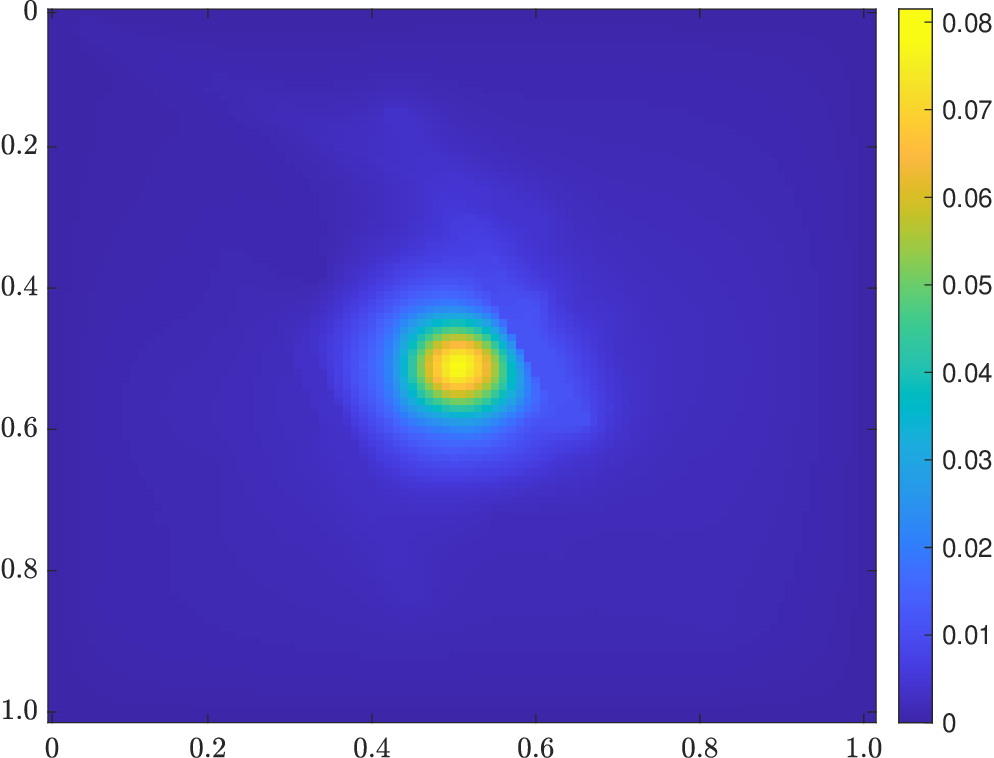}\quad
\includegraphics[width = 1.7in]{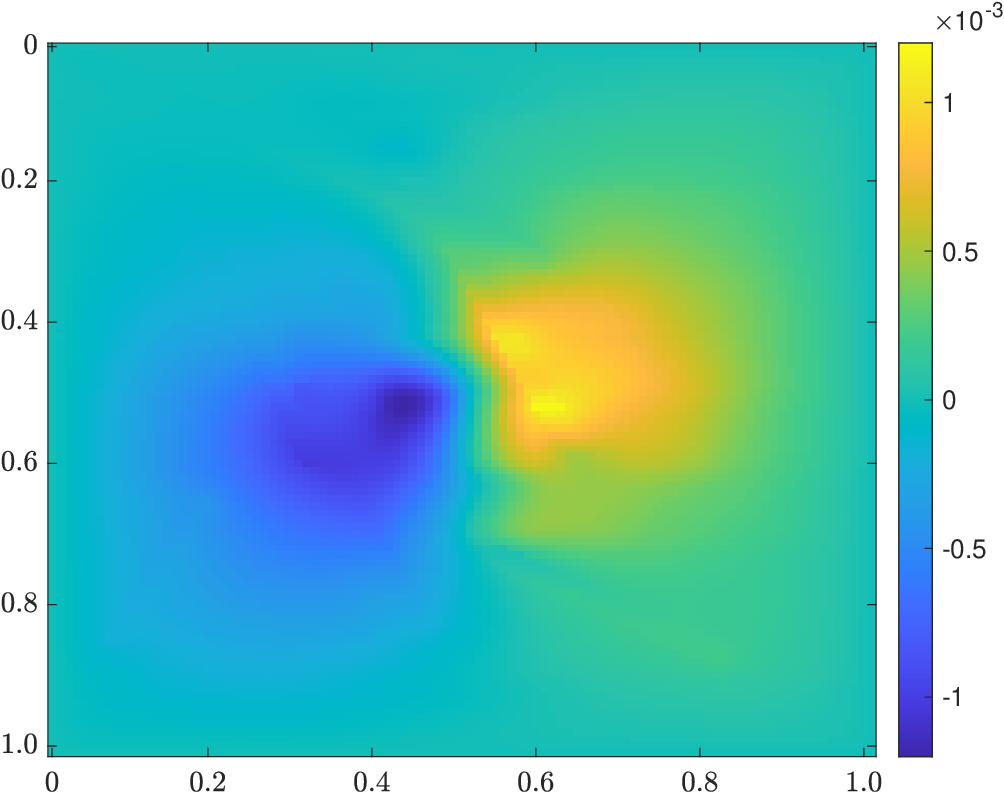}\quad
\includegraphics[width = 1.7in]{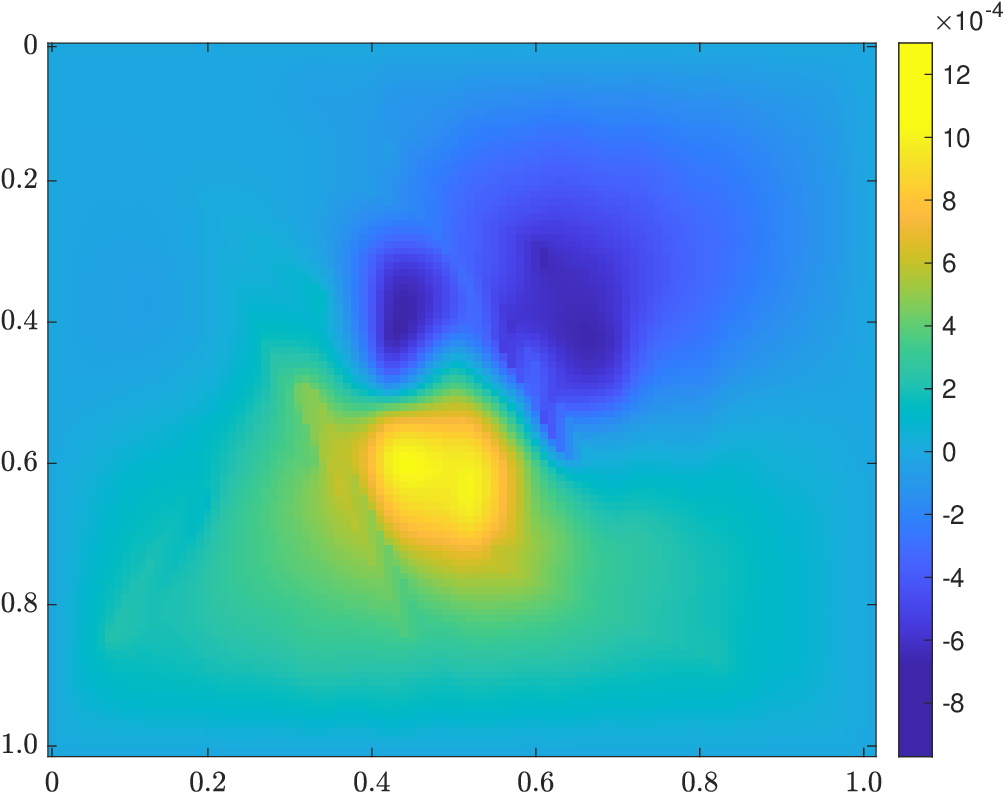}
\caption{Solution profiles for (starting from left to right) $p$, $u_1$, $u_2$ at $T=0.01$ of Example \ref{chap3_exp:3}. Top Row: Reference solutions with fully implicit fine-scale approach. Second Row: Implicit CEM-GMsFEM. Third Row: Implicit CEM-GMsFEM with additional basis $Q_{H,2}$. Bottom Row: Proposed splitting method with additional basis $Q_{H,2}$ .}
\label{fig:5_2_soln_profile_partial}
\end{figure}

\begin{figure}[H]
\centering
\includegraphics[width=3.05 in]{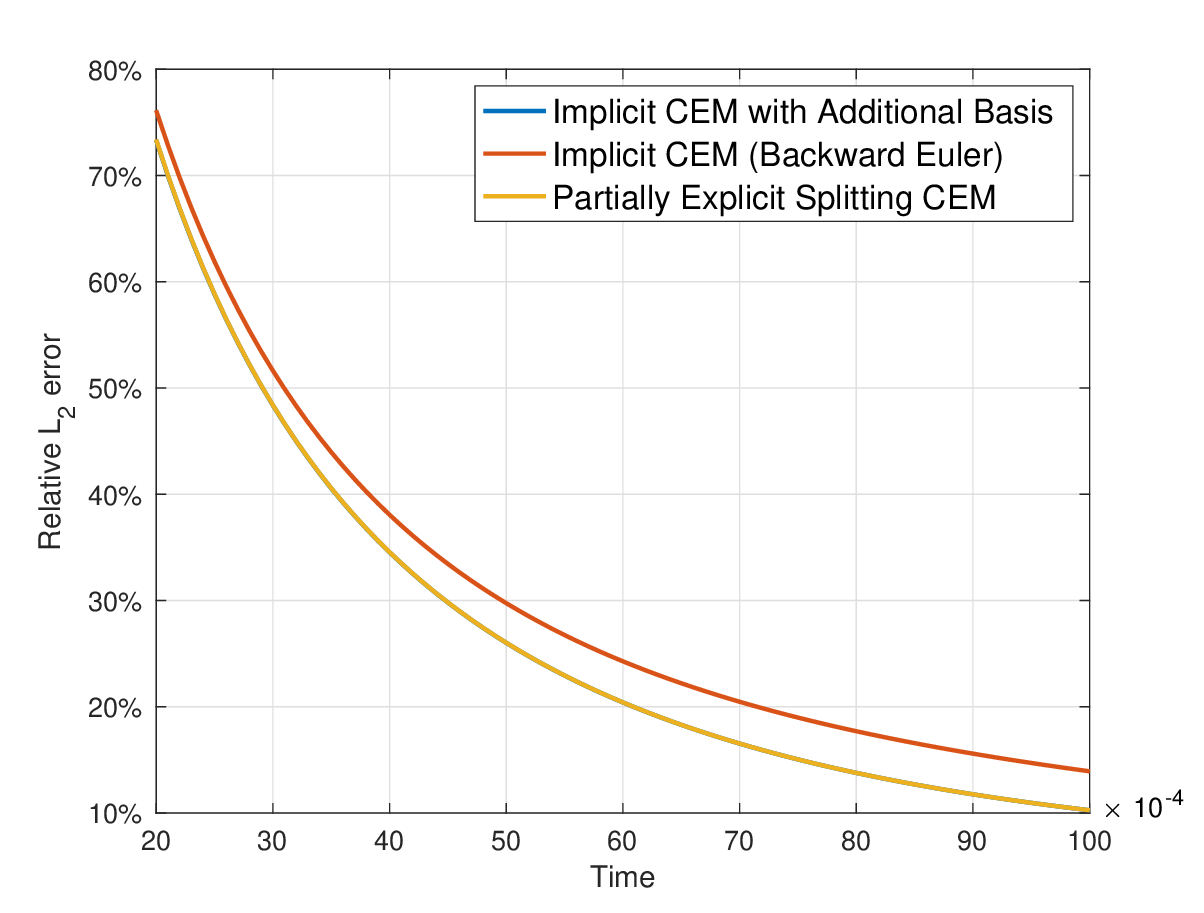} \quad
\includegraphics[width=3.05in]{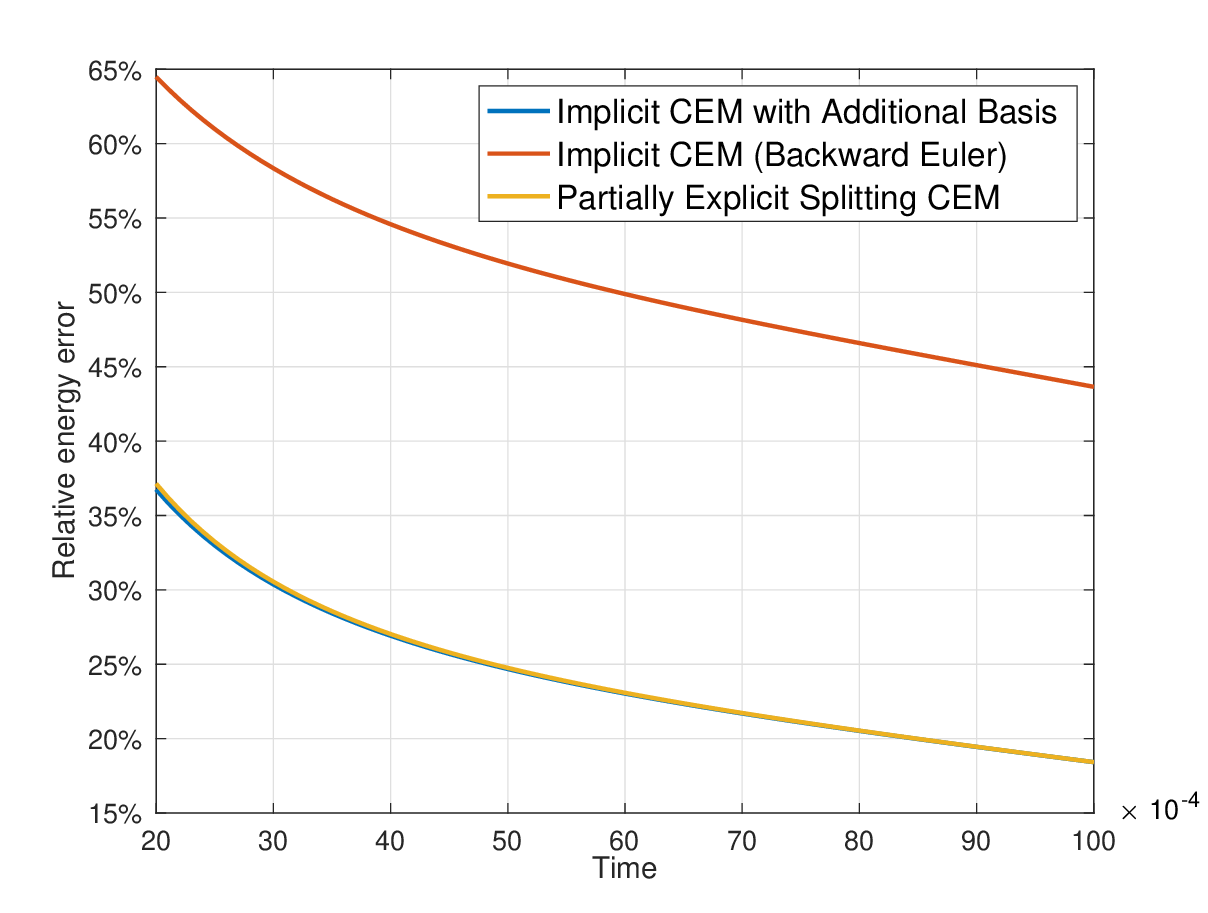}
\caption{Left: $L_2$ error against time.  Right: Energy error against time.}
\label{fig:error4_partial}
\end{figure}

In Figure \ref{fig:error4_partial}, we observe that the implicit CEM with $Q_{H,2}$ and the partially explicit method outperform the implicit CEM without $Q_{H,2}$. For the time-dependent source term in this example, we also see a bigger improvement in terms of the energy norm than $L_2$ norm. This is a similar result as Example \ref{chap3_exp:2} with time-independent source term.
\begin{table}[H]
\centering
\resizebox{\textwidth}{!}{
\begin{tabular}{c||c|c|c}
Time Step $n$& Implicit CEM-GMsFEM with $Q_{H,2}$  & Implicit CEM-GMsFEM & Partially explicit method \\ 
\hline
  $ 1   $ &  $ 367.27 \% $  & $ 379.45 \% $  & $ 382.92 \% $  \cr
   $ 21  $ &  $ 35.05 \% $  & $ 62.95 \% $  & $ 35.38 \% $  \cr
   $ 41   $ &  $ 26.40 \% $  & $ 53.98 \% $  & $ 26.50 \% $  \cr
   $ 61   $ &  $ 22.74\% $  & $ 49.52 \% $  & $ 22.79 \% $  \cr
   $ 81   $ &  $ 20.30\% $  & $ 46.29 \% $  & $ 20.32 \% $  \cr
   $ 100   $ &  $ 18.32 \% $  & $ 43.50 \% $  & $ 18.32 \% $ 
 \end{tabular}
 }
\caption{Energy error $e_{energy}^n$ of pressure $p$ against time step.}
\label{tabel:DG_error_example4}
\end{table}

  \begin{table}[H]
\centering
\resizebox{\textwidth}{!}{
\begin{tabular}{c||c|c|c}
Time Step $n$& Implicit CEM-GMsFEM with $Q_{H,2}$ & Implicit CEM-GMsFEM & Partially explicit method \\ 
\hline
 $ 1   $ &  $ 291.60 \% $  & $ 291.31 \% $  & $ 292.43 \% $  \cr
   $ 21   $ &  $ 66.89 \% $  & $ 69.81 \% $  & $ 66.93 \% $  \cr
   $ 41   $ &  $ 32.51 \% $  & $ 36.09 \% $  & $ 32.53 \% $  \cr
   $ 61   $ &  $ 19.52 \% $  & $ 23.41 \% $  & $ 19.52 \% $  \cr
   $ 81   $ &  $ 13.32\% $  & $ 17.23 \% $  & $ 13.32 \% $  \cr
   $ 100   $ &  $ 10.13 \% $  & $ 13.78 \% $  & $ 10.13 \% $
\end{tabular}
}
\caption{$L_2$ error  $e_{L_2}^n$ of pressure $p$ against time step.}
\label{tabel:L2_error_example4}
\end{table}

\section{Conclusions}\label{sec:conclusion}
We investigated a poroelasticity problem that involves multiscale Young's modulus and permeability with high-contrast features. To ensure stability with a time step that is independent of contrast, we propose a partially explicit method that requires special multiscale basis construction and temporal splitting. Our coarse space comprises CEM-GMsFEM basis functions, and we construct special multiscale basis functions for the remaining degrees of freedom. The implicit solution for the coarse space, which has only a few degrees of freedom, is supplemented by a remaining part in the specially designed multiscale space. The remaining part is updated explicitly using the proposed splitting algorithms. We demonstrate that the proposed approach is stable with a time step that is independent of contrast, and the success of the approach is contingent on an appropriate multiscale decomposition of the space, as demonstrated. We present numerical results, which indicate that the proposed partial explicit method is almost as accurate as the fully implicit method.

The partially explicit method for linear poroelasticity problems allows the explicit update for the correction part, which is more computationally efficient compared to the online adaptive enrichment method and avoids the need for iteration. Additionally, the CFL-type condition of the partially explicit method is contrast-independent, which allows larger admissible time step sizes compared to purely explicit methods. However, in practice, we found that the time step size is still relatively small. Referring to the numerical experiments in \cite{chung2021contrast}, we found that the constant in the CFL-type condition is relatively large in the additional basis space we constructed. This explains the reason why the time step size is relatively small. One possible solution is to construct a better space that results in a smaller constant in the CFL-type condition, or to develop a different scheme that relaxes the CFL-type condition.

For further study, based on the results of numerical experiments using the CEM-GMsFEM method to solve linear poroelasticity problems, the displacement error was found to be significant. It may be advantageous to incorporate a partially explicit scheme to the displacement component in order to improve the simulation's performance in this area for future research.

\section{Acknowledgement}
The work was partly done when Xin Su was affiliated with Department of Mathematics in Texas A\&M University. Xin Su would also like to thank the partial support from National Science Foundation (DMS-2208498) when she was affiliated with Department of Mathematics in Texas A\&M University.
\bibliographystyle{unsrt}
\bibliography{ThesisReference.bib}
\end{document}